\documentclass[12pt]{amsart}
\usepackage{tikz, hyperref, amssymb,amsmath,enumerate,mathtools}
\usetikzlibrary{decorations.markings}
\usepackage[mathscr]{eucal}

\mathtoolsset{centercolon}

\newtheorem{thm}{Theorem}[section]
\newtheorem{cor}[thm]{Corollary}
\newtheorem{lem}[thm]{Lemma}
\newtheorem{prp}[thm]{Proposition}
\newtheorem{conj}[thm]{Conjecture}

\theoremstyle{definition}
\newtheorem{dfn}[thm]{Definition}

\newtheorem{ntn}[thm]{Notation}

\theoremstyle{remark}
\newtheorem{rmk}[thm]{Remark}

\newtheorem{example}[thm]{Example}
\newtheorem{examples}[thm]{Examples}

\numberwithin{equation}{section}

\newcommand{\CC}{\mathbb{C}}
\newcommand{\NN}{\mathbb{N}}

\newcommand{\TT}{\mathbb{T}}
\newcommand{\ZZ}{\mathbb{Z}}
\newcommand{\RR}{\mathbb{R}}

\newcommand{\Ff}{\mathcal{F}}
\newcommand{\Gg}{\mathcal{G}}
\newcommand{\Hh}{\mathcal{H}}

\newcommand{\Kk}{\mathcal{K}}
\newcommand{\Ll}{\mathcal{L}}
\newcommand{\Mm}{\mathcal{M}}
\newcommand{\Oo}{\mathcal{O}}
\newcommand{\Pp}{\mathcal{P}}
\newcommand{\Qq}{\mathcal{Q}}
\newcommand{\Tt}{\mathcal{T}}
\newcommand{\Uu}{\mathcal{U}}

\newcommand{\Ad}{\operatorname{Ad}}

\newcommand{\clsp}{\overline{\lsp}}
\newcommand{\coker}{\operatorname{coker}}

\newcommand{\id}{\operatorname{id}}

\newcommand{\lsp}{\operatorname{span}}

\newcommand{\Cliff}[1]{\operatorname{\mbox{$\CC\!\operatorname{liff}$}}_{#1}}

\newcommand{\IdKK}[1]{[\id_{#1}]}

\newcommand{\hatimes}{\mathbin{\widehat{\otimes}}}

\newcommand{\Kgr}{K^{\operatorname{gr}}}

\title{Graded $C^*$-algebras, graded $K$-theory, and twisted $P$-graph $C^*$-algebras}

\author{Alex Kumjian}
\address[A. Kumjian]{Department of Mathematics (084)\\ University
of Nevada\\ Reno NV 89557-0084\\ USA} \email{alex@unr.edu}

\author{David Pask}
\address[D.Pask and A. Sims]{School of Mathematics and Applied Statistics\\
University of Wollongong\\
NSW  2522\\
AUSTRALIA}
\email{dpask, asims@uow.edu.au}
\author{Aidan Sims}

\thanks{This research was supported by the Australian Research Council, grant DP150101598.
Part of this work was carried out while the third-named author was at the Centre de Recerca Mathematica, Universitat
Aut\'onoma de Barcelona as part of the Intensive Research Program \emph{Operator algebras: dynamics and
interactions} in 2017.
The first author  would like to thank his coauthors for their hospitality and support while visiting the
University of Wollongong where much of this work was done.
He would also like to acknowledge support from Simons Foundation Collaboration grant \#353626.}
\date{\today}

\begin{document}

\begin{abstract}
We develop methods for computing graded $K$-theory of $C^*$-algebras as defined in terms
of Kasparov theory. We establish graded versions of Pimsner's six-term sequences for
graded Hilbert bimodules whose left action is injective and by compacts, and a graded
Pimsner--Voiculescu sequence. We introduce the notion of a twisted $P$-graph
$C^*$-algebra and establish connections with graded $C^*$-algebras. Specifically, we show
how a functor from a $P$-graph into the group of order two determines a grading of the
associated $C^*$-algebra. We apply our graded version of Pimsner's exact sequence to
compute the graded $K$-theory of a graph $C^*$-algebra carrying such a grading.
\end{abstract}

\maketitle

\section{Introduction}
This paper has two objectives. The first is to develop techniques for computing graded
$K$-theory of $C^*$-algebras as defined in terms of Kasparov theory, with a view to
expanding on Haag's computation of graded $K$-theory of Cuntz algebras \cite{H1, H2}. The
second is to introduce twisted $P$-graph algebras, generalising \cite{BSV, CKSS, RenWil},
and use them to study connections between $\ZZ_2$-gradings of $C^*$-algebras, and the
twisted $k$-graph $C^*$-algebras studied in \cite{KPS3, KPS4}. The idea is that
$\ZZ_2$-valued functors on $P$-graphs determine gradings of the associated
$C^*$-algebras. The twisted $C^*$-algebras associated to $\{-1,1\}$-valued $2$-cocycles
on cartesian products of $P$-graphs can then be used to model graded tensor products of
graded $C^*$-algebras.

We begin by discussing graded $K$-theory for $C^*$-algebras. Although $K$-theory for
graded Banach algebras has been extensively studied by Karoubi (see, for example,
\cite{Karoubi}), the modern literature on complex graded $C^*$-algebras essentially
begins with the work of Kasparov \cite{K1,K2} on $KK$-theory. Various definitions of
graded $K$-theory for graded $C^*$-algebras have been used in the literature (see
\cite{vD1, vD2, GH, H1, H2, Roe, S} to name but a few). We take as our definition of
$\Kgr_0(A)$ the Kasparov group $KK(\CC, A)$ for the graded $C^*$-algebra $A$, and
likewise define $\Kgr_1(A) := KK(\CC, A \hatimes \Cliff1)$ where $\Cliff1$ is the first
complex Clifford algebra. We establish that perturbing the grading of a $C^*$-algebra $A$
by conjugation by an odd self-adjoint unitary in $\Mm(A)$ does not alter the graded
$K$-theory of $A$. In particular, we show that the graded $K$-theory of the crossed
product of a $C^*$-algebra $A$ by its grading automorphism is identical to the graded
$K$-theory of $A \hatimes \Cliff1$.

To help compute graded $K$-theory in examples, we revisit the work of Pimsner in \cite{P}
to show that his six-term sequences in $KK$-theory for Cuntz--Pimsner algebras are also
valid, with suitable modifications, for graded $C^*$-algebras. Unlike Pimsner, we
restrict to Hilbert bimodules in which the left action is both compact and injective. We
obtain a 6-term exact sequence in graded $K$-theory, which in turn gives a
Pimsner--Voiculescu sequence for graded crossed products by $\ZZ$.

We next develop substantial classes of graded $C^*$-algebras to which we can apply our
theorems. We introduce twisted $P$-graph $C^*$-algebras by straightforward generalisation
of the notion of a twisted $k$-graph algebra. We establish a number of fundamental
structure results for these $C^*$-algebras, including a version of the gauge-invariant
uniqueness theorem, to help us make identifications between these $C^*$-algebras and key
examples later in the paper. We prove that if $P$ has the form $\NN^k \times F$ where $F$
is a countable abelian group, then every $P$-graph is a crossed product of a $k$-graph by
an action of the group $F$, in a sense analogous to that studied in \cite{FPS}.

We next discuss how a functor from a $P$-graph to $\ZZ_2$ induces gradings of the
associated twisted $C^*$-algebras. We show that if $Z_2$ denotes a copy of the order-two
group $\ZZ_2$, regarded as a $\ZZ_2$-graph with one vertex, then $C^*(Z_2)$, under the
grading induced by the degree functor, is isomorphic to $\Cliff1$. More generally, we
consider the situation where $P = \NN^k \times \ZZ_2^l$. Any $P$-graph $\Lambda$ carries
both a natural functor $\delta_\Lambda$ taking values in $\ZZ_2$, and a natural
$\{-1,1\}$-valued $2$-cocycle $c_\Lambda$. We establish a universal description of the
graded twisted $C^*$-algebra determined by this functor and cocycle.

These threads come together when we study graded tensor products in terms of cartesian
products of $P$-graphs. We prove that if $\Lambda$ is a $P$-graph, $\Gamma$ is a
$Q$-graph, and we consider the associated graded, twisted $C^*$-algebras for the functors
and cocycles described in the preceding paragraph, then the graded tensor product is
isomorphic to the graded twisted $C^*$-algebra of the $(P \times Q)$-graph $\Lambda
\times \Gamma$ under its own natural cocycle and functor. Combining this with the results
of previous sections, we show that the higher complex Clifford algebras can be realised
as graded twisted $P$-graph algebras for appropriate $P$, and also that graded tensor
products of graded twisted $P$-graph algebras with $\Cliff1$ can be realised as graded
twisted $(P \times \ZZ_2)$-graph algebras.

We apply our graded Pimsner sequence to the $C^*$-algebras of row-finite 1-graphs $E$
with no sources under gradings of the sort discussed above. The result is an elegant
generalisation of the well-known formula for the $K$-theory of a directed graph: if
$A_\delta$ denotes the $E^0 \times E^0$ matrix with $A_\delta(v,w) = \sum_{e \in v E^1 w}
(-1)^{\delta(e)}$, then the graded $K$-groups of $C^*(E)$ are the cokernel and kernel of
$1 - A_\delta^t$. This recovers Haag's formulas for the graded $K$-theory of Cuntz
algebras \cite{H2}. If $\delta(e) = 1$ for all $e$ (this corresponds to the grading of
$C^*(E)$ coming from the order-two element of the gauge action), then $A_\delta$ is the
negative of the usual adjacency matrix $A_E$, and so $\Kgr_*(C^*(E))$ is given by the
cokernel and kernel of $1 + A_E^t$. We also apply our results to compute the graded
$K$-theory of certain crossed-products of graph algebras by $\ZZ_2$. Our examples and
results lead us to conjecture that the graded $K_0$-group of a $C^*$-algebra can be
described along the lines of the standard picture of ungraded $K_0$, as a group generated
by equivalence classes of graded projective modules.

\smallskip

We begin by collecting relevant background in Section~\ref{sec:background}. In
Section~\ref{sec:graded Kth} we introduce graded $K$-theory in terms of Kasparov theory,
and establish some fundamental results about it. In Section~\ref{sec:graded Pimsner} we
establish graded versions of Pimsner's six-term sequences for Hilbert bimodules with
injective left actions by compacts (see Theorem~\ref{thm:6-term}) and apply them to
obtain a graded Pimsner--Voiculescu sequence for grossed products by $\ZZ$
(Corollary~\ref{cor:gradedPV}). In Section~\ref{sec:P-graphs} we introduce twisted
$P$-graph $C^*$-algebras and establish the basic structure theory for them that we will
need later in the paper. In Section~\ref{sec:elvis}, we discuss gradings of $P$-graph
$C^*$-algebras induced by functors on the underlying $P$-graphs. In
Section~\ref{sec:graded tensor}, we establish our main results about graded tensor
products of graded $P$-graph algebras: Theorem~\ref{thm:swapping} shows that for
appropriate gradings and twisting cocycles, we have $C^*(\Lambda, c_\Lambda) \hatimes
C^*(\Gamma, c_\Gamma) \cong C^*(\Lambda \times \Gamma, c_{\Lambda \times \Gamma})$. In
section~\ref{sec:graph Kth}, we apply our results from Section~\ref{sec:graded Pimsner}
to calculate graded $K$-theory for graph $C^*$-algebras (Lemma~\ref{lem:graph K-th}). We
apply this lemma and our graded Pimsner--Voiculescu sequence to some illustrative
examples. We conclude in Section~\ref{sec:conj} by formulating our conjecture about the
structure of the graded $K_0$ group.

\section{Background}\label{sec:background}

\subsection*{Notation}
We will denote the cyclic group of order $n$ by $\ZZ_n$. We frequently regard $\ZZ_2 =
\{0,1\}$ as a ring, so we always use additive notation for the group operation, and make
any identification of $\ZZ_2$ with $\{-1,1\} \subseteq \TT$ explicit. We typically denote
the multiplication operation in the ring $\ZZ_2$ by $\cdot$.

\subsection*{Graded \texorpdfstring{$C^*$}{C*}-algebras}

Let $A$ be a $C^*$-algebra. A \textit{grading} of $A$ is an automorphism $\alpha$ of $A$
such that $\alpha^2=1$ ($\alpha$ is sometimes referred to as the grading automorphism).
We define
\[
    A_0 := \{a \in A : \alpha(a)= a\}\qquad\text{ and }\qquad
    A_1 := \{a \in A : \alpha(a) = -a\}.
\]
So $A_0, A_1$ are closed, linear self-adjoint subspaces of $A$ and satisfy $A_i A_j
\subset A_{i+j}$ for $i,j \in \ZZ_2$. We have $A = A_0 \oplus A_1$ as a Banach space.
Note that $A_0$ is a $C^*$-subalgebra of $A$. Elements of $A_0$ are called \textit{even}
and elements of $A_1$ are called \textit{odd}. To calculate $A_0$ and $A_1$, it is
helpful to note that
\begin{equation}\label{eq:Ai calc}\textstyle
A_0 = \big\{\frac{a + \alpha(a)}{2} : a \in A\big\}\quad\text{ and }\quad
A_1 = \big\{\frac{a - \alpha(a)}{2} : a \in A\big\}.
\end{equation}

If $a \in A_i$ then we say that $a$ is \emph{homogeneous} of \emph{degree $i$} and write
$\partial a = i$. If $\alpha$ is the identity map on $A$ then $A_0=A$ and $A_1 = \{ 0
\}$. The resulting grading is called the \textit{trivial grading}. Since a $C^*$-algebra
may admit several different gradings (we discuss explicit examples of this in
Examples~\ref{eg:exampleorama} below), we shall frequently write a $C^*$-algebra $A$ with
grading $\alpha$ as the pair $(A, \alpha)$.

A graded $C^*$-algebra is \textit{inner-graded} if there exists a self-adjoint unitary
(called a grading operator) $U \in \mathcal{M} (A) $ such that $\alpha(a) = U a U$ for
all $a \in A$. In \cite[\S14.1]{B} Blackadar calls an inner-grading \emph{even}. A
\emph{graded homomorphism} $\pi : (A,\alpha) \to (B,\beta)$ between graded $C^*$-algebras
is a homomorphism from $A$ to $B$ which intertwines the gradings (i.e.\ $\pi \circ \alpha
= \beta \circ \pi$).

Given a graded $C^*$-algebra $(A, \alpha_A)$ and homogeneous elements $a, b \in A$, the
\emph{graded commutator} $[a,b]^{\operatorname{gr}}$ is defined as
$[a,b]^{\operatorname{gr}} = ab - (-1)^{\partial a \cdot \partial b}ba$. This formula
extends to arbitrary $a$ and $b$ by bilinearity. In particular, if $a \in A_1$, then
\begin{equation}\label{eq:graded commutator}
[a,b]^{\operatorname{gr}} = ab - \alpha_A(b)a.
\end{equation}
If $A$ is trivially graded, then $[a,b]^{\operatorname{gr}}$ reduces to the usual
commutator $[a,b] = ab - ba$.

There is a graded tensor product operation for graded $C^*$-algebras, defined as follows.
Let $(A, \alpha)$, $(B, \beta)$ be graded $C^*$-algebras, and $A \odot B$ be their
algebraic tensor product. This becomes a $^*$-algebra when endowed with multiplication
and involution given by
\[
(a \hatimes b)(a' \hatimes b')
    = (-1)^{\partial b \cdot \partial a'} aa' \hatimes bb'
    \qquad \text{and}\qquad
(a \hatimes b)^* = (-1)^{\partial a \cdot \partial b} a^* \hatimes b^*.
\]
for homogeneous elements $a,a' \in A$ and $b, b' \in B$. We decorate the $\odot$ symbol
with a hat, $\widehat{\odot}$ to indicate that we are using this $^*$-algebra structure
on the algebraic tensor product. We write $A \hatimes B$ for the completion of the
$^*$-algebra $A \widehat{\odot} B$ in the minimal tensor-product norm. (If either $A$ or
$B$ is nuclear then this agrees with the maximal norm.)

The grading automorphism of $A \hatimes B$ is $\alpha \hatimes \beta$. So $a \hatimes b$
is homogeneous of degree $\partial a + \partial b$ if $a$ and $b$ are both homogeneous.
We have
\[
(A \hatimes B)_0 = A_0 \hatimes B_0 + A_1 \hatimes B_1\qquad\text{and}\qquad
(A \hatimes B)_1 = A_0 \hatimes B_1 + A_1 \hatimes B_0.
\]
It is straightforward to show that the graded tensor product operation is associative
(modulo the natural isomorphism $(a \hatimes b) \hatimes c \mapsto a \hatimes (b \hatimes
c)$). We have $A \hatimes B \cong B \hatimes A$ as graded $C^*$-algebras. For this and
other basic facts about graded $C^*$-algebras we refer the reader to \cite[\S14]{B}.

It may aid intuition to observe that for unital graded $C^*$-algebras $A,B$, under this
grading we have $[a \hatimes 1_B, 1_A \hatimes b]^{\operatorname{gr}} = 0 = [1_A \hatimes
b, a \hatimes 1_B]^{\operatorname{gr}}$ for all $a \in A$, and $b \in B$

An example of a graded $C^*$-algebra that we shall use very frequently is $M_{2n}(\CC)$
with grading automorphism $\alpha(\theta_{i,j}) = (-1)^{i-j} \theta_{i,j}$. This is an
inner grading, as it is implemented by the grading operator $U \in M_{2n}(\CC)$ given by
$U_{i,j} = (-1)^i\delta_{ij}$. We often write $\widehat{M}_{2n}(\CC)$ to emphasise that
we are using this grading.

\subsection*{Clifford algebras over \texorpdfstring{$\CC$}{C}}

Following \cite[Examples 14.1.2(b)]{B} the $C^*$-algebra $A = \CC \oplus \CC$ has a
grading automorphism $\alpha$ given by $\alpha(z,w) = (w,z)$. So $( \CC \oplus \CC)_0 =
\{(z,z) : z \in \CC \}$ and $( \CC \oplus \CC )_1 = \{ (z,-z) : z \in \CC \}$. This
graded $C^*$-algebra is called the first (complex)  \emph{Clifford algebra}
\cite[Section~2]{K1}, and we denote it by $\Cliff1$ with this grading $\alpha$ implicit.
As a graded $C^*$-algebra $\Cliff1$ is generated by the odd self-adjoint unitary $u =
(1,-1)$, because, for $(z,w) \in \Cliff1$, we can write
\[
(z,w)
    = \frac{1}{2}(z+w,z+w) + \frac{1}{2}(z-w,w-z)
    = \frac{1}{2} (z+w) 1 + \frac{1}{2}(z-w) u.
\]
Note that $\Cliff1$ is not inner-graded (because it is abelian), and is isomorphic to the
group $C^*$-algebra $C^* ( \ZZ_2 )$ with grading given by the dual action of
$\widehat{\ZZ_2} \cong \ZZ_2$.

The higher complex Clifford algebras are defined inductively: $\Cliff{n+1} = \Cliff{n}
\hatimes \Cliff1$ for $n \ge 1$. It is straightforward to show that $\Cliff2 \cong
\widehat{M}_2(\CC)$ as graded $C^*$-algebras (see also Example~\ref{ex:above}(i)).
Observe that $\Cliff{n}$ is generated by $n$ mutually anticommuting odd self-adjoint
unitaries.

For other basic facts about complex Clifford $C^*$-algebras we refer the reader to
\cite[\S14]{B}.

\subsection*{Graded Hilbert modules}

Let $B$ be a graded $C^*$-algebra.  A graded (right) Hilbert $B$-module is a Hilbert
$B$-module $X$ together with a decomposition of $X$ as a direct sum of two closed
subspaces $X_0$ and $X_1$ compatible with the grading of $B$ in the sense that $X_iB_j
\subset X_{i+j}$ and $\langle X_i, X_j \rangle \subset B_{i+j}$ (the graded components
$X_i$ need not be Hilbert submodules). We define the grading operator $\alpha_X$ on $X$
on homogeneous elements by $\alpha_X(x) = (-1)^j x$ if $x \in X_j$. This $\alpha_X$ is
not necessarily an adjointable operator. Given a graded Hilbert $B$-module $X$ we write
$X^\text{op}$ for the same Hilbert $B$-module with the grading components switched (so
$\alpha_{X^\text{op}} = -\alpha_X$). The grading operator on $X$ induces a grading
$\tilde{\alpha}_X$ on $\Ll(X)$ given by
\begin{equation}\label{eq:tildealpha}
    \tilde{\alpha}_X(T) = \alpha_X \circ T \circ \alpha_X \quad\text{ for all $T \in \Ll(X)$.}
\end{equation}
Under this induced grading, $T$ is homogeneous of degree $j$ if and only if $T X_k
\subseteq X_{j+k}$ for $j,k \in \ZZ_2$. For $\xi,\eta \in X$ we write $\theta_{\xi,\eta}$
for the generalised compact operator $\theta_{\xi,\eta}(\zeta) = \xi \cdot \langle \eta,
\zeta\rangle_B$. The grading $\tilde{\alpha}$ of $\Ll(X)$ restricts to a grading of
$\Kk(X)$ satisfying $\tilde{\alpha}_X(\theta_{\xi,\eta}) = \theta_{\alpha_X(\xi),
\alpha_X(\eta)}$.

Naturally $B$ may be regarded as a graded Hilbert $B$-module $B_B$ with inner product
$\langle a,b\rangle_B = a^*b$, right action given by multiplication, and grading operator
$\alpha_B$. We write $\Hh_B$ for the graded Hilbert $B$-module obtained as the direct sum
of countably many copies of $B_B$. We define $\widehat{\Hh}_B := \Hh_B \oplus
\Hh_B^\text{op}$. If $X$ is a graded Hilbert $B$-module, then $X \oplus \widehat{\Hh}_B
\cong \widehat{\Hh}_B$ by Kasparov's Stabilization Theorem (see \cite[Theorem
14.6.1]{B}).

A graded Hilbert $B$-module which is a finitely generated projective module will be
called a graded projective $B$-module throughout the paper. If $X$ is a graded projective
$B$-module, then $\Kk(X) = \Ll(X)$ and so in particular $1_X \in \Kk(X)$. Kasparov's
Stabilization Theorem implies that every  graded projective $B$-module $X$ is isomorphic
to $p\widehat{\Hh}_B$ for some even projection $p \in \Kk(\widehat{\Hh}_B)$. Moreover
$p\widehat{\Hh}_B$ is a graded projective $B$-module for any such projection. Given even
projections $p, q \in \Kk(\widehat{\Hh}_B)$, we have $p\widehat{\Hh}_B \cong
q\widehat{\Hh}_B$ if and only if there is an even partial isometry $v \in
\Kk(\widehat{\Hh}_B)$ such that $p = v^*v$ and $q = vv^*$.

\subsection*{\texorpdfstring{$C^*$}{C*}-correspondences and Cuntz--Pimsner algebras}

We briefly recap the notion of a $C^*$-correspondence and of the associated
Cuntz--Pimsner algebra and its Toeplitz extension.  For a detailed introduction to
$C^*$-correspondences, see \cite{Lance, tfb}. For more background on Cuntz--Pimsner
algebras, see \cite{P} and \cite[\S8]{Raeburn}.

Given separable $C^*$-algebras $A$ and $B$, an $A$--$B$-correspondence $X$ is a pair
$(\phi, X)$ consisting of a right-Hilbert $B$-module $X$ and a homomorphism $\phi : A \to
\Ll(X)$. We regard $\phi$ as implementing a left action of $A$ on $X$ by adjointable
operators, so we often write $\phi(a)x = a\cdot x$ for $a \in A$ and $x \in X$. We say
that $X$ is \emph{full} if $\clsp\{\langle \xi,\eta\rangle_A : \xi,\eta \in X\} = A$. Any
right-Hilbert $B$-module $X$ may be regarded as  $\CC$--$A$ correspondence and we write
$\ell$ for the canonical left action of $\CC$ by scalar multiplication. We say that $X$
is countably generated if there is a sequence $(x_i)^\infty_{i=1}$ in $X$ such that $X =
\clsp\{x_i \cdot b : i \ge 1, b \in B\}$. Note that $B_B$ is a countably generated
correspondence if $B$ is $\sigma$-unital, so in particular if $B$ is separable.

If $B$ is a $C^*$-algebra then there is an isomorphism of $\Mm(B)$ onto $\Ll(B_B)$ that
carries a multiplier $m$ to the operator of left-multiplication by $m$ on $A$. So any
homomorphism $\phi : A \to \Mm(B)$ determines an $A$--$B$-correspondence structure on
$B_B$. We denote this correspondence by ${_\phi B}$. The isomorphism of $\Mm(B)$ onto
$\Ll(B_B)$ carries $B$ onto $\Kk(B_B)$, so the left action of $A$ on ${_\phi B}$ is by
compacts if and only if $\phi$ takes values in $B$.

If $(\phi, X)$ is an $A$--$B$-correspondence and $(\psi, Y)$ is a
$B$--$C$-correspondence, then the internal tensor product $X \otimes_\psi Y$ is formed as
follows: define $[ \cdot, \cdot ]_C$ on the algebraic tensor product $X \odot Y$ by
sesquilinear extension of the formula $[ x \odot y, x' \odot y']_C := \langle y,
\psi(\langle x, x'\rangle_B)y'\rangle_C$. Let $N = \{\xi \in X \odot Y : [\xi,\xi]_C =
0\}$. Then $X \otimes_\psi Y$ is defined to be the completion of $(X \odot Y)/N$ in the
norm determined by the inner-product $\langle \xi + N, \eta + N\rangle_C = [\xi,
\eta]_C$. For $x \in X$ and $y\in Y$, we write $x \otimes y$ for $(x \odot y) + N \in X
\otimes_\psi Y$. There is then a homomorphism $\tilde\psi : \Ll(X) \to \Ll(X \otimes_\psi
Y)$ given by
\begin{equation}\label{eq:T otimes 1}
    \tilde\psi(T)(x \otimes y) = (Tx) \otimes y \quad\text{ for all $x \in X$, $y \in Y$ and $T \in \Ll(X)$.}
\end{equation}
In particular, $\tilde\psi \circ \phi$ is a homomorphism of $A$ into $\Ll(X \otimes_\psi
Y)$, making $X \otimes_\psi Y$ into an $A$--$C$-correspondence with
\[
a \cdot(x \otimes y) = \tilde\psi(\phi(a))(x \otimes y) = (a \cdot x) \otimes y.
\]
When the actions $\phi,\psi$ are clear from context, we frequently write $X \otimes_B Y$
instead of $X \otimes_\psi Y$. If $\phi : A \to B$ and $\psi : B \to C$ are homomorphisms
and $\psi$ is nondegenerate, then ${_\phi B} \otimes_B {_\psi C} \cong {_{\psi\circ\phi}
C}$ under an isomorphism taking $b \otimes c$ to $\psi(b)c$.

If $X$ is an $A$--$A$ correspondence, then we can form its tensor powers $X^{\otimes n}$
given by $X^{\otimes 0} := A$, $X^{\otimes 1} := X$ and $X^{\otimes (n+1)}:= X \otimes_A
X^{\otimes n}$. The \emph{Fock space} of $X$ is the completion $\Ff_X$ of the algebraic
direct sum $\bigoplus_{n=0}^\infty X^{\otimes n}$ in the norm coming from the inner
product $\langle \oplus_n x_n, \oplus_n y_n\rangle_A = \sum_n \langle x_n, y_n\rangle_A$.
This $\Ff_X$ is a $C^*$-correspondence over $A$ with respect to the pointwise actions.
Observe that $\Ff_X$ is full even if $X$ is not, but that the left action of $A$ on
$\Ff_X$ is not by compacts even if the action of $A$ on $X$ is.

A \emph{representation} of $X$ in a $C^*$-algebra $B$ is a pair $(\psi,\pi)$ where $\psi
: X \to B$ is a linear map, $\pi : A \to B$ is a homomorphism, and we have $\psi(a \cdot
\xi) = \pi(a)\psi(\xi)$, $\psi(\xi \cdot a) = \psi(x)\pi(a)$ and $\pi(\langle
\xi,\eta\rangle_A) = \psi(\xi)\psi(\eta)^*$ for all $\xi,\eta \in X$ and $a \in A$. There
is a universal $C^*$-algebra $\Tt_X$, called the \emph{Toeplitz algebra} of $X$,
generated by a representation $(i_X, i_A)$ of $X$. There is also a representation $(L_0,
L_1)$ of $X$ in $\Ll(\Ff_X)$ such that $L_0(a)\rho = a \cdot \rho$ for $a \in A$ and $\xi
\in \Ff_X$ and such that $L_1(\xi)\rho = \xi \otimes \rho$ for $\rho \in \bigcup_{n \ge
1} X^{\otimes n}$ and $L_1(\xi) a = \xi \cdot a$ for $a \in X^{\otimes 0} = A$. The
universal property of $\Tt_X$ gives a homomorphism $L_1 \times L_0 : \Tt_X \to
\Ll(\Ff_X)$ satisfying $(L_1 \times L_0) \circ i_X = L_1$ and $(L_1 \times L_0) \circ i_A
= L_0$. Pimsner proves \cite[Proposition~3.3]{P} that $L_1 \times L_0$ is injective.

To describe the Cuntz--Pimsner algebra of $X$, we will restrict attention to the
situation where the left action of $A$ on $X$ is injective and by compacts. As discussed
on \cite[Page~202]{P}, given a representation $(\psi,\pi)$ of $X$ in $B$ there is a
homomorphism $\psi^{(1)} : \Kk(X) \to B$ such that $\psi^{(1)}(\theta_{\xi,\eta}) =
\psi(\xi)\psi(\eta)^*$ for all $\xi,\eta \in X$. We say that the representation
$(\psi,\pi)$ is \emph{Cuntz--Pimsner covariant}, or just covariant, if
$\psi^{(1)}(\phi(a)) = \pi(a)$ for all $a \in A$. The \emph{Cuntz--Pimsner algebra}
$\Oo_X$ of $X$ is the universal $C^*$-algebra generated by a covariant representation
$(j_X, j_A)$ of $X$; so it coincides with the quotient of $\Tt_X$ by the ideal generated
by elements of the form $i_A(a) - i^{(1)}_X(\phi(a))$, $a \in A$. Under our hypotheses,
$\Kk(\Ff_X) \subseteq \Tt_X$ and is generated as an ideal by $\{L_0(a) -
L_1^{(1)}(\phi(a)) : a \in A\}$, and so $\Oo_X \cong \Tt_X/\Kk(\Ff_X)$.

If $A$ is a $C^*$-algebra and $\alpha$ is an automorphism of $A$, then there is an
isomorphism of the Cuntz--Pimsner algebra of the $A$--$A$ correspondence $X :={_\alpha
A}$ onto the crossed product $A \times_\alpha \ZZ$ that intertwines $i_A : A \to \Oo_{X}$
with the canonical inclusion $\iota : A \hookrightarrow A \times_\alpha \ZZ$, and carries
$i_{X}(a) \in \Oo_{X}$ to $U \iota(a)$, where $U \in A \times_\alpha \ZZ$ is the unitary
generator of the copy of $\ZZ$. There is a corresponding isomorphism of $\Tt_{X}$ onto
the natural Toeplitz extension of the crossed product (that is, Stacey's endomorphism
crossed-product of $A$ by $\alpha$ \cite{Stacey}).

\subsection*{Elements of Kasparov theory}\label{sec:KKbackground}

We introduce the elements of Kasparov theory needed for our work on graded $K$-theory
later. For more background, see \cite{B}.

Let $A,B$ be $C^*$-algebras, and let $X$ be an $A$--$B$-correspondence. Given gradings
$\alpha_A$ of $A$ and $\alpha_B$ of $B$, a grading operator on $X$ is a map $\alpha_X : X
\to X$ such that $\alpha_X^2 = 1$, $\alpha_X(a \cdot x \cdot b) =
\alpha_A(a)\cdot\alpha_X(a)\cdot\alpha_B(b)$ for all $a, x, b$, and $\alpha_B(\langle x,
y\rangle_B) = \langle\alpha_X(x), \alpha_X(y)\rangle_B$ for all $x,y$. We call the pair
$(X, \alpha_X)$ (or just $X$ if the grading operator is understood from context) a graded
$A$--$B$-correspondence (see, for example, \cite{HaoNg}).

Given graded $C^*$-algebras $A, B, C$, a graded $A$--$B$-correspondence $(X, \alpha_X)$
and a graded $B$--$C$-correspondence $(Y, \alpha_Y)$, there is a well-defined grading
operator $\alpha_X \hatimes\alpha_Y$ on $X \hatimes_B Y$ characterised by $(\alpha_X
\hatimes\alpha_Y)(x \hatimes y) = \alpha_X(x) \hatimes\alpha_Y(y)$; note that if $\phi :
A \to B$ is a graded homomorphism of $C^*$-algebras, then ${_\phi B}$ is a graded Hilbert
module.

If $(A, \alpha_A)$ and $(B, \alpha_B)$ are separable graded $C^*$-algebras, then a
\emph{Kasparov $A$--$B$-module} is a quadruple $(X, \phi, F, \alpha_X)$ where $(\phi, X)$
is a countably generated $A$--$B$-correspondence, $\alpha_X$ is a grading operator on
$X$, and $F \in \Ll(X)$ is odd with respect to the grading $\tilde\alpha_X$ described
at~\eqref{eq:tildealpha} in the sense that $F \circ \alpha_X = - \alpha_X \circ F$ and
satisfies
\begin{gather*}
(F  - F^*)\phi(a) \in \Kk(X), \quad (F^2 - 1)\phi(a) \in \Kk(X),\quad\text{ and }\\
    [F, \phi(a)]^{\operatorname{gr}} \in \Kk(X)\text{ for all $a \in A$.}
\end{gather*}
Observe that since $F$ is odd graded, we have $[\phi(a), F]^{\operatorname{gr}} =
\phi(a)F - F\phi(\alpha_A(A))$ by~\eqref{eq:graded commutator}. We say that $(X, \phi, F,
\alpha_X)$ is a \emph{degenerate} Kasparov module if
\begin{gather*}
(F  - F^*)\phi(a) = 0, \quad (F^2 - 1)\phi(a) = 0,\quad\text{ and }\\
    [F, \phi(a)]^{\operatorname{gr}} = 0 \text{ for all $a \in A$.}
\end{gather*}

We say that graded Kasparov modules $(X, \phi, F, \alpha_X)$ and $(Y, \psi, G, \alpha_Y)$
are unitarily equivalent if there is a unitary $U \in \Ll(X, Y)$ of degree zero (in the
sense that $\alpha_Y U = U \alpha_X$) such that $U F = G U$ and $U \phi(a) = \psi(a)U$
for all $a \in A$.

In what follows, $C([0,1])$ always has the trivial grading. Fix a $C^*$-algebra $B$, and
for each $t \in [0,1]$, define $\epsilon_t : C([0,1]) \hatimes B \to B$ by $\epsilon_t(f
\hatimes b) = f(t)b$. A \emph{homotopy} of Kasparov $(A, \alpha_A)$--$(B,
\alpha_B)$-modules from $(X_0, \phi_0, F_0, \alpha_{X_0})$ to $(X_1, \phi_1, F_1,
\alpha_{X_1})$ is a Kasparov $A$--$(B \hatimes C([0,1]))$-module $(X, \phi, F, \alpha_X)$
such that, for $t \in \{0,1\}$, and with $\tilde{\epsilon_t} : \Ll(X) \to \Ll(X \hatimes
{_BB_B})$ as described in~\eqref{eq:T otimes 1}, the module
\[
(X \hatimes_{\epsilon_t} {_BB_B}, \tilde{\epsilon_t}\circ\phi, \tilde{\epsilon_t}(F), \alpha_X \hatimes\alpha_B)
\]
is unitarily equivalent to $(X_t, \phi_t, F_t, \alpha_{X_t})$. We write $(X_0, \phi_0,
F_0, \alpha_{X_0}) \sim_h (X_1, \phi_1, F_1, \alpha_{X_1})$ and say that these two
Kasparov modules are \emph{homotopy equivalent} if there exists a homotopy from $(X_0,
\phi_0, F_0, \alpha_{X_0})$ to $(X_1, \phi_1, F_1, \alpha_{X_1})$.

We write $KK(A,B)$ for the collection of all homotopy-equivalence classes of Kasparov
$A$--$B$-modules. This $KK(A,B)$ forms an abelian group with addition given by direct
sum:
\[
[X, \phi, F, \alpha_X] + [Y, \psi, G, \alpha_Y] = [X \oplus Y, \phi \oplus \psi, F \oplus G, \alpha_X \oplus \alpha_Y],
\]
and identity element equal to the class of the trivial module $[B_B, 0, 0, \id_B]$; this
class coincides with the class of any degenerate Kasparov $A$--$B$-module. As detailed in
the proof of \cite[Proposition~17.3.3]{B}, the (additive) inverse of a class in $KK(A,B)$
is given by
\begin{equation}\label{eq:KasInverse}
    -[X, \phi, F, \alpha_X] = [X, \phi\circ\alpha_A, -F, -\alpha_X].
\end{equation}

Let $(A, \alpha_A)$ and $(B, \alpha_B)$ be graded $C^*$-algebras, let $(\phi, X)$ be an
$A$--$B$-correspondence, and suppose that $\alpha_X$ is a grading operator on $X$. If
$G,H \in \Ll(X)$ are operators for which $(X, \phi, G, \alpha_X)$ and $(X, \phi, H,
\alpha_X)$ are both Kasparov $A$--$B$-modules, then an \emph{operator homotopy} between
these Kasparov modules is a norm-continuous map $t \mapsto F_t$ from $[0,1]$ to $\Ll(X)$
such that $(X, \phi, F_t, \alpha_X)$ is a Kasparov module for each $t$, $F_0 = G$ and
$F_1 = H$. An operator homotopy is a special case of a homotopy in the following sense:
the space $\overline{X} := C([0,1], X)$ is a graded $A$--$C([0,1], B)$-correspondence
with left action given by $\big(\overline{\phi}(a)(x)\big)(t) = \phi(a)x(t)$ and grading
operator $\big(\alpha_{\overline{X}}(x)\big)(t) = \alpha_X(x(t))$. Moreover, there is an
operator $\overline{F} \in \Ll(C([0,1], X))$ given by $\overline{F}(x)(t) = F_t(x(t))$,
and then $(\overline{X}, \overline{\phi}, \overline{F}, \alpha_{\overline{X}})$ is a
Kasparov $A$--$C([0,1], B)$-module. Identifying $C([0,1], B)$ with $B \hatimes C([0,1])$
in the canonical way, we see that $(\overline{X}, \overline{\phi}, \overline{F},
\alpha_{\overline{X}})$ is a homotopy from $(X, \phi, G, \alpha_X)$ to $(X, \phi, H,
\alpha_X)$.

There is a category whose objects are separable graded $C^*$-algebras and whose morphisms
from $A$ to $B$ are homotopy classes of Kasparov $A$--$B$-modules. The composition in
this category is called the \emph{Kasparov product}, denoted $\hatimes_B : KK(A,B) \times
KK(B,C) \to KK(A,C)$. The identity morphism for the object $(A, \alpha_A)$ is the class
of the Kasparov module $(A_A, \id, 0, \alpha_A)$. Given a Kasparov $A$--$B$-module $(X,
\phi, F, \alpha_X)$ and a Kasparov $B$--$C$-module $(Y, \psi, G, \alpha_Y)$, the Kasparov
product
\[
    [X, \phi, F, \alpha_X] \hatimes_B [Y, \psi, G, \alpha_Y]
\]
has the form
\[
    [X \hatimes_\psi Y, \tilde{\phi}, H, \alpha_X \hatimes\alpha_Y]
\]
for an appropriate choice of operator $H$; the details are formidable in general, but we
will not need them here. For us it will suffice to consider Kasparov products in which
one of the factors has the form $[B_B, \phi, 0, \alpha_B]$ for some graded homomorphism
$\phi : (A, \alpha_A) \to (B, \alpha_B)$ of $C^*$-algebras.

In detail, suppose that $\phi : (A, \alpha_A) \to (B, \alpha_B)$ is a graded homomorphism
of graded $C^*$-algebras. Since $B \cong \Kk(B_B)$ via the map $b \mapsto (a \mapsto
ba)$, the quadruple $(B_B, \phi, 0, \alpha_B)$ is a Kasparov $A$--$B$-module. If $(X,
\psi, F, \alpha_X)$ is a Kasparov $B$--$C$-module, then $(X, \psi \circ \phi, F,
\alpha_X)$ is also a Kasparov module, whose class in $KK(A, C)$ we denote by $\phi^*[X,
\psi, F, \alpha_X]$. Proposition~18.7.2(b) of \cite{B} shows that
\[
[B_B, \phi, 0, \alpha_B] \hatimes_B [X, \psi, F, \alpha_X]
    = \phi^*[X, \psi, F, \alpha_X].
\]
Likewise, if $(Y, \psi, G, \alpha_Y)$ is a Kasparov $C$--$A$-module, then $(Y
\hatimes_\phi B_B, \psi \hatimes1, G \hatimes1, \alpha_Y \hatimes\alpha_B)$ is a Kasparov
$C$--$B$-module whose class we denote by $\phi_*[Y, \psi, G, \alpha_Y]$, and
\cite[Proposition~18.7.2(b)]{B} shows that
\[
[Y, \psi, G, \alpha_Y] \hatimes_A [B_B, \phi, 0, \alpha_B]
    = \phi_*[Y, \psi, G, \alpha_Y].
\]

Observe that if $(A, \alpha)$ is a graded $C^*$-algebra, then the discussion above shows
that $KK(A, A)$ is a ring under the Kasparov product, with multiplicative identity
\[
\IdKK{A} := [A, \id_A, 0, \alpha_A].
\]
By definition of the additive inverse (see~\eqref{eq:KasInverse}), the tuple $(A,
\alpha_A, 0, -\alpha_A)$ is a Kasparov module, and
\begin{equation}\label{eq:special Kas neg}
[A, \alpha_A, 0, -\alpha_A] = -\IdKK{A}
\end{equation}

\section{Graded \texorpdfstring{$K$}{K}-theory of \texorpdfstring{$C^*$}{C*}-algebras}\label{sec:graded Kth}

In this section, we consider graded $K$-theory for $C^*$-algebras. There does not appear
to be a universally-accepted definition of graded $K$-theory for $C^*$-algebras in the
literature to date. We have chosen to take Kasparov's $KK$-theory as the basis for our
definition (see \cite[Definition~17.3.1]{B}). We establish some basic properties of
graded $K$-theory; in particular, that both taking graded tensor products with $\Cliff1$,
and taking crossed products by $\ZZ_2$ with respect to a suitable grading, interchange
graded $K$-groups.

The following definition is used implicitly in \cite{H1, H2}.

\begin{dfn} \label{dfn:kgr}
Let $A$ be a separable $C^*$-algebra and let $\alpha$ be a grading automorphism of $A$.
We define the graded $K$-theory of $A$ as follows: $\Kgr_0(A, \alpha) := KK(\CC, A)$ and
$\Kgr_1(A, \alpha) := KK(\CC, A \hatimes \Cliff1)$. When $\alpha$ is understood from
context we often write $\Kgr_i(A)$ for $\Kgr_i(A, \alpha)$.
\end{dfn}

\begin{rmk}
From the above definition and results from Kasparov theory, we see that $\Kgr_j$ is
covariantly functorial, natural, continuous with respect to direct limits, and invariant
under Morita equivalence. For $j \ge 2$ we define $\Kgr_j (A) = KK(\CC, A \hatimes
\Cliff{j})$. The functors $\Kgr_j$ satisfy Bott periodicity (see Remark~\ref{rmk:later}
below). So we write $\Kgr_*(A) = (\Kgr_0(A), \Kgr_1(A))$ as we do for ungraded
$K$-theory.
\end{rmk}

Up to isomorphism Definition~\ref{dfn:kgr} agrees with the usual definition of $K$-theory
if $A$ is $\sigma$-unital and trivially graded (see \cite[18.5.4]{B}).  Furthermore (see
\cite[14.5.1~and~14.5.2]{B}), if $A$ is inner graded, then $\Kgr_*(A) \cong K_*(A)$.

\begin{rmk} \label{rmk:later}
By definition $\Kgr_0(A \hatimes \Cliff1) = KK(\CC, A \hatimes \Cliff1) = \Kgr_{1}(A)$.
We also have $\Kgr_1(A \hatimes \Cliff1) \cong \Kgr_0(A)$. To see this, recall that by
\cite[Corollary~17.8.8]{B}, we have $KK(\CC, A \hatimes \widehat{M_2}(\CC)) \cong KK(\CC,
A)$. Using this at the last step, we calculate
\[
\Kgr_1 ( A \hatimes \Cliff1 )
    = KK(\CC, A \hatimes \Cliff1 \hatimes \Cliff1)
    \cong KK(\CC, A \hatimes \widehat{M_2}(\CC))
     = \Kgr_0(A)
\]
as claimed.
\end{rmk}

\begin{example}
Since $\CC$ is trivially graded we have $\Kgr_*(\CC) = K_*(\CC) = (\ZZ, 0)$. From
Remark~\ref{rmk:later} we have $\Kgr_i ( A \hatimes \Cliff1 ) = \Kgr_{i+1} (A)$, and it
is easy to show that $ \CC \hatimes \Cliff1 \cong \Cliff1$. Hence putting $A = \CC$ we
have
\[
\Kgr_i ( \Cliff1 ) = \Kgr_{i+1} ( \CC ) = K_{i+1} ( \CC ) = \begin{cases}
0 & \text{ if } i =0 \\
\ZZ & \text{ if } i=1 .
\end{cases}
\]

Since $\Cliff1 \hatimes \Cliff1 = \Cliff2$, the preceding paragraph applied with $A =
\Cliff1$ gives $\Kgr_i ( \Cliff2 ) \cong K_i ( \CC )$. Repeating this procedure we find
that $\Kgr_i(\Cliff{n}) \cong K_{i+n}(\CC)$. So $\Kgr_i ( \Cliff{n} ) \cong \ZZ$ if $i+n$
is even and it is trivial otherwise.
\end{example}

Before moving on to some tools for computing graded $K$-theory, it is helpful to relate
it to our intuition for ordinary $K$-theory. We think of $K_0(A)$ as a group generated by
equivalence classes of projections  in $A \otimes \Kk$ so that, in particular, $[v^*v] =
[vv^*]$ whenever $v$ is a partial isometry. The following example indicates that in
graded $K$-theory similar relations hold for \emph{homogeneous} partial isometries in
graded $C^*$-algebras, but with an additional dependence on the parity of the partial
isometry in question. We discuss this further in Section~\ref{sec:conj}

\begin{example}\label{eg:odd pi}
Let $A$ be a graded $C^*$-algebra with grading automorphism $\alpha$ and suppose that $v$
is an odd partial isometry in $\Kk(\widehat{\Hh}_A)$.
We obtain graded Kasparov modules as follows: let $p = v^*v$ and $q = vv^*$. By the usual
abuse of notation we let $\alpha$ denote the grading operator on $\widehat{\Hh}_A$ and
observe that if $p \in \Kk(\widehat{\Hh}_A)$ is a homogeneous projection then $\alpha$
restricts to a grading operator on the graded  submodule $p\widehat{\Hh}_A \subseteq
\Hh_A$ denoted $\alpha|_{p\widehat{\Hh}_A}$. Recall that $\ell$ denotes the action of
$\CC$ by scalar multiplication on any Hilbert module. We can form the Kasparov modules
$(\ell, p\widehat{\Hh}_A, 0, \alpha|_{p\widehat{\Hh}_A})$ and $(\ell, q\widehat{\Hh}_A,
0, \alpha|_{q\widehat{\Hh}_A})$. We claim that the classes $[p\widehat{\Hh}_A]_K$ and
$[q\widehat{\Hh}_A]_K$ of these Kasparov modules satisfy
$[p\widehat{\Hh}_A]_K=-[q\widehat{\Hh}_A]_K$ in $KK (\CC, A)$. These are Kasparov modules
because $\Kk(p\widehat{\Hh}_A)$ is isomorphic to $p\Kk(\widehat{\Hh}_A)p$ which is unital
with unit $p$, and so all adjointables are compact. In particular the zero operator $F =
0$ trivially has the property that $F^2 - 1$, $F^* - F$ and $[F, a]^{\operatorname{gr}}$
are compact for all $a \in \CC$.

To see that $[p\widehat{\Hh}_A]_K = -[q\widehat{\Hh}_A]_K$, observe that $\Ad v$
implements an isomorphism $p\widehat{\Hh}_A \to q\widehat{\Hh}_A,$. This isomorphism is
odd in the sense that $\alpha|_{q\widehat{\Hh}_A} = \Ad v \circ
({-\alpha|_{p\widehat{\Hh}_A}})$. We claim that
\begin{equation}\label{eq:Kas sum}
    \Big(p\widehat{\Hh}_A \oplus q\widehat{\Hh}_A, \ell, 0, \alpha|_{p\widehat{\Hh}_A} \oplus \alpha|_{q\widehat{\Hh}_A}\Big)
\end{equation}
is operator homotopic to a degenerate Kasparov module. To see this, note that the module
\begin{equation}\label{eq:degenerate}
    \Big(p\widehat{\Hh}_A \oplus q\widehat{\Hh}_A, \ell, \Big(\begin{smallmatrix} 0 & v^* \\ v & 0\end{smallmatrix}\Big),
        \alpha|_{p\widehat{\Hh}_A} \oplus \alpha|_{q\widehat{\Hh}_A}\Big)
\end{equation}
is a degenerate Kasparov module because $F_v := \big(\begin{smallmatrix} 0 & v^* \\
v & 0\end{smallmatrix}\big)^2 = \id$. Since the straight-line path from $0$ to $F_v$
implements an operator homotopy from~\eqref{eq:Kas sum} to~\eqref{eq:degenerate}, we
conclude that~\eqref{eq:Kas sum} represents $0_{KK(\CC, A)}$ as required.
\end{example}

We now discuss how crossed products by $\ZZ_2$ relate to graded $K$-theory. Let $A$ be a
graded $C^*$-algebra with grading automorphism $\alpha$. Let $v$ be an odd self-adjoint
unitary in $\Mm(A)$ and define $\tilde{\alpha} =\alpha \circ \operatorname{Ad} v$. Since
$\alpha$ commutes with $\operatorname{Ad} v$, this makes $A$ bi-graded in the sense that
it admits two commuting gradings by $\ZZ_2$, or equivalently a grading by $\ZZ_2 \times
\ZZ_2$. Let $A_{jk}$ denote the bihomogeneous elements of degree $j$ with respect to
$\alpha$ and of degree $k$ with respect to $\operatorname{Ad} v$; that is, $a \in A_{jk}$
if and only if $\alpha(a) = (-1)^j a$ and $\operatorname{Ad} v (a) = (-1)^k a$. So if $a
\in A_{jk}$, then $\tilde{\alpha} (a) = (-1)^{j+k} a$. Let $\tilde{A}$ denote the
$C^*$-algebra $A$ graded by $\tilde{\alpha}$.

In the following statement and proof, indices in $\ZZ_2$ are denoted $j,k,l$, and $i$ is
reserved for the imaginary number $i = \sqrt{-1}$.

\begin{thm} \label{thm:newgrading}
With notation as above there is an isomorphism of graded $C^*$-algebras $\phi : A
\hatimes \Cliff1 \to \tilde{A} \hatimes \Cliff1$ such that
\[
\phi ( a_{jk} \hatimes u^l ) = a_{jk} (iv)^k \hatimes u^l ,
\]

\noindent where $a_{jk} \in A_{jk}$ for $j, k, l\in \ZZ_2$, and $u$ is the odd generator
of $\Cliff1$.
\end{thm}
\begin{proof}
First, we check that $\phi$ preserves the grading: The element $a_{jk} \hatimes u^l$ is
$(j+l)$-graded in $A \hatimes \Cliff1$ (with $A$ graded by $\alpha$).  The element
$a_{jk} (iv)^k \hatimes u^l$ is homogeneous of degree $(j+k+k)+l = j + l$ in $\tilde{A}
\hatimes \Cliff1$  with $\tilde{A}$ graded by $\tilde{\alpha}$. In $A \hatimes \Cliff1$
we compute $(a_{jk} \hatimes u^l)(a'_{j'k'} \hatimes u^{l'}) = (-1)^{lj'} a_{jk}
a'_{j'k'} \hatimes u^{l+l'}$. Applying $\phi$ to both sides, and computing in $\tilde{A}
\hatimes \Cliff1$ (with $\tilde{A}$ graded by $\tilde{\alpha}$) we obtain
\begin{align*}
\phi ( a_{jk} \hatimes u^l ) \phi ( a'_{j'k'} \hatimes u^{l'} )
    &= ( a_{jk} (iv)^k \hatimes u^l ) ( a'_{j'k'} (iv)^{k'} \hatimes u^{l'} ) \\
    &= (-1)^{l(j'+k'+k')}a_{jk} (iv)^k a'_{j'k'} (iv)^{k'} \hatimes u^{l+l'}.
\end{align*}
Since $v a'_{j'k'} v = (-1)^{k'}a'_{j'k'}$, we have $(iv)^k a'_{j'k'} (iv)^{k'} =
(-1)^{k\cdot k'} a'_{j'k'} (iv)^k(iv)^{k'}$. Since $k,k'$ belong to the ring $\ZZ_2$, we
have $i^k i^{k'} = (-1)^{k\cdot k'} i^{k+k'}$. So the factors of $(-1)^{k\cdot k'}$
cancel, and we obtain
\begin{align*}
\phi ( a_{jk} \hatimes u^l ) \phi ( a'_{j'k'} \hatimes u^{l'} )
    &= (-1)^{lj'} a_{jk} a'_{j'k'} (iv)^{k+k'} \hatimes u^{l+l'}\\
    &= \phi \big( (-1)^{lj'} a_{jk} a'_{j'k'} \hatimes u^{l+l'} \big) \\
    &= \phi \big( ( a_{jk} \hatimes u^l ) ( a'_{j'k'} \hatimes u^{l'} ) \big).
\end{align*}
The result follows because $A$ is spanned by its bihomogeneous elements.
\end{proof}

Recall that $\widehat{M}_2(\CC)$ denotes the algebra of $2 \times 2$ complex matrices
with the standard inner grading (so diagonal elements are even and off diagonal elements
are odd). Similarly, let $\widehat{\Kk}$ denote the $C^*$-algebra of compact operators
with the standard inner grading. We define $\widehat{M}_2(A) := A \hatimes
\widehat{M}_2(\CC) \cong A \hatimes\Cliff2 \cong A \hatimes\Cliff1 \hatimes\Cliff1$.

\begin{cor} \label{cor:2x2iso}
Continuing with the notation of Theorem~\ref{thm:newgrading}, the isomorphism $\phi : A
\hatimes \Cliff1 \to \tilde{A} \hatimes \Cliff1$ induces an isomorphism $\widehat{M}_2(A)
\cong \widehat{M}_2( \tilde{A} )$, which in turn induces an isomorphism $A \hatimes
\widehat{\Kk} \cong \tilde{A} \hatimes \widehat{\Kk}$. In particular, $\Kgr_*(A) \cong
\Kgr_*(\tilde{A})$.
\end{cor}
\begin{proof}
By Theorem ~\ref{thm:newgrading}, we have $A \hatimes \Cliff1 \cong \tilde{A} \hatimes
\Cliff1$. Hence,
\[
\widehat{M}_2(A) \cong  A \hatimes \Cliff1 \hatimes \Cliff1 \cong \tilde{A} \hatimes \Cliff1 \hatimes \Cliff1 \cong \widehat{M}_2( \tilde{A} )
\]
by the associativity of the graded tensor product. The second assertion then follows from
the canonical isomorphism $\widehat{\Kk} \cong \Kk \hatimes\widehat{M_2}(\CC)$. The final
statement follows from the stability of Kasparov theory.
\end{proof}

Let $B$ be a graded $C^*$-algebra with grading automorphism $\beta$. By
\cite[Proposition~14.5.4]{B} we have $B \hatimes \Cliff1 \cong B \rtimes_\beta \ZZ_2$ as
$C^*$-algebras, and the grading on $B \hatimes \Cliff1$ is determined by the automorphism
$\alpha := (\beta \times 1)\circ\hat{\beta}$ where $\hat{\beta}$ is the grading
determined by the dual action on $B \rtimes_\beta \ZZ_2$ under this identification. Now
let $u$ be the canonical self-adjoint unitary generator of $\Cliff1$ and let $v:=1
\hatimes u \in \Mm(B \hatimes \Cliff1)$. Then $v$ is also an odd self-adjoint unitary
(with respect to the grading $\alpha$); moreover, we have $\operatorname{Ad} v = \beta
\times 1$.

\begin{cor}\label{cor:kthy-iso}
With notation as above, if we endow $B \rtimes_\beta \ZZ_2$ with the grading associated
to the dual action, then $\Kgr_i(B \rtimes_\beta \ZZ_2) \cong \Kgr_i(B \hatimes \Cliff1)
= \Kgr_{i+1}(B)$ for $i = 0,1$.
\end{cor}
\begin{proof}
This follows immediately from the final assertion of Corollary~\ref{cor:2x2iso} with $A
:= B \hatimes \Cliff1$ and $\alpha := (\beta \times 1)\circ\hat{\beta}$.
\end{proof}
\begin{rmk}
With notation as in the above corollary, observe that since the canonical embedding $B \to B \rtimes_\beta \ZZ_2$
may be regarded as a graded homomorphism when $B$ is given the trivial grading and $B \rtimes_\beta \ZZ_2$
is given the dual grading, there is a natural homomorphism
\[
K_i(B) \to \Kgr_i(B \rtimes_\beta \ZZ_2) \cong \Kgr_{i+1}(B) \quad\text{for }i = 0,1.
\]
\end{rmk}

\begin{example}
Let $\beta$ be the grading automorphism of $C(\TT)$ such that $\beta(f)(z) =
f(\overline{z})$. Then there is a short exact sequence of graded $C^*$-algebras:
\[
0 \to C_0(\RR) \hatimes \Cliff1 \xrightarrow{\imath} C(\TT) \xrightarrow{\epsilon} \CC \oplus \CC
\]
where $C_0(\RR)$ and $\CC \oplus \CC$ are trivially graded and $\epsilon(f) = f(1) \oplus
f(-1)$ for all $f \in C(\TT)$. Hence by \cite[Theorem~1.1]{S} we have a six-term exact
sequence:
\[
\begin{tikzpicture}[yscale=0.7]
    \node (00) at (0,0) {$\Kgr_1(\CC \oplus \CC)$};
    \node (40) at (4,0) {$\Kgr_1(C(\TT))$};
    \node (80) at (8,0) {$\Kgr_1(C_0(\RR) \hatimes \Cliff1)$};
    \node (82) at (8,2) {$\Kgr_0(\CC \oplus \CC)$};
    \node (42) at (4,2) {$\Kgr_0(C(\TT))$};
    \node (02) at (0,2) {$\Kgr_0(C_0(\RR) \hatimes \Cliff1)$};
    \draw[-stealth] (02)-- node[above] {${\scriptstyle \imath_*}$} (42);
    \draw[-stealth] (42)-- node[above] {${\scriptstyle \epsilon_*}$} (82);
    \draw[-stealth] (82) -- (80);
    \draw[-stealth] (80)-- node[above] {${\scriptstyle \imath_*}$} (40);
    \draw[-stealth] (40)-- node[above] {${\scriptstyle \epsilon_*}$} (00);
    \draw[-stealth] (00) -- (02);
\end{tikzpicture}
\]
Since $\Kgr_0(C_0(\RR) \hatimes \Cliff1) \cong \ZZ$, $\Kgr_0(\CC \oplus \CC) \cong \ZZ^2$
and $\Kgr_1(C_0(\RR) \hatimes \Cliff1)  = \Kgr_1(\CC \oplus \CC) = 0$, we obtain
$\Kgr_*(C(\TT)) = (\ZZ^3, 0)$ (cf.\ \cite[p.~105]{H2}). It follows by Corollary~\ref{cor:kthy-iso} and
Remark~\ref{rmk:later} that $\Kgr_*(C(\TT) \rtimes_\beta \ZZ_2) = (0, \ZZ^3)$. Note that
$C(\TT) \rtimes_\beta \ZZ_2$ is isomorphic to the $C^*$-algebra of the infinite dihedral
group $\ZZ \rtimes \ZZ_2 \cong \ZZ_2 * \ZZ_2$. Under the isomorphism $C(\TT)
\rtimes_\beta \ZZ_2 \cong C^*(\ZZ_2 * \ZZ_2)$, the dual grading $\hat{\beta}$ becomes the
canonical grading on $C^*(\ZZ_2 * \ZZ_2)$ determined by requiring that both self-adjoint
unitary generators be odd. Note that $C^*(\ZZ_2 * \ZZ_2)$ is the universal unital
$C^*$-algebra generated by two projections (see \cite{RaeSin}).
\end{example}
\begin{rmk}
Let $\alpha$ be the grading automorphism of $C_0(\RR)$ given by $\alpha(f)(x) = f(-x)$ for $f \in C_0(\RR)$.
A  computation similar to the above shows that $\Kgr_*(C_0(\RR)) \cong (\ZZ^2, 0)$.

\end{rmk}

\section{Pimsner's exact sequences for graded \texorpdfstring{$C^*$}{C*}-algebras}\label{sec:graded Pimsner}

The main result of this section, Theorem~\ref{thm:6-term}, shows how to compute the
graded $K$-theory of the Cuntz--Pimsner algebra of a graded $C^*$-correspondence over a
nuclear, separable $C^*$-algebra. We also obtain a graded version of the
Pimsner--Voiculescu 6-term exact sequence for crossed products in
Corollary~\ref{cor:gradedPV}.

To prove our main theorem we follow Pimsner's computation of the $KK$-theory of the
Cuntz--Pimsner algebra $\Oo_X$ in \cite[\S4]{P}, keeping track of the gradings.

\subsection*{Set up} Throughout this section we fix a graded, separable, nuclear
$C^*$-algebra $(A, \alpha_A)$, and a graded $A$--$A$-correspondence $(X, \alpha_X)$ in
the sense of Section~\ref{sec:KKbackground}.

We assume that the left action $\varphi: A \to \Ll(X)$ is injective, by compacts (i.e.\
$\varphi (A) \subseteq \Kk (X)$) and essential in the sense that
$\overline{\varphi(A)X}=X$. We also assume that $X$ is full, that is $\clsp\{\langle
\xi,\eta\rangle_A : \xi,\eta \in X\} = A$. It is not clear that all of these hypotheses
are required for our arguments (for example, Pimsner does not require that the left
action should be by compacts or injective in \cite{P}), but they simplify the discussion
and cover the examples that interest us most.

Recall that there is an induced grading $\tilde{\alpha}_X$ of $\Ll(X)$ given by
$\tilde{\alpha}_X(T) = \alpha_X \circ T \circ \alpha_X$.

Let $[X] \in KK(A, A)$ denote the class of the Kasparov module $\big(X, \varphi, 0,
\alpha_X\big)$.

\begin{lem}\label{lem:grading adjointable}
With notation as above, if $\alpha_A$ is trivial, then $\alpha_X \in \Ll(X)$, and it is
an even self-adjoint unitary with respect to $\tilde{\alpha}_X$. Let
\begin{align*}
X_0 &:= \clsp\{x + \alpha_X(x) : x \in A\} \quad\text{and}\\
X_1 &:= \clsp\{x - \alpha_X(x) : x \in X\}.
\end{align*}
Then $X \cong X_0 \oplus X_1$ as $A$--$A$-correspondences, and in $KK(A,A)$, we have $[X]
= [X_0] - [X_1]$.
\end{lem}
\begin{proof}
Since $\alpha_X$ is idempotent and $\alpha_A$ is trivial, for all $\xi , \eta \in X$ we
have
\[
\langle \alpha_X(\xi), \eta\rangle_A
    = \langle \alpha_X(\xi), \alpha^2_X(\eta)\rangle_A
    = \alpha_A(\langle \xi, \alpha_X(\eta)\rangle_A)
    = \langle \xi, \alpha_X(\eta)\rangle_A.
\]
Hence $\alpha_X$ is a self-adjoint unitary in $\Ll(X)$ and since
$\tilde{\alpha}_X(\alpha_X) = \alpha_X \circ \alpha_X \circ \alpha_X = \alpha_X$ it
follows that $\alpha_X$ is even. Since $\alpha_A$ is trivial, for $a,b \in A$, we have
\[
a \cdot (x \pm \alpha_X(x)) \cdot b
    = a\cdot x \cdot b \pm a \cdot \alpha_X(x) \cdot b
    = a\cdot x \cdot b \pm \alpha_X(a\cdot x\cdot b),
\]
so $A \cdot X_i, X_i \cdot A \subseteq X_i$ for $i = 0,1$.

For $\xi \in X_0$ and $\eta \in X_1$, we have
\[
\langle \xi, \eta\rangle_A
    = \langle \alpha_X(\xi), \eta\rangle_A
    = \langle \xi, \alpha_X(\eta)\rangle_A
    = \langle \xi, -\eta\rangle_A
    = -\langle \xi, \eta\rangle_A.
\]
So $X_0 \perp X_1$, giving $X \cong X_0 \oplus X_1$ as right-Hilbert $A$-modules. Since
$\alpha_A$ is trivial, if $\varphi : A \to \Kk(X)$ is the homomorphism defining the left
action, then $\varphi(A)X_j \subseteq X_j$ for $j = 0,1$. So $X \cong X_0 \oplus X_1$ as
$C^*$-correspondences. We write $\varphi_j : A \to \Kk(X_j)$ for the homomorphism $a
\mapsto \varphi(a)|_{X_j}$.

We now have
\[
\big(X, \varphi, 0, \alpha_X\big)
    \cong \Big(X_0 \oplus X_1, \varphi_0 \oplus \varphi_1, 0, \big(\begin{smallmatrix}1 & 0 \\0 & -1\end{smallmatrix}\big)\Big)
\]
as graded Kasparov modules. The class of the right-hand side is the Kasparov sum of
$[X_0, \varphi_0, 0, \id]$ and $[X_1, \varphi_1, 0, -\id]$. Since $\alpha_A = \id_A$ we
have $\varphi_1 = \varphi_1 \circ \alpha_A$, and so $[X_1, \varphi_1, 0, -\id]$ is
precisely the inverse of $[X_1] = [X_1, \varphi_1, 0, \id]$ in $KK(A, A)$ as described
at~\eqref{eq:KasInverse}, and the result follows.
\end{proof}

Whether or not $A$ is trivially graded, for $\xi, \eta \in X$, and $a \in A$, we have
\[
\alpha_X(\xi \cdot a) \hatimes\alpha_X(\eta)
    = \alpha_X(\xi) \hatimes\alpha_X(a \cdot \eta),
\]
and using this we see that there is an isometric idempotent operator $\alpha_X \hatimes
\alpha_X : X \hatimes_A X \to X \hatimes_A X$ characterised by $\xi \hatimes\eta \mapsto
\alpha_X(\xi) \hatimes\alpha_X(\eta)$. So $\alpha_X$ induces isometric operators
$\alpha_X^{\hatimes n}$ on $X^{\hatimes n}$. We regard the $X^{\hatimes n}$ as graded
$A$--$A$-correspondences with respect to these operators. When we want to emphasise this
grading, we write $X^{\hatimes n}$ for the tensor-product module. Under this grading, if
$\xi_1, \dots, \xi_n$ are homogeneous, say $\xi_k \in X_{j_k}$, then $\xi_1 \hatimes_A
\cdots \hatimes_A \xi_n$ is homogeneous with degree $\sum_k j_k$. When convenient we
write $X^{\hatimes 0}$ for $A$.

If $A$ is trivially graded, then Lemma~\ref{lem:grading adjointable} shows that each
$\alpha^{\hatimes n}_X$ is a self-adjoint unitary.

Let $\Ff_X := \bigoplus_{n=0}^\infty X^{\hatimes n}$ be the Fock space of $X$ \cite{P}.
Then $\Ff_X$ is a $C^*$-correspondence over $A$. We write $\varphi^\infty$ for the
homomorphism $A \to \Ll(\Ff_X)$ implementing the diagonal left action.

The operator $\alpha_X^\infty := \bigoplus^\infty_{n=0} \alpha_X^{\hatimes n}$ is a
grading of $\Ff_X$ and the induced grading on $\Ll(\Ff_X)$ restricts to gradings
$\alpha_\Kk$ and $\alpha_\Tt$ of $\Kk(\Ff_X)$ and $\Tt_X$ respectively. These satisfy
\[
\alpha_\Kk(\theta_{\xi, \eta})
    = \theta_{\alpha^{\hatimes n}_X(\xi), \alpha^{\hatimes m}_X(\eta)}
    \quad\text{ and}\quad
\alpha_\Tt(T_\xi)
    = T_{\alpha^{\hatimes n}_X(\xi)}
\]
for $\xi \in X^{\hatimes n}$ and $\eta \in X^{\hatimes m}$. Since these gradings are
compatible with the inclusion $\Kk(\Ff_X) \hookrightarrow \Tt_X$, they induce a grading
$\alpha_\Oo$ on $\Oo_X \cong \Tt_X/\Kk(\Ff_X)$.

If $A$ is trivially graded, then Lemma~\ref{lem:grading adjointable} shows that
$\alpha_\Tt$ and $\alpha_\Kk$ are inner gradings; but $\alpha_\Oo$ need not be, as we
shall see later.

\subsection*{Pimsner's six-term exact sequence in \texorpdfstring{$KK$}{KK}-theory.}
Let $i_A : A \to \Tt_X$ denote the canonical inclusion homomorphism. Then $i_A$
determines a Kasparov class
\begin{equation}\label{eq:iA class}
[i_A] = [\Tt_X, i_A, 0, \alpha_\Tt] \in KK(A, \Tt_X).
\end{equation}

Pimsner constructs a class in $KK(\Tt_X , A)$ as follows: Let $\pi_0 : \Tt_X \to
\Ll(\Ff_X)$ denote the canonical representation determined by $\pi_0(T_\xi)\rho := \xi
\hatimes_A \rho$ for $\xi \in X$ and $\rho \in X^{\hatimes n}$. One checks, using the
universal property of $\Tt_X$, that there is a second representation $\pi_1 : \Tt_X \to
\Ll(\Ff_X)$ such that for $\rho \in X^{\hatimes n} \subseteq \Ff_X$,
\[
\pi_1(a) \rho
    = \begin{cases}
        \pi_0(a)\rho &\text{ if $n \ge 1$} \\
        0 &\text{ if $n = 0$.}
    \end{cases}
\]
Arguing as in \cite[Lemma~4.2]{P}, we see that $\pi_0(T) - \pi_1(T) \in \Kk(\Ff_X)$ for
all $T \in \Tt_X$. The operator $\big(\begin{smallmatrix} 0&1\\1&0\end{smallmatrix}\big)$
is odd-graded with respect to the grading operator $\overline{\alpha}^\infty_X :=
\big(\begin{smallmatrix} \alpha^\infty_X&0\\0&-\alpha^\infty_X\end{smallmatrix}\big)$,
and so for $T \in \Tt_X$, using the formula~\eqref{eq:graded commutator}, we compute the
graded commutator:
\begin{align*}
\Big[\big(\begin{smallmatrix} 0&1\\1&0\end{smallmatrix}\big), \big(\begin{smallmatrix} \pi_0&0\\0&\pi_1 \circ \alpha_\Tt\end{smallmatrix}\big)(T)\Big]^{\operatorname{gr}}
    &= \Big(\begin{smallmatrix} 0&\pi_1 \circ \alpha_\Tt(T)\\ \pi_0(T)&0\end{smallmatrix}\Big)
            - \Big(\begin{smallmatrix} 0&\pi_0(\alpha_\Tt(T))\\\pi_1 \circ \alpha_\Tt(\alpha_\Tt(T))&0\end{smallmatrix}\Big)\\
    &= \Big(\begin{smallmatrix} 0&(\pi_1 - \pi_0)\circ \alpha_\Tt(T)\\ (\pi_0-\pi_1)(T)&0\end{smallmatrix}\Big),
\end{align*}
which is compact since $(\pi_0 - \pi_1)(T)$ acts only on the finitely-generated submodule
$A \subseteq \Ff_X$. Hence we obtain a Kasparov module
\[
M:= \Big(\Ff_X \oplus \Ff_X, \pi_0 \oplus \pi_1 \circ \alpha_\Tt,
    \big(\begin{smallmatrix} 0&1\\1&0\end{smallmatrix}\big),
    \big(\begin{smallmatrix} \alpha^\infty_X&0\\0&-\alpha^\infty_X\end{smallmatrix}\big)\Big).
\]
Since the essential subspace of $\Ff_X \oplus \Ff_X$ for $\pi_0 \oplus \pi_1 \circ
\alpha_T$ is complemented, replacing $\Ff_X \oplus \Ff_X$ with the essential subspace for
$\pi_0 \oplus \pi_1 \circ \alpha_\Tt$, and adjusting the Fredholm operator accordingly
yields a module representing the same class (see \cite[Proposition~18.3.6]{B}). Hence,
writing $P : \Ff_X \to \Ff_X \ominus A = \bigoplus^\infty_{n=1} X^{\hatimes n}$ for the
projection onto the orthogonal complement of the 0\textsuperscript{th} summand, we have
\begin{equation}\label{eq:alt [M]}
[M] = \Big[\Ff_X \oplus (\Ff_X \ominus A), \pi_0 \oplus \pi_1 \circ \alpha_\Tt,
    \big(\begin{smallmatrix} 0&1\\P&0\end{smallmatrix}\big),
    \big(\begin{smallmatrix} \alpha^\infty_X&0\\0&-\alpha^\infty_X\end{smallmatrix}\big)  \Big]
    \in KK(\Tt_X, A).
\end{equation}

In the ungraded case, the classes $[M] \in KK(\Tt_X, A)$ and $[i_A] \in KK ( A , \Tt_X )$
described at \eqref{eq:alt [M]}~and~\eqref{eq:iA class} are denoted $\alpha$ and $\beta$
in \cite[\S 4]{P}.

\begin{thm}[{cf. \cite[Theorem~4.4]{P}}]\label{thm:kk-equiv}
With notation as above, the pair $[M]$ and $[i_A]$ are mutually inverse. That is, $[i_A]
\hatimes_{\Tt_X} [M] = \IdKK{A}$ and $[M] \hatimes_{A} [i_A] = \IdKK{\Tt_X}$. In
particular $A$ and $\Tt_X$ are $KK$-equivalent as graded $C^*$-algebras.
\end{thm}
\begin{proof}
Since the class $[i_A]$ is induced by a homomorphism of $C^*$-algebras, we can compute
the products  $[i_A] \hatimes_{\Tt_X} [M]$ and $[M] \hatimes_A [i_A]$ using
\cite[Proposition~18.7.2]{B}.

As discussed in Pimsner's proof, \cite[Proposition~18.7.2(b)]{B} implies that the product
$[i_A] \hatimes_{\Tt_X} [M]$ is equal to $(i_A)^*[M]$, and hence, using the
representative~\eqref{eq:alt [M]} of $[M]$, we obtain
\[
[i_A] \hatimes_{\Tt_X} [M] = \Big[\Ff_X \oplus (\Ff_X \ominus A),
     (\pi_0 \oplus \pi_1 \circ \alpha_\Tt) \circ i_A,
    \big(\begin{smallmatrix} 0&1\\P&0\end{smallmatrix}\big),
    \big(\begin{smallmatrix} \alpha^\infty_X&0\\0&-\alpha^\infty_X\end{smallmatrix}\big)\Big].
\]
We have
\begin{align}
\Big(\Ff_X \oplus (\Ff_X \ominus A), (\pi_0 \oplus \pi_1 \circ \alpha_\Tt)&{} \circ i_A,
    \big(\begin{smallmatrix} 0&1\\P&0\end{smallmatrix}\big),
    \big(\begin{smallmatrix} \alpha^\infty_X&0\\0&-\alpha^\infty_X\end{smallmatrix}\big)\Big)
    \oplus \big(A, \alpha_A, 0, -\alpha_A\big)\nonumber\\
    &\cong
    \Big(\Ff_X \oplus \Ff_X, (\pi_0 \oplus \pi_0 \circ \alpha_\Tt) \circ i_A,
        \big(\begin{smallmatrix} 0&1\\1&0\end{smallmatrix}\big),
        \big(\begin{smallmatrix} \alpha^\infty_X&0\\0&-\alpha^\infty_X\end{smallmatrix}\big)\Big).\label{eq:last display}
\end{align}
The operator $F =  \big(\begin{smallmatrix} 0&1\\1&0\end{smallmatrix} \big)$ satisfies
$F^2 = 1$, $F = F^*$. For $a \in A$, we have
\begin{align*}
F \big(\pi_0 \oplus \pi_0 \circ \alpha_\Tt\big)(i_A(a))
    &= \Big(\begin{matrix} 0 & 1 \\ 1 & 0\end{matrix}\Big)
        \Big(\begin{matrix} \pi_0(i_A(a)) & 0 \\ 0 & \pi_0(i_A(\alpha_A(a))) \end{matrix}\Big)\\
    &= \Big(\begin{matrix} \pi_0(i_A(\alpha_A(a))) & 0 \\ 0 & \pi_0(i_A(a)) \end{matrix}\Big)
        \Big(\begin{matrix} 0 & 1 \\ 1 & 0\end{matrix}\Big)\\
    &= \big(\pi_0 \oplus \pi_0 \circ \alpha_\Tt\big)(i_A(\alpha_A(a))) F.
\end{align*}
So the graded commutator $\big[F, \big(\pi_0 \oplus \pi_0 \circ
\alpha_\Tt\big)(i_A(a))\big]^{\operatorname{gr}}$ is zero by~\eqref{eq:graded
commutator}. Hence~\eqref{eq:last display} is a degenerate Kasparov module, and therefore
represents the zero class. By~\eqref{eq:special Kas neg}, we have $\big[A, \alpha_A, 0,
-\alpha_A\big] = -\IdKK{A}$, so we have $[i_A] \hatimes_{\Tt_X} [M] - \IdKK{A} =
0_{KK(A,A)}$ giving $[i_A] \hatimes_{\Tt_X} [M] = \IdKK{A}$.

For the reverse composition, \cite[Proposition~18.7.2(a)]{B} shows that $[M] \hatimes_{A}
[i_A]$ is equal to $(i_A)_*[M]$, which is represented by
\[
\Big((\Ff_X \oplus \Ff_X) \hatimes_A \Tt_X, (\pi_0 \oplus \pi_1 \circ \alpha_\Tt) \hatimes 1_{\Tt_X},
    \big(\begin{smallmatrix} 0&1\\1&0\end{smallmatrix}\big),
    \big(\begin{smallmatrix} \alpha^\infty_X \hatimes \alpha_\Tt&0\\
                             0&-\alpha^\infty_X \hatimes \alpha_\Tt\end{smallmatrix}\big)\Big).
\]
We write $\pi_0'$ and $\pi_1'$ for $\pi_0 \hatimes 1_{\Tt_X}$ and $(\pi_1 \circ
\alpha_\Tt) \hatimes 1_{\Tt_X}$. Since $X$ is essential as a left $A$-module, we have $A
\hatimes_A \Tt_X \cong \Tt_X$ as graded $A$--$\Tt_X$-correspondences, so the grading
$\alpha_\Tt$ implements a left action of $\Tt_X$ on $A \hatimes_A \Tt_X$. We regard this
as an action $\tau$ of $\Tt_X$ on $\Ff_X \hatimes_A \Tt_X$ that acts nontrivially only on
the 0\textsuperscript{th} summand. Consider the Kasparov module
\[
\Big((\Ff_X \oplus \Ff_X) \hatimes_A \Tt_X, (0 \oplus \tau),
        \big(\begin{smallmatrix} 0&1\\1&0\end{smallmatrix}\big),
        \big(\begin{smallmatrix} \alpha^\infty_X \hatimes \alpha_\Tt&0\\
                             0&-\alpha^\infty_X \hatimes \alpha_\Tt\end{smallmatrix}\big)\Big).
\]
The essential subspace of the action $0 \oplus \tau$ is equal to the copy of $A
\hatimes_A \Tt_X$ in the graded submodule $(0 \oplus \Ff_X) \hatimes_A \Tt_X$
of $(\Ff_X \oplus \Ff_X) \hatimes_A \Tt_X$. Moreover, the restriction of $0 \oplus \tau$
to this submodule is just $\alpha_\Tt$. Hence
\[
\big[(\Ff_X \oplus \Ff_X) \hatimes_A \Tt_X, (0 \oplus \tau),
        \big(\begin{smallmatrix} 0&1\\1&0\end{smallmatrix}\big),
        \big(\begin{smallmatrix} \alpha^\infty_X \hatimes \alpha_\Tt&0\\
                             0&-\alpha^\infty_X \hatimes \alpha_\Tt\end{smallmatrix}\big)\big]
    = [\Tt_X, \alpha_{\Tt}, 0, -\alpha_\Tt] = - \IdKK{\Tt_X}
\]
by~\eqref{eq:special Kas neg}.

Therefore,
\begin{align}
(i_A)_*[M] - [ 1_{\Tt_X}]
    &= \Big[ (\Ff_X \oplus \Ff_X) \hatimes_A \Tt_X, \pi_0' \oplus \pi_1',
    \big(\begin{smallmatrix} 0&1\\1&0\end{smallmatrix}\big),
    \big(\begin{smallmatrix} \alpha^\infty_X \hatimes \alpha_\Tt&0\\
                             0&-\alpha^\infty_X \hatimes \alpha_\Tt\end{smallmatrix}\big) \Big] \nonumber\\
    &\qquad\qquad + \Big[ (\Ff_X \oplus \Ff_X) \hatimes_A \Tt_X, (0 \oplus \tau),
        \big(\begin{smallmatrix} 0&1\\1&0\end{smallmatrix}\big),
        \big(\begin{smallmatrix} \alpha^\infty_X \hatimes \alpha_\Tt&0\\
                             0&-\alpha^\infty_X \hatimes \alpha_\Tt\end{smallmatrix}\big) \Big] \nonumber\\
    &= \Big[  (\Ff_X \oplus \Ff_X) \hatimes_A \Tt_X, \pi_0' \oplus (\pi_1' + \tau),
        \big(\begin{smallmatrix} 0&1\\1&0\end{smallmatrix}\big),
        \big(\begin{smallmatrix} \alpha^\infty_X \hatimes \alpha_\Tt&0\\
                             0&-\alpha^\infty_X \hatimes \alpha_\Tt\end{smallmatrix}\big) \Big].\label{eq:mustbe0}
\end{align}
We claim that there is a homotopy of graded homomorphisms $\pi'_t : \Tt_X \to \Ll(\Ff_X  \hatimes_A \Tt_X)$
from $\pi_0' \circ \alpha_\Tt$ to $\pi_1' + \tau$ such that for each $t \in  [0,1]$,
\begin{equation}\label{eq:interpolating Kasmod}
\Big((\Ff_X \oplus \Ff_X) \hatimes_A \Tt_X, \pi_0' \oplus \pi_t',
        \big(\begin{smallmatrix} 0&1\\1&0\end{smallmatrix}\big),
        \big(\begin{smallmatrix} \alpha^\infty_X \hatimes \alpha_\Tt&0\\
                             0&-\alpha^\infty_X \hatimes \alpha_\Tt\end{smallmatrix}\big)\Big)
\end{equation}
is a Kasparov module. To see this, we invoke the universal property of $\Tt_X$. For each
$t$, following Pimsner, define a linear map $\psi_t : X \to \Ll(\Ff_X \hatimes_A \Tt_X)$
by
\[
\psi_t(\xi) = \big(\cos(\pi t/2) (\pi'_0(\alpha_\Tt(T_\xi)) - \pi'_1(T_\xi))
    + \sin(\pi t/2) \tau(\xi)\big) + \pi'_1(T_\xi).
\]
Recall that $\varphi^\infty : A \to \Ll(\Ff_X)$ denotes the homomorphism given by the
diagonal left action of $A$. We write $\tilde\varphi^\infty$ for $\varphi^\infty
\hatimes_A 1_{\Tt_X}$.
We aim to prove that $(\tilde{\varphi}^\infty \circ \alpha_A, \psi_t)$ is
a Toeplitz representation of $X$ for each $t \in [0,1]$.
Since each $\psi_t$ is a convex combination of bimodule maps, we
see that
\[
\tilde\varphi^\infty(\alpha_A(a))\psi_t(\xi) = \psi_t(a \cdot \xi)
    \quad\text{ and }\quad
\psi_t(\xi)\tilde\varphi^\infty(\alpha_A(a)) = \psi_t(\xi\cdot a)
\]
for all $a, \xi, t$. Next we check that $\psi_t$ is compatible with the inner product.
Note that for all $\xi,\eta, \zeta \in X$ the operators
$(\pi'_0(\alpha_\Tt(T_\xi)) - \pi'_1(T_\xi))$, $\tau(\eta)$ and $\pi'_1(T_\zeta)$
have mutually orthogonal ranges in $\Ff_X \hatimes_A \Tt_X$ (the same
observation is made in Pimsner's argument, and the only difference between his operators
and ours is post-composition with $\alpha_\Tt$).
Given $\xi,\eta \in X$ and $t \in [0,1]$ we have
\begin{align*}
\psi_t(\xi)^*\psi_t(\eta)
    &=  \big( \big(\cos(\pi t/2)(\pi'_0(\alpha_\Tt(T_\xi)) - \pi'_1(T_\xi))
        + \sin(\pi t/2) \tau(\xi)\big) + \pi'_1(T_\xi)\big)^*  \\
      &  \qquad \big(\cos(\pi t/2)(\pi'_0(\alpha_\Tt(T_\eta)) - \pi'_1(T_\eta))
        + \sin(\pi t/2) \tau(\eta)\big) + \pi'_1(T_\eta)\big)
          \\
    &= \big( \big(\cos( \pi t/2)(\pi'_0(\alpha_\Tt(T_\xi)) - \pi'_1(T_\xi)) + \sin(\pi t/2) \tau(\xi)\big)\big)^*  \\
     &   \qquad \big(\cos(\pi t/2)(\pi'_0(\alpha_\Tt(T_\eta)) - \pi'_1(T_\eta)) + \sin(\pi t/2) \tau(\eta)\big)
       + \pi'_1(T_\xi)^*\pi_1'(T_\eta).
\end{align*}
Write $\tilde{P}$ for the projection onto $(\Ff_X \hatimes_A \Tt_X) \ominus (A \hatimes_A
\Tt_X)$. Since $\pi'_1$ is a homomorphism, we have
\[
\pi'_1(T_\xi)^*\pi_1'(T_\eta)
    = \pi'_1(T^*_\xi T_\eta))
    = \pi'_1(\langle \xi, \eta\rangle_A)
    = \tilde{P}\tilde\varphi^\infty(\alpha_A(\langle \xi, \eta\rangle_A))\tilde{P}.
\]
For $\zeta, \zeta' \in X$ the range of $\tau(\zeta)$ is contained in
$A \hatimes_A \Tt_X \subseteq \Ff_X \hatimes_A \Tt_X$, which is orthogonal
to the range of $((\pi'_0\circ\alpha_\Tt) - \pi'_1)(T_{\zeta'})$.
Also, $((\pi'_0 \circ \alpha_\Tt) - \pi'_1)(T_\zeta) =
\pi_0'(\alpha_\Tt(T_\zeta))( 1 - \tilde P)$.
Using these two points, and resuming our computation of
$\psi_t(\xi)^*\psi_t(\eta)$ from above, we have
\begin{align*}
\big(\cos(\pi t/2)&(\pi'_0(\alpha_\Tt(T_\xi)) - \pi'_1(T_\xi)) + \sin(\pi t/2) \tau(\xi)\big)^*\\
    &\qquad\qquad
       \big(\cos(\pi t/2)(\pi'_0(\alpha_\Tt(T_\eta)) - \pi'_1(T_\eta)) + \sin(\pi t/2) \tau(\eta)\big)\\
    &= \cos(\pi t/2)^2 (\pi'_0(\alpha_\Tt(T_\xi)) (1 - \tilde P) )^*(\pi'_0(\alpha_\Tt(T_\eta)) (1 - \tilde P))
     + \sin(\pi t/2)^2 \tau(\xi)^*\tau(\eta)\\
    &= (1 - \tilde P)  ( \cos(\pi t/2)^2\tilde\varphi^\infty(\alpha_A(\langle \xi, \eta\rangle_A))
        + \sin(\pi t/2)^2\tilde\varphi^\infty(\alpha_A(\langle \xi, \eta\rangle_A)) ) (1 - \tilde P)  \\
    &= (1-\tilde{P}) \tilde\varphi^\infty(\alpha_A(\langle \xi, \eta\rangle_A)) (1-\tilde{P}).
\end{align*}
Since $\tilde{P}$ commutes with the range of $\tilde\varphi^\infty$, we have
$\psi_t(\xi)^*\psi_t(\eta) = \tilde\varphi^\infty(\alpha_A(\langle \xi, \eta\rangle_A))$
and so $(\tilde{\varphi}^\infty \circ \alpha_A, \psi_t)$ is
a Toeplitz representation of $X$ for each $t \in [0,1]$.
Thus the universal property of
$\Tt_X$ ensures that there exists a homomorphism $\pi'_t : \Tt_X \to \Ll(\Ff_X \hatimes_A
\Tt_X)$ such that $\pi'_t(T_\xi) = \psi_t(\xi)$ for $\xi \in X$ and $\pi'_t(a) =
\tilde\varphi^\infty(\alpha_A(a))$ for $a \in A$.

For $t \in[0,1]$ and $\xi \in X$, we have
\[
\pi'_t(T_\xi) - \pi'_1(T_\xi)
    = \big(\cos(\pi t/2) (\pi'_0(\alpha_\Tt(T_\xi)) - \pi'_1(T_\xi)) + \sin(\pi t/2) \tau(\xi)\big).
\]
The kernel of this operator contains $\tilde{P}(\Ff_X \hatimes_A \Tt_X)$, and since $A$
acts compactly on $(1-\tilde{P})(\Ff_X \hatimes_A \Tt_X)$, Cohen factorisation ensures
that $\pi'_t(T_\xi) - \pi'_1(T_\xi) \in \Kk(\Ff_X \hatimes_A \Tt_X)$. So for each $t$,
the homomorphism $\pi'_t$ is a compact perturbation of the homomorphism $\pi'_1$, which
determines a Kasparov module, and therefore~\eqref{eq:interpolating Kasmod} is a Kasparov
module for each $t$ as claimed.

The claim shows that the class~\eqref{eq:mustbe0} is equal to the class of
\[
\Big((\Ff_X \oplus \Ff_X) \hatimes_A \Tt_X, \pi_0' \oplus \pi_0' \circ \alpha_\Tt,
        \big(\begin{smallmatrix} 0&1\\1&0\end{smallmatrix}\big),
        \big(\begin{smallmatrix} \alpha^\infty_X \hatimes \alpha_\Tt&0\\
                             0&-\alpha^\infty_X \hatimes \alpha_\Tt\end{smallmatrix}\big) \Big).
\]
This is a degenerate Kasparov module (just calculate directly that $F^2 = 1$, $F^* = F$
and $\big[F, \big(\pi_0' \oplus \pi_0' \circ \alpha_\Tt)(a)\big]^{\operatorname{gr}} = 0$
for all $a \in \Tt_X$), so it represents the zero class. Hence $(i_A)_*[M] = \IdKK{\Tt_X}$.
\end{proof}

Let $\iota :  \Kk(\Ff_X) \to  \Ll(\Ff_X)$ denote the canonical inclusion. Then $(\Ff_X,
\iota, 0, \alpha_X^\infty)$ is a Kasparov module and we have  $[\Ff_X, \iota, 0,
\alpha_X^\infty] \in KK ( \Kk ( \Ff_X ) , A)$. As Pimsner points out, this is the
$KK$-equivalence given by the equivalence bimodule $\Ff_X$. Let $j: \Kk ( \Ff_X ) \to
\Tt_X$ be the natural inclusion.

\begin{lem}\label{lem:KKformula} (cf. \cite[Lemma~4.7]{P})
With notation as above we have
\[
[j] \hatimes_{\Tt_X} [M] = [ \Ff_X, \iota, 0, \alpha_X^\infty ]  \hatimes_A  (\IdKK{A} - [X])
\]
in $KK ( \Kk ( \Ff_X ) , A )$.
\end{lem}
\begin{proof}
By \cite[Proposition~18.7.2(b)]{B} we have $[j] \hatimes_{\Tt_X} [M] = j^*[M]$. Using the
representation~\eqref{eq:alt [M]} of $[M]$ we therefore obtain
\[
[j] \hatimes_{\Tt_X} [M]
    = \Big[ \Ff_X \oplus (\Ff_X \ominus A), (\pi_0 \oplus \pi_1 \circ \alpha_\Tt)\circ j,
    \big(\begin{smallmatrix} 0&1\\P&0\end{smallmatrix}\big),
    \big(\begin{smallmatrix} \alpha^\infty_X&0\\0&-\alpha^\infty_X\end{smallmatrix}\big) \Big].
\]
Since $\pi_0 \circ j(\Kk(\Ff_X)) \subseteq \Kk(\Ff_X)$ and similarly for $\pi_1$, the
straight-line path from $\big(\begin{smallmatrix} 0&1\\P&0\end{smallmatrix}\big)$ to $0$
gives an operator homotopy, so
\begin{align*}
[j] \hatimes_{\Tt_X} [M]
    &= \Big[ \Ff_X \oplus (\Ff_X \ominus A), (\pi_0 \oplus \pi_1 \circ \alpha_\Tt)\circ j,
    0, \big(\begin{smallmatrix} \alpha^\infty_X&0\\0&-\alpha^\infty_X\end{smallmatrix}\big) \Big]\\
    &= \big[ \Ff_X, \pi_0 \circ j, 0, \alpha_X^\infty \big]
        + \big[ \Ff_X \ominus A, \pi_1 \circ j \circ \alpha_\Kk, 0, -\alpha_X^\infty \big]\\
    &= \big[ \Ff_X, \iota, 0, \alpha_X^\infty \big]
        + \big[ \Ff_X \ominus A, \pi_1 \circ j \circ \alpha_\Kk, 0, -\alpha_X^\infty \big].
\end{align*}
We have $\Ff_X \hatimes_A X \cong \Ff_X \ominus A$ as right-Hilbert modules, and this
isomorphism carries $\pi_0 \hatimes 1_X$ to $\pi_1$, and hence $(\pi_0 \circ j) \hatimes
1_X$ to $\pi_1 \circ j$. So
\[
(\Ff_X \ominus A, \pi_1 \circ j \circ \alpha_\Kk, 0, -\alpha_X^\infty)
     \cong \big(\Ff_X \hatimes_A X, (\pi_0 \circ j \circ \alpha_\Kk) \hatimes 1_X, 0, -\alpha_X^\infty \hatimes\alpha_X\big).
\]
The right-hand side represents $[\Ff_X, \iota \circ \alpha_\Kk, 0, -\alpha_X^\infty]
\hatimes_A [X]$. We have $[\Ff_X, \iota \circ \alpha_\Kk, 0, -\alpha_X^\infty] = -[\Ff_X,
\iota, 0, \alpha_X^\infty]$ by~\eqref{eq:KasInverse}. Thus
\[
[j] \hatimes_{\Tt_X} [M]
    = \big[ \Ff_X, \iota, 0, \alpha_X^\infty \big] - ([\Ff_X, \iota, 0, \alpha_X^\infty] \hatimes_A [X])
    = \big[ \Ff_X, \iota, 0, \alpha_X^\infty \big] \hatimes_A(\IdKK{A} - [X])
\]
as claimed.
\end{proof}

Finally, we obtain two six-term exact sequences as in \cite[Theorem~4.9]{P}. For the
purposes of computing graded $K$-theory, we are most interested in the first of the two
sequences, and in the situation where $B = \CC$; but both could be useful in general. We
write $i : A \to \Oo_X$ for the canonical inclusion.

\begin{thm}[{cf. \cite[Theorem~4.9]{P}}]\label{thm:6-term}
Let $A$ and $B$ be graded separable  $C^*$-algebras such that $A$ is nuclear and let $X$
be a graded correspondence over $A$ such that the left action is injective and by
compacts and $X$ is full. Then with notation as above we have two exact six-term
sequences as follows.
\begin{equation}\label{eq:exact1}
\parbox[c]{0.8\textwidth}{\hfill
\begin{tikzpicture}[yscale=0.8, >=stealth]
    \node (00) at (0,0) {$KK_1(B, \Oo_X)$};
    \node (40) at (4,0) {$KK_1(B, A)$};
    \node (80) at (8,0) {$KK_1(B, A)$};
    \node (82) at (8,2) {$KK_0(B, \Oo_X)$};
    \node (42) at (4,2) {$KK_0(B, A)$};
    \node (02) at (0,2) {$KK_0(B, A)$};
    \draw[->] (02)-- node[above] {${\scriptstyle{\hatimes_A (\IdKK{A} - [X])}}$} (42);
    \draw[->] (42)-- node[above] {${\scriptstyle i_*}$} (82);
    \draw[->] (82)--(80);
    \draw[->] (80)-- node[above] {${\scriptstyle{\hatimes_A (\IdKK{A} - [X])}}$} (40);
    \draw[->] (40)-- node[above] {${\scriptstyle i_*}$} (00);
    \draw[->] (00)--(02);
\end{tikzpicture}\hfill\hfill}
\end{equation}
\begin{equation}\label{eq:exact2}
\parbox[c]{0.8\textwidth}{\hfill
\begin{tikzpicture}[yscale=0.8, >=stealth]
    \node (00) at (0,0) {$KK_1(\Oo_X, B)$};
    \node (40) at (4,0) {$KK_1(A, B)$};
    \node (80) at (8,0) {$KK_1(A, B)$};
    \node (82) at (8,2) {$KK_0(\Oo_X, B)$};
    \node (42) at (4,2) {$KK_0(A, B)$};
    \node (02) at (0,2) {$KK_0(A, B)$};
    \draw[<-] (02)-- node[above] {${\scriptstyle{(\IdKK{A} - [X]) \hatimes_A}}$} (42);
    \draw[<-] (42)-- node[above] {${\scriptstyle i^*}$} (82);
    \draw[<-] (82)--(80);
    \draw[<-] (80)-- node[above] {${\scriptstyle{(\IdKK{A} - [X]) \hatimes_A}}$} (40);
    \draw[<-] (40)-- node[above] {${\scriptstyle i^*}$} (00);
    \draw[<-] (00)--(02);
\end{tikzpicture}\hfill\hfill}
\end{equation}
These sequences are, respectively, contravariantly and covariantly natural in $B$. They
also natural in the other variable in the following sense: if $C$ is a graded
$C^*$-algebra, and $Y_C$ is a full graded correspondence over $C$ whose left action is
injective and by compacts, and if $\theta_A : A \to C$ and $\theta_X : X \to Y$
constitute a morphism of $C^*$-correspondences, then $\theta_A$ and the induced
homomorphism $(\theta_A \times \theta_X) : \Oo_X \to \Oo_Y$ induce morphisms of exact
sequences from~\eqref{eq:exact1} for $(A,X)$ to~\eqref{eq:exact1} for $(C,Y)$ and
from~\eqref{eq:exact2} for $(C,Y)$ to~\eqref{eq:exact2} for $(A,X)$.
\end{thm}
\begin{proof}
We just prove exactness of the first diagram: the second follows from a similar argument.
Since $A$ is nuclear, so is $\Tt_X$ (see, for example, \cite[Theorem~6.3]{RRS}) and so
the quotient map $q :\Tt_X \to \Oo_X$ has a completely positive splitting. Hence
\cite[Theorem~1.1]{S} applied to the graded short exact sequence $0 \to \Kk(\Ff_X)
\stackrel{j}{\longrightarrow} \Tt_X \stackrel{q}{\longrightarrow} \Oo_X \to 0$ yields
homomorphisms $\delta : KK_i(B, \Oo_X) \to KK_{i+1}(B, \Kk(\Ff_X))$ for which the
following six-term sequence is exact.
\[
\begin{tikzpicture}[yscale=0.8]
    \node (00) at (0,0) {$KK_1(B, \Oo_X)$};
    \node (40) at (4,0) {$KK_1(B, \Tt_X)$};
    \node (80) at (8,0) {$KK_1(B, \Kk(\Ff_X) )$.};
    \node (82) at (8,2) {$KK_0(B, \Oo_X)$};
    \node (42) at (4,2) {$KK_0(B, \Tt_X)$};
    \node (02) at (0,2) {$KK_0(B, \Kk(\Ff_X))$};
    \draw[-stealth] (02)-- node[above] {${\scriptstyle j_*}$} (42);
    \draw[-stealth] (42)-- node[above] {${\scriptstyle q_*}$} (82);
    \draw[-stealth] (82)-- node[right] {${\scriptstyle \delta}$} (80);
    \draw[-stealth] (80)-- node[above] {${\scriptstyle j_*}$} (40);
    \draw[-stealth] (40)-- node[above] {${\scriptstyle q_*}$} (00);
    \draw[-stealth] (00)-- node[left] {${\scriptstyle \delta}$}(02);
\end{tikzpicture}
\]
Define maps $\delta' : KK_*(B, \Oo_X) \to KK_{*+1}(B, A)$ by $\delta' = (\cdot \hatimes
[\Ff_X, \iota, 0, \alpha_X^\infty]) \circ \delta$, and consider the following diagram.
\[
\begin{tikzpicture}[yscale=0.8]
    \node (00) at (0,0) {$KK_1(B, \Oo_X)$};
    \node (40) at (4,0) {$KK_1(B, \Tt_X)$};
    \node (80) at (8,0) {$KK_1(B, \Kk(\Ff_X) )$};
    \node (82) at (8,2) {$KK_0(B, \Oo_X)$};
    \node (42) at (4,2) {$KK_0(B, \Tt_X)$};
    \node (02) at (0,2) {$KK_0(B, \Kk(\Ff_X))$};
    \draw[-stealth] (02)-- node[above] {${\scriptstyle j_*}$} (42);
    \draw[-stealth] (42)-- node[above] {${\scriptstyle q_*}$} (82);
    \draw[-stealth] (82)-- node[right] {${\scriptstyle \delta}$} (80);
    \draw[-stealth] (80)-- node[below] {${\scriptstyle j_*}$} (40);
    \draw[-stealth] (40)-- node[below] {${\scriptstyle q_*}$} (00);
    \draw[-stealth] (00)-- node[left] {${\scriptstyle \delta}$} (02);
    \node (00') at (-2,-2) {$KK_1(B, \Oo_X)$};
    \node (40') at (4,-2) {$KK_1(B, A)$};
    \node (80') at (10,-2) {$KK_1(B, A)$};
    \node (02') at (-2,4) {$KK_0(B, A)$};
    \node (42') at (4,4) {$KK_0(B, A)$};
    \node (82') at (10,4) {$KK_0(B, \Oo_X)$};
    \draw[-stealth] (02')-- node[above] {${\scriptstyle \hatimes (\IdKK{A} - [X])}$} (42');
    \draw[-stealth] (42')-- node[above] {${\scriptstyle i_*}$} (82');
    \draw[-stealth] (82')-- node[right] {${\scriptstyle \delta'}$} (80');
    \draw[-stealth] (80')-- node[below] {${\scriptstyle \hatimes (\IdKK{A} - [X])}$} (40');
    \draw[-stealth] (40')-- node[below] {${\scriptstyle i_*}$} (00');
    \draw[-stealth] (00')-- node[left] {${\scriptstyle \delta'}$} (02');
    \draw[-stealth] (02)--(02') node[pos=0.25, anchor=south west, inner sep=0pt] {$\scriptstyle \hatimes [\Ff_X, \iota, 0, \alpha_X^\infty]$};
    \draw[-stealth, out=75, in=285] (42) to node[pos=0.5, right] {${\scriptstyle[M]}$} (42');
    \draw[-stealth, out=255, in=105] (42') to node[pos=0.5, left] {${\scriptstyle (i_A)_*}$} (42);
    \draw[-stealth] (82') to node[anchor=south east, inner sep=1pt] {$\scriptstyle\id$} (82);
    \draw[-stealth] (80)--(80') node[pos=0.25,  anchor=north east, inner sep=0pt] {$\scriptstyle \hatimes [\Ff_X, \iota, 0, \alpha_X^\infty]$};
    \draw[-stealth, out=255, in=105] (40) to node[pos=0.5, left] {${\scriptstyle[M]}$} (40');
    \draw[-stealth, out=75, in=285] (40') to node[pos=0.5, right] {${\scriptstyle (i_A)_*}$} (40);
    \draw[-stealth] (00') to node[anchor=north west, inner sep=1pt] {$\scriptstyle\id$} (00);
\end{tikzpicture}
\]
The left-hand and right-hand squares commute by definition of the maps $\delta'$.
Lemma~\ref{lem:KKformula} implies that the top left and bottom right squares commute.
Since $q \circ i_A = i$ as homomorphisms, we have $q_* \circ (i_A)_* = i_*$, and so the
top right and bottom left squares commute as well. Since all the maps linking the inner
rectangle to the outer rectangle are isomorphisms, it follows that the outer rectangle is
exact as required.

Naturality follows from naturality of Pimsner's exact sequences, which in turn follows
from naturality of the $KK$-functor for graded $C^*$-algebras \cite[Section~17.8]{B}.
\end{proof}

\begin{cor} \label{cor:6-termgraded}
Let $(A, \alpha)$ be a separable, nuclear graded $C^*$-algebra and let $(X,\alpha_X)$ be
a countably generated, graded correspondence over $A$ such that the left action is
injective and by compacts and $X$ is full. Then there is a 6-term exact sequence for
graded $K$-theory as follows.
\begin{equation} \label{eq:6-termgraded}
\parbox[c]{10.6cm}{\begin{tikzpicture}[yscale=0.8, >=stealth]
    \node (00) at (0,0) {$\Kgr_1( \Oo_X , \alpha_{\Oo} )$};
    \node (40) at (4,0) {$\Kgr_1(A,\alpha)$};
    \node (80) at (8,0) {$\Kgr_1(A,\alpha)$};
    \node (82) at (8,2) {$\Kgr_0(\Oo_X,\alpha_{\Oo})$};
    \node (42) at (4,2) {$\Kgr_0(A,\alpha)$};
    \node (02) at (0,2) {$\Kgr_0(A,\alpha)$};
    \draw[->] (02)-- node[above] {${\scriptstyle{\hatimes_A (\IdKK{A} - [X])}}$} (42);
    \draw[->] (42)-- node[above] {${\scriptstyle i_*}$} (82);
    \draw[->] (82)--(80);
    \draw[->] (80)-- node[above] {${\scriptstyle{\hatimes_A (\IdKK{A} - [X])}}$} (40);
    \draw[->] (40)-- node[above] {${\scriptstyle i_*}$} (00);
    \draw[->] (00)--(02);
\end{tikzpicture}}
\end{equation}
\end{cor}

\subsection*{The Pimsner--Voiculescu exact sequence for graded \texorpdfstring{$C^*$}{C*}-algebras}

If $(A, \alpha)$ is a graded $C^*$-algebra, and $\gamma$ is an automorphism of $A$ that
is graded in the sense that it commutes with $\alpha$, then functoriality of $KK$ shows
that $\gamma$ induces a map $\gamma_0$ on $\Kgr_0(A,\alpha) = KK(\CC, A)$ and $\gamma_1$
on $\Kgr_1(A, \alpha) = KK(\CC , A \hatimes \Cliff1)$.

The crossed product $A \rtimes_\gamma \ZZ$ has two natural grading automorphisms, which
we will denote by $\beta^0$ and $\beta^1$. To describe them, write $i_A : A \to A
\rtimes_\gamma \ZZ$ and $i_\ZZ : \ZZ \to \mathcal{U}\Mm(A \rtimes_\gamma \ZZ)$ for the
canonical inclusions. Then for $k \in \ZZ_2$, the automorphism $\beta^k$ is characterised
by
\begin{equation}\label{eq:beta-k}
\beta^k (i_A(a)i_\ZZ(n)) = (-1)^{kn} i_A(\alpha(a)) i_\ZZ(n).
\end{equation}
So $\beta^1 = \beta^0 \circ \hat{\gamma}_{-1}$ where $\hat{\gamma}$ is the action of
$\TT$ on the crossed product dual to $\gamma$. The inclusion $i_A : A \to A
\rtimes_\gamma \ZZ$ is a graded homomorphism with respect to $\alpha$ and $\beta^k$ for
each of $k = 0, 1$.

\begin{cor}[Graded Pimsner--Voiculescu sequence] \label{cor:gradedPV}
Let $(A, \alpha)$ be a separable, nuclear graded $C^*$-algebra and $\gamma$ an
automorphism of $A$ that commutes with $\alpha$. Fix $k \in \{0,1\}$ and let $\beta^k$ be
the grading automorphism of $A \rtimes_\gamma \ZZ$ described above. Then we obtain a
six-term exact sequence in graded $K$-theory as follows.
\begin{equation}\label{eq:gradedPV}
\parbox[c]{0.8\textwidth}{\hfill
\begin{tikzpicture}[yscale=0.8, >=stealth]
    \node (00) at (0,0) {$\Kgr_1(A \rtimes_\gamma \ZZ, \beta^k)$};
    \node (40) at (4,0) {$\Kgr_1(A, \alpha)$};
    \node (80) at (8,0) {$\Kgr_1(A, \alpha)$};
    \node (82) at (8,2) {$\Kgr_0(A \rtimes_\gamma \ZZ, \beta^k)$};
    \node (42) at (4,2) {$\Kgr_0(A, \alpha)$};
    \node (02) at (0,2) {$\Kgr_0(A, \alpha)$};
    \draw[->] (02)-- node[above] {${\scriptstyle{\id - (-\alpha_*)^k\gamma_*}}$} (42);
    \draw[->] (42)-- node[above] {${\scriptstyle i_*}$} (82);
    \draw[->] (82)--(80);
    \draw[->] (80)-- node[above] {${\scriptstyle{\id - (-\alpha_*)^k\gamma_*}}$} (40);
    \draw[->] (40)-- node[above] {${\scriptstyle i_*}$} (00);
    \draw[->] (00)--(02);
\end{tikzpicture}\hfill\hfill}\end{equation}
The sequence is natural in the sense that if $(B, \kappa)$ is another graded
$C^*$-algebra, $\theta$ is an automorphism of $B$ that commutes with $\kappa$, and $\phi
: A \to B$ is a graded homomorphism that intertwines $\gamma$ and $\theta$, then $\phi$
and $\phi \times 1 : A \rtimes_\gamma \ZZ \to B \rtimes_\delta \ZZ$ induce a morphism of
exact sequences from~\eqref{eq:gradedPV} for $(A,\alpha)$ to~\eqref{eq:gradedPV} for $(B,
\kappa)$.
\end{cor}
\begin{proof}
Let $X := {}_{\gamma} A$ as a Hilbert module, endowed with the grading $(-1)^k \alpha$.
Write $(i_A, i_X)$ for the inclusions of $A$ and $X$ in $\Oo_X$, write $j_A : A \to A
\rtimes_\gamma \ZZ$ for the canonical inclusion, and write $U$ for the unitary element of
$\Mm(A \rtimes_\gamma \ZZ)$ implementing $\gamma$. Then, as pointed out in
\cite[Example~3, p.193]{P}, there is an isomorphism $\rho : \Oo_X \to A \rtimes_\gamma
\ZZ$ such that $\rho(i_A(a)) = j_A(a)$ and $\rho(i_X(a)) = U j_A(a)$ for all $a \in A$.
It is routine to check that this isomorphism is graded with respect to the grading
$\beta^k$ of $A \rtimes_\gamma \ZZ$ and the grading $\alpha_\Oo$ of $\Oo_X$ induced by
the grading $\alpha$ of $A$ and the grading $(-1)^k\alpha$ of $X$.

Hence Corollary~\ref{cor:6-termgraded} gives an exact sequence
\begin{equation}
\parbox[c]{0.9\textwidth}{\mbox{}\hfill\begin{tikzpicture}[yscale=0.8, >=stealth]
    \node (00) at (0,0) {$\Kgr_1(A \rtimes_\gamma \ZZ , \beta^k)$};
    \node (40) at (4,0) {$\Kgr_1(A,\alpha)$};
    \node (80) at (8,0) {$\Kgr_1(A,\alpha)$};
    \node (82) at (8,2) {$\Kgr_0(A \rtimes_\gamma \ZZ , \beta^k)$};
    \node (42) at (4,2) {$\Kgr_0(A,\alpha)$};
    \node (02) at (0,2) {$\Kgr_0(A,\alpha)$};
    \draw[->] (02)-- node[above] {${\scriptstyle{\hatimes_A (\IdKK{A} - [X])}}$} (42);
    \draw[->] (42)-- node[above] {${\scriptstyle i_*}$} (82);
    \draw[->] (82)--(80);
    \draw[->] (80)-- node[above] {${\scriptstyle{\hatimes_A (\IdKK{A} - [X])}}$} (40);
    \draw[->] (40)-- node[above] {${\scriptstyle i_*}$} (00);
    \draw[->] (00)--(02);
\end{tikzpicture}\hfill\mbox{}}
\end{equation}

By definition, we have $\alpha^k_*[X] = [A_A, \gamma \circ \alpha^k, 0, (-1)^k\alpha]$.
So~\eqref{eq:special Kas neg} shows that $\alpha^k_*[X] = (-1)^k \gamma_*$. Since
$\alpha$---and hence $\alpha_*$---has order 2, we deduce that $[X] = (-\alpha_*)^k
\gamma_*$, giving the desired six-term exact sequence.

Naturality follows immediately from naturality in Theorem~\ref{thm:6-term}.
\end{proof}

\section{Twisted \texorpdfstring{$P$}{P}-graph \texorpdfstring{$C^*$}{C*}-algebras and actions by countable
groups}\label{sec:P-graphs}

We now begin our investigation of how to use $P$-graphs to construct examples of graded
$C^*$-algebras. Let $F$ be a countable (discrete) abelian group, fix $k \ge 0$ and let $P
:= \NN^k \times F$ regarded as a cancellative abelian monoid and let $G_P$ denote the
Grothendieck group of $P$. Following \cite[Definition 2.1]{CKSS}, a $P$-graph consists of
a countable small category $\Lambda$ equipped with a functor $d : \Lambda \to P$
satisfying the factorisation property: if $d(\lambda)=p+q$ then there exist unique
$\mu,\nu \in \Lambda$ with $d(\mu)=p$, $d(\nu)=q$ and $\lambda = \mu \nu$. We say that
$\Lambda$ is \emph{row-finite} if
\[
v \Lambda^p :=  \{\lambda \in \Lambda : r(\lambda) = v, d(\lambda) = p\}
    \text{ is finite for $v \in \Lambda^0$ and $p \in P$,}
\]
and that it has \emph{no sources} if $v\Lambda^p$ is  nonempty for every $v \in
\Lambda^0$ and $p \in P$.

There is a natural pre-order on $P$ given by $p \le q$ if $q=p+u$ for some $u \in P$.
Note that $\le$ need not be a partial order: if $F$ is nontrivial, then $\le$ is not
antisymmetric.

\begin{examples} \label{ex:pex}
\begin{enumerate}[(i)]
\item Let $\Omega_P = \{ (p,q) \in P \times P : p \le q \}$. Regarding $P$ as a
    subsemigroup of $\ZZ^k \times F$, we can define $d : \Omega_P \to P$ by the
    expression $d(p,q) = q - p$. Identify $\Omega^0_P := d^{-1}(0)$ with $P$ via
    $(p,p) \mapsto p$, and define $r,s : \Omega_P \to \Omega_P^0$ by $r(p,q) = p$ and
    $s(p,q) = q$. Finally define composition by $(p,q)(q,n) = (p,n)$. Then $\Omega_P$
    is a $P$-graph.
\item When $P$ is regarded as a category with one object it becomes a $P$-graph with
    degree map the identity map, composition the group operation, and range and
    source both given by the trivial map $p \mapsto 0$.
\item Every $k$-graph is a $\NN^k$-graph.
\item In particular, when $P = \NN^k$ in Example~(ii), we obtain the $k$-graph with
    one vertex and one path of each degree in $\NN^k$. As is standard \cite{KPS3,
    KPS4}, we denote this $k$-graph by $T_k$; its $C^*$-algebra is isomorphic to
    $C(\TT^k)$.
\item Another example we shall use frequently is the $1$-graph $B_n$ with one vertex
    and $n$ distinct edges, whose $C^*$-algebra is the Cuntz algebra $\Oo_n$. The
    ``B" here stands for ``bouquet" and we sometimes refer to $B_n$ as the ``bouquet
    of $n$ loops."
\end{enumerate}
\end{examples}

The definition of the categorical cohomology of a $k$-graph given in \cite[\S3]{KPS4}
applies to $P$-graphs and we use the formalism and notation from there. In detail, let
$A$ be an abelian group and $\Lambda^{*r}$ the collection of composable $r$-tuples of
elements of $\Lambda$.
Then $Z^2(\Lambda, A)$, the group of normalised $2$-cocycles on $\Lambda$,
consists of all functions  $f:  \Lambda^{*2} \to A$ such that
\[
    f(\lambda,\mu) + f(\lambda\mu, \nu)
    = f(\mu, \nu) + f(\lambda, \mu\nu)
\]
for all $(\lambda,\mu,\nu) \in \Lambda^{*3}$  and  $f(r(\lambda),\lambda) = 0 = f(\lambda, s(\lambda))$
for all $\lambda \in \Lambda$ (cf.\ \cite[Lemma 3.8]{KPS4}).
Furthermore $f_1 , f_2 \in Z^2 ( \Lambda , A )$ are \textit{cohomologous} if they differ by a
coboundary: that is there is a map  $b : \Lambda \to A$ such that $(f_1 - f_2)(\lambda ,
\mu) = (\delta^1 b)(\lambda, \mu)  = b(\lambda) - b(\lambda\mu) + b(\mu)$ for all
$(\lambda, \mu) \in \Lambda^{*2}$.
As usual, when $A = \TT$ the group operation is written multiplicatively.

The following example of a $2$-cocycle on $\ZZ_2^l$ will prove important later.

\begin{example}\label{eg:kappa twists}
Let $A = \ZZ^l_2$ for some $l \ge 1$. Let $Z^2(A, \ZZ_2)$ be the group of $\ZZ_2$-valued
group $2$-cocycles on $A$. There is a natural map from $Z^2(A, \ZZ_2)$ to $Z^2(A, \TT)$
given as follows: for any $\kappa \in Z^2(A, \ZZ_2)$ define $c_\kappa \in Z^2(A, \TT)$ by
\[
    c_\kappa(m,n)= (-1)^{\kappa(m,n)} \text{ for } (m,n) \in A \times A.
\]

For example, consider $\kappa : A \times A \to \ZZ_2$ given by
\begin{equation} \label{eq:cdef}
    \kappa(m,n) = \sum_{1 \le j < i \le l} m_i \cdot n_j , \text{ where } m_i ,n_j \in \ZZ_2 .
\end{equation}

Then $\kappa \in Z^2(A, \ZZ_2)$. Indeed, $\kappa$ is biadditive and on pairs $(e_i,
e_j)$ of generators of $\ZZ_2^l$, it satisfies $\kappa(e_i, e_j) = 1 \in \ZZ_2$ if $j <
i$ and $\kappa(e_i, e_j) = 0 \in \ZZ_2$ if $i \le j$.

Let $\sigma$ be a permutation of $\{1, \dots, l\}$. Define $(m^\sigma)_i= m_{\sigma (i)}$
for $m \in A$. Then $m \to m^\sigma$ is an automorphism of $A$, and then for $\kappa \in
Z^2(A, \ZZ_2)$ we may form the $2$-cocycle $\kappa^\sigma \in Z^2(A, \ZZ_2)$, by
$\kappa^\sigma(m,n) = \kappa(m^\sigma, n^\sigma)$ for $(m,n) \in A \times A$. We then
have
\[
    c_{\kappa^\sigma}(m,n)  
    = \prod_{j < i} (-1)^{m_{\sigma(i)} n_{\sigma(j)}}.
\]
\end{example}

If $b : A \to \TT$ is a function, then $\delta^1 b$ is the associated $2$-coboundary
given by $\delta^1 b(m, n) = b(m) b(m+n)^{-1} b(n)$.

\begin{lem} \label{lem:permute}
With notation as above $c_\kappa$ is cohomologous to $c_{\kappa^\sigma}$ in $Z^2 ( A ,
\TT )$. Hence there is a map $b : A \to \TT$ such that
$c_{\kappa^\sigma} = c_\kappa\delta^1 b$.
\end{lem}
\begin{proof}
Let $\chi_{c_{\kappa^\sigma}} : A \times A \to \{1,-1\} \subseteq \TT$ be the bicharacter
of $A$ defined by
\[
\chi_{c_{\kappa^\sigma}} (m,n) = c_{\kappa^\sigma}(m,n) c_{\kappa^\sigma}(n,m)^{-1} =  c_{\kappa^\sigma}(m,n) c_{\kappa^\sigma}(n,m)
\]
for all $m,n \in A$. For generators $e_i , e_j \in A$ we have
\[
\chi_{c_{\kappa^\sigma}} ( e_i , e_j ) = \begin{cases} 1 & \text{ if } i = j \\
-1 & \text{ otherwise.}
\end{cases}
\]
So $\chi_{c_{\kappa^\sigma}}$ does not depend on $\sigma$, and in particular
$\chi_{c_{\kappa^\sigma}} = \chi_{c_{\kappa^{\id}}} = \chi_{c_\kappa}$. Thus
\cite[Proposition~3.2]{OPT} implies that $c_\kappa$ and $c_{\kappa^\sigma}$ are
cohomologous (see also \cite[Lemmata 7.1 and 7.2]{Kleppner}). The final statement follows
by the definition of a coboundary (see \cite[Definition~3.2]{KPS4}).
\end{proof}

For $P := \NN^k \times \ZZ_2^l$, there is a natural surjection $\rho : P \to \ZZ_2^{k+l}$
given by taking the residue of each coordinate modulo $2$.

\begin{prp} \label{prp:clambdadef}
Let $\Lambda$ be a $P$-graph where  $P = \NN^k \times \ZZ_2^l$. With $\kappa$ defined as
in~\eqref{eq:cdef}, the formula $c_\Lambda(\lambda,\mu) = c_\kappa\big(\rho(d(\lambda))),
\rho(d(\mu))\big)$ defines a $2$-cocycle  $c_\Lambda \in Z^2 (\Lambda, \TT)$ with values
in $\{\pm 1\}$.
\end{prp}
\begin{proof}
Since $\kappa$ is biadditive on $\ZZ_2^{k+l} \times \ZZ_2^{k+l}$ and since $\rho$ is a
homomorphism, the map $(m,n) \mapsto (-1)^{\kappa(\rho(m), \rho(n))}$ is a bicharacter of
$\NN^k \times \ZZ_2^l$. It follows immediately from the definition of $\kappa$ that
$c_\Lambda(r(\lambda), \lambda) = 1 = c_\Lambda(\lambda, s(\lambda))$. Since $d$ is a
functor, for a composable triple $\lambda,\mu,\nu$,
\[
c_\Lambda(\lambda\mu,\nu) = c_\Lambda(\lambda,\nu) c_\Lambda(\mu,\nu)\qquad\text{ and }\qquad
    c_\Lambda(\lambda,\mu\nu) = c_\Lambda(\lambda,\mu)c_\Lambda(\lambda,\nu).
\]
Hence
\[
c_\Lambda(\lambda,\mu)c_\Lambda(\lambda\mu, \nu)
    = c_\Lambda(\lambda,\mu) c_\Lambda(\lambda,\nu) c_\Lambda(\mu,\nu)
    = c_\Lambda(\lambda,\mu\nu)c_\Lambda(\mu,\nu).\qedhere
\]
\end{proof}

\begin{dfn}
Let $\Lambda$ be a row-finite $P$-graph with no sources, and $c \in Z^2(\Lambda, \TT)$. A
Cuntz--Krieger $(\Lambda,c)$-family in a $C^*$-algebra $B$ is a function $t : \lambda
\mapsto t_\lambda$ from $\Lambda$ to $B$ such that
\begin{itemize}
\item[(CK1)] $\{t_v : v \in \Lambda^0\}$ is a collection of mutually orthogonal
    projections;
\item[(CK2)] $t_\mu t_\nu = c(\mu,\nu)t_{\mu\nu}$ whenever $s(\mu) = r(\nu)$;
\item[(CK3)] $t^*_\lambda t_\lambda = t_{s(\lambda)}$ for all $\lambda \in \Lambda$;
    and
\item[(CK4)] $t_v = \sum_{\lambda \in v\Lambda^p} t_\lambda t^*_\lambda$ for all $v
    \in \Lambda^0$ and $p \in P$.
\end{itemize}
\end{dfn}

The following lemma is more or less standard.

\begin{lem}
Let $\Lambda$ be a row-finite $P$-graph with no sources, and take $c \in Z^2(\Lambda,
\TT)$. There exists a universal $C^*$-algebra $C^*(\Lambda, c)$ generated by a
Cuntz--Krieger $(\Lambda, c)$-family $\{s_\lambda : \lambda \in \Lambda\}$.
\end{lem}
\begin{proof}
This follows from a standard argument (see, for example \cite[Theorem~2.10]{Loring})
using that the relations force the $t_\lambda$ to be partial isometries.
\end{proof}

Let $\Lambda$  be a row-finite $P$-graph with no sources.  The path space
$\Lambda^\Omega$ is defined to be the collection of all morphisms $x: \Omega \to \Lambda$
where $\Omega =  \Omega_P$ is as in Examples \ref{ex:pex} above. The collection of  sets
$Z(\lambda) :=  \{x \in \Lambda^\Omega : \lambda = x(0, d(\lambda))\}$, indexed by
$\lambda \in \Lambda$, is a basis of compact open sets for a locally compact Hausdorff
topology on $\Lambda^\Omega$. For each $r \in P$ we define the shift map $\sigma^r :
\Lambda^\Omega \to \Lambda^\Omega$ by $(\sigma^r x)(p,q) := x(p+r, q+r)$. We will need to
use  the path groupoid $\Gg_\Lambda$ of  $\Lambda$ introduced by Carlsen et al.\  in
\cite[\S{2}]{CKSS}:
\[
\Gg_\Lambda := \{(x, p - q, y):  p,q \in P \text{ and }\sigma^p(x) = \sigma^q(y)\}.
\]
We identify the path space  $\Lambda^\Omega$ with the  unit space $\Gg_\Lambda^{(0)}$ via
the map $x \mapsto (x, 0, x)$. For $\mu,\nu \in \Lambda$ with $s(\mu) = s(\nu)$, let
\[
Z(\mu, \nu) := \{ (x, d(\mu) - d(\nu), y):  x \in Z(\mu), y \in Z(\nu), \sigma^{d(\mu)}(x) = \sigma^{d(\nu)}(y) \}.
\]
The topology on $\Gg_\Lambda$ is the one with basis $\Uu_\Lambda = \{Z(\mu,\nu) : \mu,\nu
\in \Lambda, s(\mu) = s(\nu)\}$. This is a locally compact Hausdorff topology on
$\Gg_\Lambda$ and $\Gg_\Lambda$ is \'etale in this topology because $s : Z(\mu,\nu) \to
Z(\mu)$ is a homeomorphism for each $\mu$. The elements of $\Uu_\Lambda$ are all compact
open sets. Given $\mu, \nu \in \Lambda$ with $s(\mu) = s(\nu)$ and $p \in P$ we have
$Z(\mu, \nu) = \bigsqcup_{\alpha \in s(\mu)\Lambda^p}Z(\mu\alpha, \nu\alpha)$.

\begin{prp}\label{prp:grpd}
Let $\Lambda$ be a row-finite $P$-graph with no sources and take $c \in Z^2(\Lambda,
\TT)$. There is a continuous normalised $\TT$-valued groupoid 2-cocycle $\varsigma_c$ on
$\Gg_\Lambda$ and a surjective homomorphism $\pi : C^*(\Lambda, c) \to C^*(\Gg_\Lambda,
\varsigma_c)$ such that $\pi(s_\lambda) = 1_{Z(\lambda, s(\lambda))} \not= 0$ for all
$\lambda \in \Lambda$.
\end{prp}
\begin{proof}
We follow the argument of \cite[\S{6}]{KPS4}, which proves the analogous result in the
case that $\Lambda$ is a $k$-graph (see \cite[Theorem 6.7]{KPS4}). The only additional
difficulty in our current setting is that the notion of minimal common extension for a
pair of elements in $\Lambda$ does not in general make sense in a $P$-graph. So we must
check that we can still construct a partition of $\Gg_\Lambda$ as in
\cite[Lemma~6.6]{KPS4} without using minimal common extensions.

We must show that there is a partition $\Qq$ of $\Gg_\Lambda$ consisting of elements of
$\Uu_\Lambda$ such that $Z(\lambda, s(\lambda)) \in \Qq$ for all $\lambda \in \Lambda$.
We first observe that for each $p \in P$ the set $M_p := \{ (x, p, \sigma^p(x)) : x \in
\Lambda^{\Omega}\}$ is a clopen set and that $M_p = \bigsqcup_{d(\lambda ) = p}
Z(\lambda, s(\lambda))$. Moreover, the $M_p$ are pairwise disjoint, $M :=\bigsqcup_{p \in
P} M_p$ is also open and each $Z(\mu, \nu)$ is either contained in some $M_p$ or disjoint
from $M$. This implies that $M$ is also closed, so it is a clopen set. It remains to show
that the complement of $M$ can be expressed as a disjoint union  of elements of
$\Uu_\Lambda$. Since $\Uu_\Lambda$ is a countable basis, we can write the complement of
$M$ as a countable union of  sets in $\Uu_\Lambda$.

\emph{Claim:} Let $Z(\kappa, \lambda)$, $Z(\mu, \nu) \in \Uu_\Lambda$.
Then both $Z(\kappa, \lambda)\cap Z(\mu, \nu)$ and $Z(\kappa, \lambda)\setminus Z(\mu, \nu)$
can be expressed as a disjoint union of elements of $\Uu_\Lambda$.\\
Note that $Z(\kappa, \lambda)$ and $Z(\mu, \nu)$ are disjoint unless $d(\kappa) - d(\lambda) = d(\mu) - d(\nu)$;
so we may assume that $d(\kappa) - d(\lambda) = d(\mu) - d(\nu)$.
Then there exist $p, q \in P$ such that $d(\kappa) + p = d(\mu)  + q$ and $d(\lambda) + p = d(\nu)  + q$.
Let
\[
A:= \{ \alpha \in s(\lambda)\Lambda^p : Z(\kappa\alpha, \lambda\alpha)= Z(\mu\beta, \nu\beta)
 \text{ for some }\beta \in s(\nu)\Lambda^q \}
 \]
 and let $B := s(\lambda)\Lambda^p \setminus A$; then we have
\[
Z(\kappa, \lambda)\cap Z(\mu, \nu) = \bigsqcup_{\alpha \in A} Z(\kappa\alpha, \lambda\alpha) \quad\text{and}\quad
Z(\kappa, \lambda)\setminus Z(\mu, \nu) =  \bigsqcup_{\alpha \in B} Z(\kappa\alpha, \lambda\alpha).
\]
This proves the claim.

Since  intersections and set differences of elements of $\Uu_\Lambda$ can be expressed as
disjoint unions of such sets, any countable union of elements of $\Uu_\Lambda$ can also
be expressed as a disjoint union of such sets. This gives us the desired partition $\Qq$.

The groupoid 2-cocycle $\varsigma_c$ is constructed as in  \cite[Lemma 6.3]{KPS4} (there
the cocycle is denoted $\sigma_c$ and the partition is denoted $\Pp$). By construction of
$\varsigma_c$ the map $\lambda \mapsto 1_{Z(\lambda, s(\lambda))}$ constitutes a
Cuntz--Krieger $(\Lambda, c)$-family. Hence, there is a homomorphism  $\pi : C^*(\Lambda,
c) \to  C^*(\Gg_\Lambda, \varsigma_c)$ such that $\pi(s_\lambda) = 1_{Z(\lambda,
s(\lambda))}$. Moreover, for every $Z(\mu, \nu) \in \Uu_\Lambda$, we have
\[
1_{Z(\mu, \nu)} = 1_{Z(\mu, s(\mu))}1_{Z(\nu, s(\nu))}^* = \pi(s_{\mu}s_{\nu}^*);
\]
and since the span of elements of the form $1_{Z(\mu, \nu)}$ is dense, $\pi$ is
surjective.
\end{proof}

If $c \in Z^2(\Lambda, \TT)$ is the trivial cocycle, then our definition of the twisted
$C^*$-algebra $C^*(\Lambda, c)$ reduces to the definition of the $C^*$-algebra
$C^*(\Lambda)$ (see \cite[Definition 2.4]{CKSS}). If $P = \NN^k$ and $c$ is an arbitrary
cocycle, then our definition of $C^*(\Lambda, c)$ agrees with the existing definition of
the twisted $k$-graph algebra (see \cite[Definition 5.2]{KPS4}).

\begin{examples} \label{ex:pgraphalgebras}
\begin{enumerate}[(i)]
\item Let $\Omega_P$ be the $P$-graph of Examples~\ref{ex:pex}(i). Then
    $C^*(\Omega_P) \cong \Kk(\ell^2(P))$.
\item Let $P$ be the $P$-graph from Examples~\ref{ex:pex}(ii). Then $C^*(P)$ is
    isomorphic to the group $C^*$-algebra $C^* (G_P)$ of the Grothendieck group of
    $P$.
\end{enumerate}
\end{examples}
Let $(\Lambda,d)$ be a $P$-graph and $c \in Z^2 ( \Lambda , \TT )$. Following
\cite[Lemma~2.5]{CKSS} we describe the gauge action of $\widehat{G_P}$ on $C^* ( \Lambda
, c)$. For $\chi \in \widehat{G_P}$ and $\lambda \in \Lambda$ set
\begin{equation} \label{eq:gauge}
\gamma^\Lambda_\chi ( s_\lambda ) = \chi ( d(\lambda) ) s_\lambda .
\end{equation}

A standard argument (see \cite[Lemma~2.5]{CKSS}) shows that the above formula defines a
strongly continuous action of $\widehat{G_P}$ on $C^*(\Lambda, c)$.

\begin{prp}[Gauge invariant uniqueness theorem] \label{prp:giut}
Let $\Lambda$ be a row-finite $P$-graph with no sources, and fix $c \in Z^2
(\Lambda,\TT)$. Let $t : \Lambda \to B$ be a Cuntz--Krieger $(\Lambda,c)$-family in a
$C^*$-algebra $B$. Suppose that there is a strongly continuous action $\beta$ of
$\widehat{G_P}$ on $B$ satisfying $\beta_\chi (t_\lambda) = \chi(d(\lambda)) t_\lambda$
for all $\lambda \in \Lambda$ and $\chi \in \widehat{G_P}$. Then the induced homomorphism
$\pi_t : C^*(\Lambda,c) \to B$ is injective if and only if $t_v \not= 0$ for all $v \in
\Lambda^0$.
\end{prp}
\begin{proof}
The result follows from the same argument as in \cite[Proposition~2.7]{CKSS} and the
observation that the fixed point algebra for the gauge action on $C^*(\Lambda, c)$ is
identical to the fixed point algebra for the gauge action on $C^*(\Lambda)$ (cf.
\cite[Theorem~4.2]{KP5}).
\end{proof}

\begin{cor}
With notation as in Proposition \ref{prp:grpd} the map  $\pi : C^*(\Lambda, c) \to
C^*(\Gg_\Lambda, \varsigma_c)$ is an isomorphism.
\end{cor}
\begin{proof}
We argue as in \cite[Corollary 7.8]{KPS4}. The cocycle $\tilde{d} : \Gg_\Lambda \to G_P$
given by $\tilde{d}(x, m, y) = m$ induces an action $\beta$ of $\widehat{G_P}$ on
$C^*(\Gg_\Lambda, \varsigma_c)$ satisfying $\beta_\chi(f)(x,p-q,y) = \chi(p-q)
f(x,p-q,y)$ for $f \in C_c(\Gg_\Lambda)$. By construction, $\pi$ intertwines this action
with the gauge action. Proposition~\ref{prp:grpd} shows that each $\pi(s_v)$ is nonzero.
So the result follows from Proposition~\ref{prp:giut}.
\end{proof}

\begin{lem}(cf.\ \cite[Proposition 3.2]{FPS})\label{lem:FPS}
Let $F$ be a countable abelian group and let $\Lambda$ be a row-finite $k$-graph with no
sources. Suppose that $F$ acts on $\Lambda$ by automorphisms $g \mapsto \rho_g$. Let $P =
\NN^k \times F$. There is a unique $P$-graph $\Lambda \times_\rho F$ such that
\begin{enumerate}
\item As a set, $\Lambda \times_\rho F = \Lambda \times F$ with degree map
    $d(\lambda,g)=(d(\lambda),g)$;
\item
$r(\lambda, g) = (r(\lambda ), 0)$, $s(\lambda, g) = (s(\rho_{-g}(\lambda)), 0)$;
\item $(\mu, g)(\nu, h) = (\mu \rho_g(\nu), g+h)$ whenever $s(\mu) = r(\rho_g \nu)$.
\end{enumerate}
\end{lem}
\begin{proof}
For associativity we compute
\begin{align*}
\big((\lambda,g)(\mu,h)\big)(\nu,k) &= (\lambda\rho_g(\mu), g+h)(\nu,k)
     = (\lambda\rho_g(\mu) \rho_{g+h}(\nu), g+h+k) \\
    &= (\lambda, g)(\mu\rho_h(\nu), h+k)
     = (\lambda,g)\big((\mu,h)(\nu,k)\big).
\end{align*}
For the factorisation property we suppose that $d(\lambda, g) = (m+n,h+k)$. Then $\mu =
\lambda(0,m)$ and $\nu = \rho_{-h}(\lambda(m,m+n))$ satisfy $(\lambda,g) = (\mu, h) (\nu,
k)$ with $d(\mu,h) = (m,h)$ and $d(\nu,k) = (n,k)$. To see that this factorisation is
unique, suppose that
\[
(\lambda, g) = (\mu, h')(\nu, k') = (\mu\rho_{h'}(\nu), h'+k')
\]
with $d(\mu,h') = (m,h)$ and $d(\nu, k') = (n,k)$. Then $h' = h$ and $k' = k$ by
definition of $d$. Furthermore, $\lambda = \mu\rho_{h}(\nu)$ where $d(\mu) = m$ and
$d(\rho_h(\nu)) = d(\nu) = n$. So the factorisation property in $\Lambda$ forces $\mu =
\lambda(0,m)$ and $\rho_h(\nu) = \lambda(m,m+n)$, forcing $\nu = \rho_{-h}(\lambda(m,
m+n))$.
\end{proof}

\begin{example}\label{eq:cp=omega}
Let $F$ be a countable abelian group. We can regard $F$ as a $0$-graph and then there is
an action $\tau$ of $F$ on this $0$-graph by translation. So Lemma~\ref{lem:FPS} yields
an $F$-graph $F \times_\tau F$. It is straightforward to check that $F \times_\tau F
\cong \Omega_F$ via the map $(g,h) \mapsto (g, g+h)$.
\end{example}

\begin{prp}\label{prp:action_stations}
Continue the notation of Lemma~\ref{lem:FPS}.  Then there is an action $\tilde{\rho} : F
\to \operatorname{Aut}(C^*(\Lambda))$ such that $\tilde{\rho}_g (s_\lambda) =
s_{\rho_g(\lambda)}$.
\end{prp}
\begin{proof}
This follows from the universal property of $C^*(\Lambda)$ (cf.
\cite[Proposition~3.1]{FPS}).
\end{proof}

\begin{thm}Continue the notation of Lemma~\ref{lem:FPS} and Proposition~\ref{prp:action_stations}.
Then
\begin{enumerate}
\item there is a unitary representation of $u: F \to \mathcal{U}
    \mathcal{M}\big(C^*(\Lambda \times_\rho F)\big)$ given by $u(g) = \sum_{v \in
    \Lambda^0} s_{(v,g)}$;
\item there is a homomorphism $\phi : C^*(\Lambda) \to C^*(\Lambda \times_\rho F)$
    given by $s_\lambda \mapsto s_{(\lambda,0)}$;
\item we have $u(g) \phi(a) u(g)^* = \tilde{\rho}_g(a)$ for all $a \in C^*(\Lambda)$
    and $g \in F$;
\item There is an isomorphism $\phi \times u : C^*(\Lambda) \rtimes_{\tilde{\rho}} F
    \to C^*(\Lambda \times_\rho F)$ such that $(\phi \times u) (s_\lambda, g) =
    s_{(\lambda,g)}$.
\end{enumerate}
\end{thm}
\begin{proof}
This follows from the proof of \cite[Theorem 3.4]{FPS} \textit{mutatis mutandis}.
\end{proof}

\begin{examples}
\begin{enumerate}[(i)]
\item
   Let $B_n$ be the $1$-graph with a single vertex $v$ and edges $f_1 , \ldots , f_n$
    (see Example~\ref{ex:pex}(v)). Let $\mathbb{Z}_n$ act on $B_n$ by cyclicly
    permuting the edges. Then $C^*(B_n) \times \ZZ_n \cong C^*(B_n \times \ZZ_n)$.
\item Let $\Lambda$ be a $k$-graph, let $F$ be a countable abelian group, and $b :
    \Lambda \to F$ be a functor. Then the skew product graph $\Lambda \times_b F$
    carries a natural $F$-action $\tau$ (see \cite[Remark~5.6]{KP2}) given by
    $\tau_g(\lambda, h) = (\lambda, g+h)$. We have
    \[
        C^*((\Lambda \times_b F) \times_\tau F)
            \cong C^*(\Lambda \times_b F) \rtimes_{\tilde\tau} F
            \cong C^*(\Lambda) \hatimes\mathcal{K}(\ell^2(F)).
    \]
\end{enumerate}
\end{examples}

\begin{thm} \label{thm:pgraphstructure}(cf.\ \cite[Proposition 3.5]{FPS})
Let $P = \NN^k \times F$ where $F$ is a countable abelian group. Suppose that $\Lambda$ is
a $P$-graph, and let $\Gamma$ denote the sub-$k$-graph $d^{-1}(\NN^k \times \{0\})$ of
$\Lambda$. For each $g \in F$ and $v \in \Lambda^0$ the sets $v\Lambda^{(0,g)}$ and
$\Lambda^{(0,g)} v$ are singletons. Moreover, there is an action $\rho$ of $F$ on
$\Gamma$ such that for all $\lambda \in \Gamma$ and $g \in F$, the unique elements $\mu
\in r(\lambda)\Lambda^{(0,g)}$ and $\nu \in s(\lambda)\Lambda^{(0,g)}$ satisfy $\mu
\rho_g(\lambda) = \lambda \nu$. Furthermore, $\Lambda$ is isomorphic to the $P$-graph
$\Gamma \times_\rho F$ of Lemma~\ref{lem:FPS}.
\end{thm}
\begin{proof}
The factorisation property ensures that each vertex $v \in \Lambda^0$ has unique
factorisations $v = \mu\nu$ with $\mu \in \Lambda^{(0,g)}$ and $\nu \in \Lambda^{(0,
-g)}$. Hence $v \Lambda^{(0,g)} = \{ \mu \}$, and similarly $\Lambda^{(0,g)}v = \{\nu\}$.

The map $\rho$ preserves degree by definition. It remains to show that $\rho_{g}$  is a
functor for each $g$ and $\rho_g \circ \rho_h = \rho_{g+h}$. Fix $\lambda_1, \lambda_2
\in \Gamma$ such that $\lambda_2 \lambda_1 \in \Gamma$ and $g \in F$. Let $v_0 =
s(\lambda_1)$, $v_1 = r(\lambda_1) = s(\lambda_2)$ and $v_2 = r(\lambda_2)$. Let $\mu_i$
be the unique element of $v_i\Lambda^{(0,g)}$ for $i=0,1,2$. Then
\[
\mu_i \rho_g(\lambda_i) = \lambda_i \mu_{i-1} \text{ for } i=1,2.
\]

Combining the two equations we get
\[
\mu_2 \rho_g(\lambda_2) \rho_g(\lambda_1)
    = \lambda_2\mu_1 \rho_g(\lambda_1) = \lambda_2 \lambda_1 \mu_0,
\]
and hence $\rho_g(\lambda_2\lambda_1) = \rho_g(\lambda_2)\rho_g(\lambda_1)$ by uniqueness
of factorisations. A similar argument shows that $\rho_g \circ \rho_g = \rho_{g+h}$ for
all $g,h \in F$.

For $\lambda \in \Lambda^{(n,g)}$ define $\psi(\lambda) = \big(\lambda
\big((0,0),(n,0)\big), g \big) \in \Gamma \times_\rho F$. Then $\psi$ is an isomorphism
of $P$-graphs. Its inverse is given as follows: let $(\lambda, g) \in \Gamma \times_\rho
F$ and $\mu$ be the unique element in $s(\lambda) \Lambda^{(0,g)}$; then
$\psi^{-1}(\lambda,g) = \lambda \mu$.
\end{proof}

\begin{cor}
Suppose that $F$ is a countable abelian group. Then the only possible $F$-graphs are
disjoint unions of quotients of $\Omega_F$ by subgroups of $F$.
\end{cor}
\begin{proof}
Let $\Lambda$ be a $F$-graph. Then Theorem~\ref{thm:pgraphstructure} applied in the case
$k=0$ shows that $\Lambda \cong \Gamma \times_\rho F$ where $\Gamma$ is a $0$-graph,
which is just a countable set. The group $F$ then acts on $\Gamma$, and each orbit is of
the form $F/H$ for some subgroup $H$ of $F$. Since the orbits of $\rho$ correspond to the
connected components of $\Gamma \times_\rho F$ each component of $\Lambda$ is isomorphic
to $F/H \times_{\tau_H} F$ where $\tau_H$ is the action of $F$ on $F/H$ induced by
translation. The result then follows from the identification of $\Omega_F$ with $F
\times_\tau F$ given in Example~\ref{eq:cp=omega}
\end{proof}

\section{Gradings of \texorpdfstring{$P$}{P}-graph algebras induced by functors} \label{sec:elvis}

Let $P$ be a finitely-generated, cancellative abelian monoid of the form $P \cong \NN^k
\times \ZZ_2^l$. We will be particularly interested in gradings of twisted $P$-graph
algebras that arise from functors from the underlying $P$-graphs into $\ZZ_2$.

\begin{lem}\label{lem:induced grading}
Let  $\Lambda$ be a $P$-graph, let $\delta : \Lambda \to \ZZ_2$ be a functor and let $c$
be a $\TT$-valued $2$-cocycle on $\Lambda$. Then there is a grading automorphism
$\alpha_\delta$ of $C^*(\Lambda, c)$ such that
\begin{equation}\label{eq:induced grading}
\alpha_\delta(s_\lambda) = (-1)^{\delta(\lambda)} s_\lambda\quad\text{ for all $\lambda \in \Lambda$.}
\end{equation}
For $i \in \ZZ_2$, we have $C^*(\Lambda, c)_i = \clsp\{s_\mu s^*_\nu : \delta(\mu) -
\delta(\nu) = i\}$.
\end{lem}
\begin{proof}
The universal property of $C^*(\Lambda, c)$ yields an automorphism $\alpha_\delta$
satisfying~\eqref{eq:induced grading}. For $\mu,\nu \in \Lambda$ and $j \in \ZZ_2$, we
have
\[
\frac{s_\mu s^*_\nu + (-1)^j\alpha(s_\mu s^*_\nu)}{2}
     = \begin{cases}
        s_\mu s^*_\nu &\text{ if $\delta(\mu) - \delta(\nu) = j$}\\
        0 &\text{ if $\delta(\mu) - \delta(\nu) = j+1$,}
     \end{cases}
\]
so the description of the $C^*(\Lambda, c)_i$ follows from~\eqref{eq:Ai calc}.
\end{proof}

\begin{ntn}\label{ntn:delta}
Consider nonnegative integers $k,l$, and let $P = \NN^k \times \ZZ_2^l$, regarded as a
finitely generated abelian monoid, with generators $E_P = \{g_1, \dots, g_{k+l}\}$. Let
$\pi : P \to \ZZ_2$ be the unique homomorphism such that $\pi(g_i) = 1$ for all $i \le
k+l$. Given a $P$-graph $\Lambda$, there is a functor $\delta_\Lambda : \Lambda \to
\ZZ_2$ such that
\begin{equation}\label{eq:deltaLambda}
\delta_\Lambda(\lambda) = \pi(d(\lambda))\quad\text{ for all $\lambda \in \Lambda$.}
\end{equation}
\end{ntn}

The following examples illustrate the connection between gradings of $P$-graph
$C^*$-algebras and twisted $P$-graph algebras. Specifically, the graded tensor product of
graded $P$-graph algebras can frequently be realised as a twisted $C^*$-algebra of the
cartesian-product graph. We return to this in Theorem~\ref{thm:swapping}.

\begin{examples} \label{ex:above}
\begin{enumerate}[(i)]
\item\label{it:eg Tk} For $k \ge 1$ recall from Examples~\ref{ex:pex}(iv) that $T_k$
    denotes the $k$-graph $\NN^k$ with degree functor given by the identity functor,
    and $C^*$-algebra isomorphic to $C(\TT^k)$. Endow $C^*(T_1)$ with the grading
    automorphism $\alpha_{\delta_{T_1}}$ induced by $\delta_{T_1}$ (see
    Notation~\ref{ntn:delta}); so the unitary generator is homogeneous of odd degree.
    Then the graded tensor product $C^*(T_1) \hatimes C^*(T_1)$ is not abelian. To
    see this, let $s_1$ denote the unitary generator of $C^*(T_1)$. Then
    \[
    (s_1 \hatimes 1) (1 \hatimes s_1) = (s_1 \hatimes s_1) \neq -
    (s_1 \hatimes s_1) = (1 \hatimes s_1) (s_1 \hatimes 1).
    \]
    In particular, the graded tensor product $C^*(T_1) \hatimes C^*(T_1)$ is not
    isomorphic to $C^*(T_2)$.  

    Instead, we claim that $C^*(T_1) \hatimes C^*(T_1) \cong C^*(T_2, c)$ with
    grading automorphism $\alpha_{\delta_{T_2}}$ and twisting $2$-cocycle $c :
    (T_2)^{*2} \to \TT$ with values in $\{\pm 1\}$ given by
    \[
    c(m,n) = (-1)^{m_2 n_1} \text{ for } (m,n) \in \NN^2 \times \NN^2 = (T_2)^{*2}.
    \]

    To see this, for $n \in \NN$ let $s_n \in C^*(T_1)$ denote the corresponding
    generator. Define elements $\{t_{(m,n)} : (m,n) \in \NN^2\} \subseteq C^*(T_1)
    \hatimes C^*(T_1)$ by $t_{(m,n)} := s_m \hatimes s_n$. Routine calculations using
    the definition of multiplication and involution in the graded tensor product show
    that the $t_{(m,n)}$ are a Cuntz--Krieger $(T_2, c)$-family; for example, we can
    check~(CK2) as follows:
    \[
    t_{(m,n)} t_{(p,q)} = (s_m \hatimes s_n) (s_p \hatimes s_q)
        = (-1)^{np} (s_m s_p \hatimes s_n s_q)
        = c((m,n), (p,q)) t_{m+p, n+q}.
    \]
    The universal property of $C^*(T_2, c)$ gives a homomorphism $\psi : C^*(T_2, c)
    \to C^*(T_1) \hatimes C^*(T_1)$, and an application of the gauge-invariant
    uniqueness theorem (Theorem~\ref{prp:giut}) shows that $\psi$ is an isomorphism.
    Finally, one checks on generators that $\psi$ intertwines the grading
    automorphisms.

\item\label{it:Z2 Cliff1} Recall that $Z_2$ denotes $\ZZ_2$ considered as a
    $\ZZ_2$-graph as in Examples~\ref{ex:pex}(ii). Let $\delta :=    \delta_{Z_2}$ be
    the identity map $Z_2 \to \ZZ_2$; so the associated grading automorphism
    $\alpha_\delta$ of $C^*(Z_2)$ makes the generator $u$ homogeneous of degree one.
    Then there is an isomorphism $(C^*(Z_2), \alpha_\delta)\cong \Cliff1$ that takes
    $s_1$ to $(1,-1)$ and $s_0$ to $(1,1)$.

\item\label{it:Cliff algs realised} Consider the graded tensor product $(C^*(Z_2),
    \alpha_\delta) \hatimes (C^*(Z_2), \alpha_\delta)$. As above, this is, in
    general, a nonabelian $C^*$-algebra. Indeed, let $c \in Z^2(Z_2 \times Z_2, \TT)$
    be the cocycle given by $c(m,n) = (-1)^{m_2n_1}$. Write $\delta := \delta_{Z_2}$
    and recall that $\delta_{Z_2 \times Z_2} : Z_2 \times Z_2 \to \ZZ_2$ satisfies
    $\delta(i,j) = i+j$. Then
    \[
        (C^*(Z_2), \alpha_\delta) \hatimes (C^*(Z_2), \alpha_\delta)
            \cong \big(C^*(Z_2 \times Z_2, c), \alpha_{\delta_{Z_2 \times Z_2}}\big).
    \]
    Since $(C^*(Z_2), \alpha_\delta) \cong \Cliff1$ as graded algebras, it follows
    that there is a graded isomorphism $C^*(Z_2 \times Z_2, c) \cong \Cliff2$.
    Indeed, we will see in Corollary~\ref{prp:graded} that for any $l \ge 1$, if $c$
    is the 2-cocycle on the $\ZZ^l_2$-graph $Z^l_2$ described at~\eqref{eq:cdef},
    then $\Cliff{n} \cong C^*(Z_2^n, c)$ as graded $C^*$-algebras, where $C^*(Z_2^n,
    c)$ carries the grading induced by $\delta_{Z_2^n}$ as above.

\item\label{eg:cp by alpha} Expanding on~(\ref{it:eg Tk}), let $A$ be a $C^*$-algebra
    with grading automorphism $\alpha$. Then, as at~\eqref{eq:beta-k}, $A
    \rtimes_\alpha \ZZ$ has a grading $\beta^1$ given by $\beta^1(i_A(a)i_\ZZ(n)) =
    (-1)^n i_A(\alpha(a))i_\ZZ(n)$; that is, the copy of $A$ retains its given
    grading, and the generating unitary in the copy of $C^*(\ZZ)$ is odd. Let $T_1$
    denote $\NN$ regarded as a $1$-graph. The universal property of $A \rtimes_\alpha
    \ZZ$ and straightforward computations show that the map $i_A(a)i_\ZZ(n) \mapsto a
    \hatimes s_n$ defines a graded isomorphism $A \rtimes_\alpha \ZZ \cong A \hatimes
    C^*(T_1)$. We return to this in the context of graph $C^*$-algebras in
    Examples~\ref{ex:above2} (\ref{eg:more cp by alpha}).
\end{enumerate}
\end{examples}

In Section~\ref{sec:graded tensor}, motivated by Examples~\ref{ex:above}(\ref{it:eg
Tk})~and~(\ref{it:Cliff algs realised}), we will investigate graded tensor products of
graded $P$-graph algebras, and show that these often coincide with twisted $C^*$-algebras
of cartesian-product graphs. To do this, we first need an alternative description of the
graded $C^*$-algebras of appropriate $P$-graphs as universal graded $C^*$-algebras.

Let $\Lambda$ be a $P$-graph where $P \cong \NN^k \times \ZZ_2^l$ and let $E_P$ be the
standard generators $\{ e_i : 1 \le i \le k+l \}$ of $P$. We call elements of $\Lambda$
with degree in $E_P$ \emph{edges}.

\begin{thm} \label{thm:differentgenerators}
Let $P = \NN^k \times \ZZ_2^l$ and let $\Lambda$ be a $P$-graph. Let $c_\Lambda$ be the
2-cocycle of Proposition~\ref{prp:clambdadef}. Then for each $e \in E_p$ such that $2e =
0$ and each $\lambda \in \Lambda^e$, there is a unique $\lambda^* \in
s(\lambda)\Lambda^e$. This $\lambda^*$ satisfies $s(\lambda^*) = r(\lambda)$,
$\lambda\lambda^* = r(\lambda)$, and $\lambda^*\lambda = s(\lambda)$. Moreover there
exists a $C^*$-algebra $D$ such that
\begin{enumerate}
\item $D$ is generated by partial isometries $\{t_\lambda : d(\lambda) \in E_P\}$ and
    mutually orthogonal projections $\{ p_v : v \in \Lambda^0 \}$ such that
    \begin{itemize}
    \item[(a)] $t_\lambda t_\mu = - t_{\mu'} t_{\lambda'}$ for all $\lambda , \lambda'
        \in \Lambda^{e}$, $\mu , \mu' \in \Lambda^{e'}$ with $\lambda \mu = \mu'
        \lambda'$, $e,e' \in E_P$ and $e \neq e'$;
    \item[(b)] if $d(\lambda ) = e \in E_P$ and $2e=0$ then $t_\lambda^* =
        t_{\lambda^*}$;
    \item[(c)] $t_\lambda^* t_\lambda = p_{s(\lambda)}$ for all $\lambda \in
        \Lambda^{e}$, $e \in E_P$;
    \item[(d)] for all $v \in \Lambda^0$ and $e \in E_P$ we have
    \[
    p_v = \sum_{\lambda \in v \Lambda^{e}} t_\lambda t_\lambda^* .
    \]
    \end{itemize}
\item $D$ is universal in the sense that for any other
    $C^*$-algebra $D'$ generated by elements $t'_\lambda$ satisfying (a)--(d), there
    is a homomorphism $D \to D'$ satisfying $t_\lambda \mapsto t'_\lambda$.
\end{enumerate}
The $C^*$-algebra $D$ carries a grading automorphism $\alpha$ satisfying
$\alpha(t_\lambda) = -t_\lambda$ whenever $d(\lambda) \in E_P$, and $\alpha(p_v) = p_v$
for all $v \in \Lambda^0$.  Moreover, if $\alpha_{\delta_\Lambda}$ is the grading of
$C^*(\Lambda, c_\Lambda)$ obtained from Lemma~\ref{lem:induced grading} applied to the
functor~\eqref{eq:deltaLambda}, then there is a graded isomorphism $\pi : C^*(\Lambda,
c_\Lambda) \to D$ such that
\[
\pi ( s_v ) = p_v \text{ for all } v \in \Lambda^0 \quad \text{ and } \quad
    \pi ( s_\lambda ) = t_\lambda \text{ whenever } d(\lambda) \in E_P.
\]
\end{thm}
\begin{proof}
For the first statement, suppose that $d(\lambda) = e$ with $2e = 0$. We have
$d(s(\lambda)) = 0 = e + e$ and so the factorisation property shows that there exist
$\lambda^*,\nu \in \Lambda^e$ such that $s(\lambda) = \lambda^*\nu$. Since $r(\lambda^*)
= s(\lambda)$, the pair $(\lambda, \lambda^*)$ is composable and since $\lambda\lambda^*
\in \Lambda^{2e} = \Lambda^0$ we have $s(\lambda^*) = \lambda\lambda^* = r(\lambda)$.
Hence, the pair $(\lambda^*, \lambda)$ is composable and $\lambda^*\lambda = r(\lambda)$.
The factorisation property ensures that this $\lambda^*$ is unique.

Let $\{ s_\lambda : \lambda \in \Lambda  \}$ be partial isometries generating
$C^*(\Lambda, c_\Lambda)$. Let $D$ be the universal $C^*$-algebra generated by a family
of $t_\lambda$'s and $p_v$'s satisfying (a)--(d). The universal property guarantees that
$D$ carries a grading automorphism $\alpha$ as described.

Define $T_\lambda = s_\lambda$ and $P_v = s_v$. Then the family $\{ P , T \}$ satisfies
conditions (c)~and~(d) above. Suppose that $e \in E_P$ satisfies $2e = 0$ and that
$\lambda \in \Lambda^e$. Using (CK2)--(CK4) and the first paragraph, we have
\[
s_{\lambda^*}
    = (s_\lambda^*  s_\lambda) s_{\lambda^*}
    = s_\lambda^* c_\Lambda(\lambda, \lambda^*) s_{\lambda \lambda^*} \
    = s_\lambda^* s_{r(\lambda)}
    = s_\lambda^* (s_\lambda s_\lambda^* )
    = s_\lambda^*
\]
and hence $T_{\lambda^*} = T_\lambda^*$.

It remains to check property (a): If $i \neq j$ then $c_\Lambda( e_i , e_j ) = -1$ if
$j<i$ and $c_\Lambda( e_i , e_j ) = 1$ otherwise. Let $\lambda , \lambda' \in \Lambda^{e_i}$, $\mu , \mu' \in
\Lambda^{e_j}$ with $\lambda \mu = \mu' \lambda'$. Suppose that $j<i$. Then
\begin{align*}
T_\lambda T_\mu &= s_\lambda s_\mu = c_{\Lambda} ( \lambda , \mu ) s_{\lambda \mu} = (-1)^1 s_{\lambda \mu} = - s_{\lambda \mu} \\
T_{\mu'} T_{ \lambda'} &= s_{\mu'} s_{\lambda'} = c_{\Lambda} ( \mu' , \lambda' ) s_{\mu'\lambda'} = (-1)^0 s_{\lambda \mu} = s_{\lambda \mu} .
\end{align*}

Hence $T_\lambda T_\mu = - T_{\mu'} T_{ \lambda'}$. If $i < j$, the same argument applies
(switching the $\lambda$'s and $\mu$'s). Hence by the universal property of $D$ there is
a map $\phi : D \to C^* ( \Lambda , c_\Lambda )$ such that $\phi ( p_v ) =s_v$ for all $v
\in \Lambda^0$  and $\phi ( t_\lambda ) = s_\lambda$ for all edges $\lambda$.

To show that $\phi$ has an inverse, for each $v \in \Lambda^0$ set $s_v = p_v$. For
$\lambda \in \Lambda$ with $d(\lambda) = \sum_{i=1}^{k+l} m_i e_i$, use the factorisation
property to write
\[
\lambda = \lambda_1^1 \cdots \lambda_1^{m_1} \lambda_2^1 \cdots \lambda_2^{m_2} \cdots \lambda_{k+l}^1 \cdots \lambda_{k+l}^{m_{k+l}}
\]
where $d ( \lambda_i^j ) = e_i$ for $j=1 , \ldots , m_i$  and $i=1 , \ldots , k+l$, and
define
\[
S_\lambda := t_{ \lambda_1^1} \cdots t_{\lambda_1^{m_1}} t_{\lambda_2^1} \cdots t_{\lambda_2^{m_2}} \cdots t_{\lambda_{k+l}^1} \cdots t_{\lambda_{k+l}^{m_{k+l}}} .
\]

Direct calculation shows that $\{S_\lambda : \lambda \in \Lambda\}$ is a Cuntz--Krieger
$( \Lambda , c_\Lambda )$-family in $D$. By the universal property of $C^* ( \Lambda ,
c_\Lambda )$ there is a map $\psi : C^* ( \Lambda , c_\Lambda ) \to D$ such that $\psi (
s_\lambda ) = S_\lambda$. By construction the maps $\psi$ and $\phi$ are mutually inverse
and so $D \cong C^* ( \Lambda , c_\Lambda )$.

The final assertion follows by the universality of $D$.
\end{proof}

\section{Graded tensor products of twisted \texorpdfstring{$P$}{P}-graph \texorpdfstring{$C^*$}{C*}-algebras}\label{sec:graded
tensor}

Let $P = \NN^k \times \ZZ_2^a$ and let $Q = \NN^l \times \ZZ_2^b$. Let $\Lambda$ be a
$P$-graph and $\Gamma$ a $Q$-graph. Then $\Lambda \times \Gamma$ is a $P \times Q$ graph.
The functor $\delta_{\Lambda \times \Gamma}$ and the $2$-cocycle $c_{\Lambda \times
\Gamma}$ defined in Proposition~\ref{prp:clambdadef} via~\eqref{eq:cdef}, still make
sense as we are using the map $\pi : P \times Q \to \ZZ_2^{(k+a)+(l+b)} $ to define
$c_{\Lambda\times \Gamma}$.

\begin{thm} \label{thm:swapping}
Let $P = \NN^k \times \ZZ_2^a$ and let $Q = \NN^l \times \ZZ_2^b$. Let $\Lambda$ be a
$P$-graph and $\Gamma$ a $Q$-graph. Then there is an isomorphism of graded $C^*$-algebras
\[
C^*(\Lambda \times \Gamma, c_{\Lambda \times \Gamma})  \cong
C^*(\Lambda, c_\Lambda) \hatimes C^*(\Gamma, c_\Gamma)
\]
with respect to the gradings $\alpha_{\delta_{\Lambda \times \Gamma}}$ and
$\alpha_{\delta_\Lambda} \hatimes \alpha_{\delta_{\Gamma}}$.
\end{thm}
\begin{proof}
By Theorem~\ref{thm:differentgenerators}, the graded $C^*$-algebra
$C^*(\Lambda \times \Gamma, c_{\Lambda \times \Gamma})$ is generated by families
$\{ p_{(v,w)}, t_{(\lambda,w)}, t_{(v,\mu)} \}$ satisfying (a)--(d) of
Theorem~\ref{thm:differentgenerators}. We define
\begin{align*}
T_{(\lambda, w)} = s_\lambda \hatimes s_w & \text{ for } \lambda \in \Lambda^{e_i} , w \in \Gamma^0 \\
T_{(v , \mu )} = s_v \hatimes s_\mu & \text{ for } \mu \in \Gamma^{e_j} , v \in \Lambda^0 \\
P_{(v,w)}= s_v \hatimes s_w & \text{ for } v \in \Lambda^0 , w \in \Gamma^0 .
\end{align*}
Then $\{P_{(v,w)} : (v,w) \in \Lambda^0 \times \Gamma^0\}$ is a family of mutually
orthogonal projections. We must check that the $P_{(v,w)}$ and the $T_{(\lambda, w)}$ and
$T_{(v, \mu)}$ satisfy relations (a)--(d) for the $(P \times Q)$-graph $\Lambda \times
\Gamma$ and the cocycle $c_{\Lambda \times \Gamma}$. Condition (a) is the most difficult,
and we present it here; (b)--(d) are routine.

Let $\bar\lambda, \bar\lambda' , \bar\mu , \bar\mu'$ be edges in $\Lambda \times \Gamma$
such that $d ( \bar\lambda ) = d ( \bar\lambda')$, $d ( \bar \mu ) = d ( \bar\mu' )$, $d
( \bar\lambda ) \neq d ( \bar\mu )$, $s(\bar\lambda) = r(\bar\mu)$, and $\bar \lambda
\bar \mu = \bar\mu' \bar\lambda'$. There are four combinations to check according to
whether
\begin{align*}
d ( \bar\lambda ) &= (e_i , 0) , 1 \le i \le k \text{ or } d ( \bar\lambda ) = ( 0, e_i ) , 1 \le i \le l \text{ and}; \\
d ( \bar\mu ) &= ( e_j , 0) , 1 \le j \le k \text{ or } d ( \bar\mu ) = ( 0, e_j ) , 1 \le j \le l .
\end{align*}

First suppose that $d ( \bar\lambda ) = ( e_i , 0 )$ and $d ( \bar\mu ) = ( e_j , 0)$
where $i \neq j$. Then $\bar\lambda = (\lambda, v)$, $\bar\mu = (\mu,v)$, $\bar\lambda' =
(\lambda', v)$ and $\bar\mu' = (\mu', v)$ for some $v \in \Gamma^0$ and some
$\lambda,\mu' \in \Lambda^{e_i}$, $\mu,\lambda' \in \Lambda^{e_j}$ with $\lambda\mu =
\mu'\lambda'$. We then have
\[
T_{\bar\lambda} T_{\bar\mu} = T_{(\lambda,v)} T_{(\mu,v)} = ( s_\lambda \hatimes s_v ) ( s_\mu \hatimes s_v )
    = (-1)^{\partial{s_v}\cdot\partial{s_\mu}} ( s_\lambda s_\mu \hatimes s_v )
    = ( s_{\lambda \mu} \hatimes s_v )
\]
since $\partial s_v = 0$. On the other hand,
\begin{align*}
T_{\bar\mu'} T_{\bar\lambda'}
    &= T_{(\mu',v)} T_{(\lambda',v)}
     = ( s_{\mu'} \hatimes s_v ) ( s_{\lambda'} \hatimes s_v ) \\
    &= (-1)^{\partial{s_v}\cdot\partial{s_{\lambda'}}} ( s_{\mu'} s_{\lambda'} \hatimes s_v )
     = s_{\mu'} s_{\lambda'} \hatimes s_v
     = -( s_{\lambda \mu} \hatimes s_v ).
\end{align*}

Now suppose that $d ( \bar\lambda ) = ( e_i , 0 )$ and $d ( \bar\mu ) = ( 0, e_j )$ where
$i \neq j$. Then $\bar\lambda = (\lambda, r(\mu))$, $\bar\mu = (s(\lambda), \mu)$,
$\bar\mu' = (r(\lambda),\mu)$ and $\bar\lambda' = (\lambda, s(\mu))$ for some $\lambda
\in \Lambda^{e_i}$ and $\mu \in \Gamma^{e_j}$. So
\begin{align*}
T_{\bar\lambda} T_{\bar\mu} = T_{(\lambda,r(\mu))} T_{(s(\lambda) ,\mu)}
    &= ( s_\lambda \hatimes s_{r(\mu)} ) ( s_{s(\lambda)} \hatimes s_\mu  )\\
    &= (-1)^{\partial{s_{r(\mu)}}\cdot\partial{s_{r(\lambda)}}} ( s_{s(\lambda)} s_\lambda  \hatimes s_\mu s_{s(\mu)} )
    = ( s_{\lambda} \hatimes s_\mu ),
\end{align*}
whereas
\begin{align*}
T_{\bar\mu'} T_{\bar\lambda'}
     = T_{(r (\lambda'), \mu')} T_{(\lambda', s (\mu'))}
    &= ( s_{r(\lambda')} \hatimes s_{\mu'} ) ( s_{\lambda'} \hatimes s_{s(\mu')} ) \\
    &= (-1)^{\partial{s_{\mu'}}\cdot\partial{s_{\lambda'}}} ( s_{r ( \lambda')}  s_{\lambda'} \hatimes s_{\mu'} s_{s(\mu')}  )
     = - ( s_{\lambda} \hatimes s_\mu ) .
\end{align*}
The remaining two cases are similar.

By the universal property of $C^* ( \Lambda \times \Gamma , c_{\Lambda \times \Gamma} )$
there is a map $\pi : C^* ( \Lambda \times \Gamma , c_{\Lambda \times \Gamma} ) \to
C^*(\Lambda, c_\Lambda) \hatimes C^*(\Gamma, c_\Gamma)$ such that $\pi ( t_{(\lambda,w)})
= T_{(\lambda,w)}$, $\pi ( t_{(v,\mu)} ) = T_{(v,\mu)}$ and $\pi ( p_{(v,w)} ) =
P_{(v,w)}$. This $\pi$ is surjective because $C^*(\Lambda, c_\Lambda) \hatimes
C^*(\Gamma, c_\Gamma)$ is generated by the elements $s_\lambda \hatimes s_\mu =
T_{(\lambda, r(\mu))} T_{(s(\lambda), \mu)}$. We aim to apply the gauge invariant
uniqueness theorem for twisted $P$-graph algebras given in Proposition~\ref{prp:giut} to
show that $\pi$ is injective. For this, observe that the projections $P_{(v,w)}= s_v
\hatimes s_w$ are nonzero, so it suffices to show that, identifying $\widehat{G}_{P
\times Q}$ with $\widehat{G}_P \times \widehat{G}_Q$ in the canonical way, $\pi$ is
equivariant for the gauge-action $\gamma^{\Lambda \times \Gamma}$ on $C^* ( \Lambda
\times \Gamma , c_{\Lambda \times \Gamma} )$ and the action $\gamma^\Lambda \hatimes
\gamma^\Gamma$ on $C^*(\Lambda, c_\Lambda) \hatimes C^*(\Gamma, c_\Gamma)$ such that
\[
(\gamma^\Lambda_\chi \hatimes \gamma^\Gamma_{\chi'}) ( s_\lambda \hatimes s_\mu )
    =  \chi(d(\lambda)) \chi'(d(\mu)) (s_\lambda \hatimes s_\mu)
\]
for all $(\chi, \chi')\in \widehat{G}_P \times \widehat{G}_Q$, $\lambda \in \Lambda$ and
$\mu \in \Gamma$.

Since $\gamma^{\Lambda \times \Gamma}_{(\chi,\chi')} s_{(\lambda,\mu)} = \chi (d(\lambda)
) \chi' (d(\mu)) s_{(\lambda,\mu)}$ we see that $\pi$ is equivariant on the generators
$p_{(v,w)}$, $t_{(\lambda,w)}$, $t_{(v,\mu)}$, and therefore on $C^* ( \Lambda \times
\Gamma , c_{\Lambda \times \Gamma} )$.
\end{proof}

An interesting special case of Theorem~\ref{thm:swapping} occurs when $\Gamma$ is the
$\ZZ_2$-graph $Z_2$, so that $C^*(\Gamma) \cong \Cliff1$ as graded $C^*$-algebras.

\begin{cor}\label{cor:tensor z2}
Let $P = \NN^k \times \ZZ_2^a$ and let $\Lambda$ be a $P$-graph. Then
\[
C^*(\Lambda \times Z_2, c_{\Lambda \times Z_2})
    \cong C^*(\Lambda, c_\Lambda) \hatimes C^*(Z_2)
    \cong C^*(\Lambda, c_\Lambda) \hatimes \Cliff1.
\]
with respect to the gradings $\alpha_{\delta_{\Lambda \times Z_2}}$ and
$\alpha_{\delta_\Lambda} \hatimes \alpha_{\delta_{Z_2}}$.
\end{cor}
\begin{proof}
The first statement follows from Theorem~\ref{thm:swapping} because $Z_2$ has trivial
second cohomology, so $C^*(Z_2, c_{Z_2}) \cong C^*(Z_2)$. The second statement follows
from Examples~\ref{ex:above}(\ref{it:Z2 Cliff1}).
\end{proof}

\begin{rmk}
Since the graded tensor product with $\Cliff1$ is like a graded suspension operation,
Corollary~\ref{cor:tensor z2} has implications for graded $K$-theory. Let $P = \NN^k
\times \ZZ_2^l$ for some $k,l$, and let $\Lambda$ be a $P$-graph. Then
$\Kgr_i(C^*(\Lambda \times Z_2, c_{\Lambda \times Z_2})) \cong \Kgr_{i+1}(C^*(\Lambda,
c_{\Lambda}))$, and then inductively
\[
\Kgr_i(C^*(\Lambda \times Z_2^n, c_{\Lambda\times Z_2})) \cong
    \Kgr_{i+n}(C^*(\Lambda, c_{\Lambda})).
\]
\end{rmk}

\begin{cor}
Let $\Lambda$ be the $\ZZ_2^l$-graph $\prod^l_{i=1} Z_2$. Then $C^*(\Lambda, c_\Lambda)
\cong \Cliff{l}$, the $l$th complex Clifford algebra. This isomorphism is a graded
isomorphism with respect to the grading $\delta_{\Lambda}$ of $C^*(\Lambda, c_\Lambda)$.
\end{cor}
\begin{proof}
We have $C^*(Z_2) \cong \Cliff1$ as graded $C^*$-algebras as discussed in
Example~\ref{ex:above}(\ref{it:Z2 Cliff1}). So the result follows from an induction
argument using Corollary~\ref{cor:tensor z2} and the definition $\Cliff{l+1} = \Cliff{l}
\hatimes \Cliff1$.
\end{proof}

So far we have discussed gradings arising from functors from $k$-graphs into $\ZZ_2$, but
there are other possible gradings including those arising from order two automorphisms of $k$-graphs.
Let $\theta$ be an order two automorphism of a row-finite $k$-graph $\Lambda$ with no sources.
Then $\theta$ induces a grading $\beta_{\theta}$ of $C^*(\Lambda)$ satisfying
$\beta_{\theta}(s_\lambda) = s_{\theta(\lambda)}$. With respect to this grading,
\begin{align*}
C^*(\Lambda)_0 &= \clsp\{s_\lambda s_\mu^* + s_{\theta(\lambda)} s^*_{\theta(\mu)}
                        : s(\lambda)= s(\mu)\},\quad\text{ and}\\
C^*(\Lambda)_1 &= \clsp\{s_\lambda s_\mu^* - s_{\theta(\lambda)} s^*_{\theta(\mu)}
                        : s(\lambda)= s(\mu)\}.
\end{align*}

\begin{prp} \label{prp:graded}
With notation as above there is a graded isomorphism
\[
\rho : C^*(\Lambda \times_\theta \ZZ_2) \to C^*(\Lambda) \hatimes \Cliff1
\]
such that $\rho(s_{(\lambda, i)}) = s_\lambda \hatimes u^i$, where $\Lambda \times_\theta
\ZZ_2$ is the crossed-product $(\NN^k \times \ZZ_2)$-graph, and $C^*(\Lambda
\times_\theta \ZZ_2)$ is graded by the automorphism $\tilde{\beta}_{\theta}$ such that
$\tilde{\beta}_{\theta}(s_{(\lambda, i)}) = (-1)^i s_{(\theta(\lambda), i)}$.
\end{prp}
\begin{proof}
Direct calculation shows that the elements $t_{(\lambda, i)} := s_\lambda \hatimes u^i$
constitute a Cuntz--Krieger $(\Lambda \times_\theta \ZZ_2)$-family in $C^*(\Lambda)
\hatimes \Cliff1$. So the universal property of $C^*(\Lambda \times_\theta \ZZ_2)$ gives
a homomorphism $\rho: C^*(\Lambda \times_\theta \ZZ_2) \to C^*(\Lambda) \hatimes \Cliff1$
taking $s_{(\lambda,i)}$ to $t_{(\lambda, i)} = s_\lambda \hatimes u^i$. An application
of the gauge-invariant uniqueness theorem (Proposition \ref{prp:giut}) shows that $\rho$
is injective; it is surjective because its image contains the generators of $C^*(\Lambda)
\hatimes \Cliff1$. Therefore $\rho$ is an isomorphism.
Let $\alpha$ be the grading automorphism of $\Cliff1$. 
Then
\[
\rho(\tilde{\beta}_{\theta}(s_{(\lambda, i)}))
    = \rho((-1)^is_{(\theta(\lambda), i)})
    = (-1)^i s_{\theta(\lambda)} \hatimes u^i
    = (\beta_{\theta} \hatimes\alpha)(s_\lambda \hatimes u^i)
    = (\beta_{\theta} \hatimes\alpha)\rho(s_{(\lambda,i)});
\]
hence $\rho$ intertwines the grading automorphisms of the two algebras.
\end{proof}

\section{Graded \texorpdfstring{$K$}{K}-theory of graph \texorpdfstring{$C^*$}{C*}-algebras}\label{sec:graph Kth}
In this section we apply the sequence~\eqref{eq:6-termgraded} to a graph $C^*$-algebra
$C^*(E)$ graded by an automorphism $\alpha_\delta$ determined by a function $\delta : E^1
\to \ZZ_2$---see Corollary~\ref{cor:graph Kgr}.

\begin{rmk} \label{rmk:split}
Following \cite{FowlerRaeburn} (see also \cite[\S 8]{Raeburn}), given a $1$-graph $E$, we
can realise $C^*(E)$ as the Cuntz--Pimsner algebra of the graph module $X(E)$ defined as
the Hilbert-bimodule completion of $C_c(E^1)$, regarded as a
$C_0(E^0)$--$C_0(E^0)$-bimodule under the left and right actions given by $(a \cdot x
\cdot b)(e) = a(r(e)) x(e) b(s(e))$, under the inner-product $\langle x,
y\rangle_{C_0(E^0)}(v) = \sum_{s(e) = v} \overline{x(e)}y(e)$. By \cite[Proposition
4.4]{FowlerRaeburn} this left action is by compacts if $E$ is row-finite, and $X(E)$ is
full if $E$ has no sinks. The left action is injective if $E$ has no sources. Observe
that $C_0(E^0)$, is separable and nuclear. It carries
the trivial grading $\alpha_{C_0(E^0)} = \id$.  

Fix a function $\delta : E^1 \to \ZZ_2$. This $\delta$ extends uniquely to a functor
$\delta : E^* \to \ZZ_2$. There is a grading $\alpha_{X(E)}$ on $X(E)$ determined by
$\alpha_{X(E)}(1_e) = (-1)^{\delta(e)} 1_e$. It is straightforward to check that $(X,
\alpha_X)$ is a graded $C_0(E^0)$--$C_0(E^0)$-correspondence. By
\cite[Proposition~12]{FLR} (see also \cite[Example 8.13]{Raeburn}) we have $C^*(X(E))
\cong C^*(E)$. Hence \eqref{eq:6-termgraded} becomes
\begin{equation}\label{eq:exact seq for graph}
0 \to \Kgr_1(C^*(E) , \alpha ) \hookrightarrow
    \Kgr_0(C_0(E^0)) \stackrel{{\scriptstyle 1 - [X(E)]}}{\longrightarrow}
    \Kgr_0(C_0(E^0)) \twoheadrightarrow
    \Kgr_0(C^*(E) , \alpha ) \to 0
\end{equation}
since $\Kgr_1(A, \alpha_A) = \bigoplus_{v \in E^0} K_1(\CC) = 0$.
\end{rmk}

To apply \eqref{eq:exact seq for graph} to compute the graded $K$-theory of the
$C^*$-algebra associated to a $1$-graph $E$ we need to examine the central terms more
closely. We describe $\Kgr_0 ( C_0 ( E^0 ) )$ in a way which allows us to compute the map
$\hatimes_{C_0 (E^0)} ( 1 - [X(E)] )$.

Let $E$ be a row-finite $1$-graph with no sources. By definition, we have $\Kgr_0 ( C_0 (
E^0 ) ) = KK ( \CC , \CC^{E^0} )$. Let $\CC\delta_v$ be a copy $\{z\delta_v : z \in
\CC\}$ of $\CC$ as a vector space with inner-product given by $\langle z\delta_v,
z'\delta_v\rangle_{\CC^{E^0}}(u) = \delta_{v,u} \overline{z}z'$ and right action
$z\delta_v \cdot a = a(v)z\delta_v$. It carries a left action $\varphi_v$ of $\CC$ by
multiplication. The tuple $(\CC \delta_v, \varphi_v, 0, \id)$ is a Kasparov
$\CC$--$\CC^{E^0}$. The group $KK ( \CC , \CC^{E^0} )$ is generated by the Kasparov
$\CC$--$\CC^{E^0}$ modules $[\CC \delta_v] := [\CC \delta_v, \varphi_v, 0, \id]$ for $v
\in E^0$, and there is an isomorphism $\theta : \ZZ E^0 \to KK_0 (\CC, C_0(E^0))$ such
that $\theta(1_v) = [\CC \delta_v]$, where $1_v$ is the generator of $\ZZ E^0$
corresponding to $v$.

Now we describe the map $\hatimes_{C_0 (E^0)}  [X(E)]$ on $\ZZ E^0$ induced by the
isomorphism $\theta$. Let $A^\delta_E$ be the $E^0 \times E^0$ matrix defined by
\begin{equation}
  \label{eq:signedmatrix}
A^\delta_E(v,w) = \sum_{e \in vE^1w} (-1)^{\delta(e)}
\end{equation}
(the empty sum is equal to 0 by convention). If $E_j$ denotes the subgraph $(E^0,
\delta^{-1}(j), r, s)$ of $E$ for $j=0,1$, then $A^\delta_E$ is just $A_{E_0} - A_{E_1}$.

\begin{lem}\label{lem:graph K-th}
Let $E$ be a row-finite $1$-graph with no sources or sinks. Then with notation as above,
the following diagram commutes.
\[
\begin{tikzpicture}[>=stealth,xscale=1.5]
    \node (00) at (0,0) {$\ZZ E^0$};
    \node (02) at (0,2) {$\ZZ E^0$};
    \node (20) at (2,0) {$KK_0 (\CC, \CC^{E^0})$};
    \node (22) at (2,2) {$KK_0 (\CC, \CC^{E^0})$};
    \draw[->](02)-- node[left] {$\scriptstyle (A^\delta_E)^t$} (00);
    \draw[->](22)-- node[right] {$\hatimes_{C_0( E^0 )} [X(E)]$} (20);
    \draw[->](00)-- node[above] {$\theta$} (20);
    \draw[->](02)-- node[above] {$\theta$} (22);
\end{tikzpicture}
\]
\end{lem}
\begin{proof}
It suffices to check the diagram commutes on generators $1_v$. Fix $v \in E^0$. Using
Lemma~\ref{lem:grading adjointable} at the second equality we calculate:
\begin{align*}\textstyle
\theta(1_v) \hatimes_{C_0(E^0)} [X(E)] &= [\CC \delta_v] \hatimes_{C_0(E^0)} [X(E)]
     = [ \CC \delta_v] \hatimes_{C_0 (E^0)} \big( [X(E)_0] - [X(E)_1] \big) \\
    &= \sum_{e \in vE^1w, \delta(e)=0} [ \CC \delta_w] - \sum_{f \in vE^1w , \delta(f)=1} [ \CC \delta_w]
     = \sum_{g \in vE^1w} (-1)^{\delta(g)} [ \CC \delta_w].
\end{align*}
This is precisely $\theta( (A^\delta_E)^t 1_v)$.
\end{proof}

We now use Corollary~\ref{cor:6-termgraded} and  Lemma~\ref{lem:graph K-th} to compute
the graded $K$-theory of graph $C^*$-algebras for suitable gradings.

\begin{cor}\label{cor:graph Kgr}
Let $E$ be a row-finite $1$-graph with no sources or sinks, and let $\delta : E \to
\ZZ_2$ be the functor determined by the function $\delta : E^1 \to \ZZ_2$. Let
$\alpha_\delta$ be the associated grading $\alpha_\delta(s_e) = (-1)^{\delta(e)}s_e$ of
$C^*(E)$. Then with $A^\delta_E$ as in \eqref{eq:signedmatrix}
\begin{align*}
K_0^{gr}(C^*(E),\alpha_\delta) &\cong \coker(1 - (A^\delta_E)^t : \ZZ E^0 \to \ZZ E^0 ) \text{ and } \\
\Kgr_1(C^*(E),\alpha_\delta) & \cong \ker(1 - (A^\delta_E)^t : \ZZ E^0 \to \ZZ E^0 ) .
\end{align*}
\end{cor}

Note that if $E^0$ is finite, then $\Kgr_1 (C^*(E), \alpha_\delta)$ is a free abelian
group with the same rank as $\Kgr_0(C^*(E), \alpha_\delta)$.

\begin{rmk}
Corollary~\ref{cor:graph Kgr} provides a direct parallel with \cite[Corollary~4.2.5]{PR}
(see also \cite[Example~7.2]{Raeburn}): given $\delta : E^1 \to \ZZ_2$, the graded
$K_0$-group $\Kgr_0(C^*(E), \alpha_\delta)$ is generated as an abelian group by the
classes of the vertex projections $\{p_v : v \in E^0\}$ subject only to the relations
\[
[p_v] = \big[\sum_{e \in vE^1} s_e s^*_e\big]
    = \sum_{e \in vE^1} (-1)^{\delta(e)} [s_e^* s_e]
    = \sum_{w \in E^0} (A^\delta_E)^t(v,w) [p_w]
\]
coming from Example~\ref{eg:odd pi}. This motivates, in part, our conjecture in
Section~\ref{sec:conj} below.

In particular, taking $\delta \equiv 0$, we recover the well-known formula for the
(ungraded) $K$-theory of a $1$-graph $C^*$-algebra \cite[Theorem~4.2.4]{PR}.
\end{rmk}

\begin{examples}\label{eg:exampleorama}
\begin{enumerate}
\item\label{it:KOn} As in Examples~\ref{ex:pex}, for $1 \le n < \infty$ let $B_n$ be
    the $1$-graph with one vertex and $n$ edges. Fix $\delta : B_n^1 \to \ZZ_2$, and
    let $p := |\delta^{-1}(1)|$ and $q = |\delta^{-1}(0)|$, so that $p+q = n$. Then
    $(A_{B_n}^\delta)^t$ is the $1 \times 1$ matrix $(q-p)$. Since $C^* ( B_n ) \cong
    \Oo_n$ we recover the formula for $\Kgr_*(\Oo_n)$ obtained by Haag in
    \cite[Proposition~4.11]{H1}:
    \[
    \Kgr_*(\Oo_n, \alpha_\delta)
        \cong \begin{cases}
            \big(\ZZ_{|1+p-q|} , 0 \big) & \text{ if } 1+p-q \neq 0, \\
            (\ZZ, \ZZ) & \text{ otherwise.}
        \end{cases}
    \]

\item\label{it:K2} Let $K_2$ be the $1$-graph associated to the complete directed
    graph on two vertices. Endow $K_2$ with the map $\delta' : K_2^1 \to \ZZ_2$ for
    which $A_{K_2}^{\delta'}= \big(\begin{smallmatrix}
    -1&-1\\1&-1\end{smallmatrix}\big)$. Then $\Kgr_0(C^*(K_2), \alpha_{\delta'})
    \cong \ZZ_5$. In particular, this and Example~(\ref{it:KOn}) above show that
    although $C^*(K_2) \cong \Oo_2 \cong C^* (B_2 )$, there is no graded isomorphism
    from $(C^*(K_2), \alpha_{\delta'})$ to $( C^* ( B_2) , \alpha_\delta)$ for any
    $\delta : B^1_2 \to \{0,1\}$.

\item More generally, let $\Lambda$ be a row-finite $k$-graph with no sources and fix
    $p \in \mathbb{N}^k$. Recall that the \emph{dual graph} $p \Lambda :=
    \{\lambda \in \Lambda : d(\lambda) \ge p\}$ is a $k$-graph as follows:
    $d_p(\lambda) = d(\lambda) - p$, and if we use the factorisation property in
    $\Lambda$ to write each $\lambda \in \Lambda$ as $\lambda =
    \overline{\lambda}t(\lambda) = h(\lambda)\underline{\lambda}$ with $d(t(\lambda))
    = d(h(\lambda)) = p$, then the range and source maps on $p\Lambda$ are $h$ and
    $t$ respectively, and composition in $p \Lambda$ is given by $\lambda \circ_p \mu
    = \overline{\lambda}\mu = \lambda\underline{\mu}$ whenever $t(\lambda) = h(\mu)$
    (cf.\ \cite[Proposition~3.2]{APS}). By \cite[Theorem~3.5]{APS} there is an
    isomorphism $\theta : C^*(p\Lambda) \to C^*(\Lambda)$ such that
    $s^{p\Lambda}_\lambda \mapsto s^\Lambda_\lambda (s^\Lambda_{t(\lambda)})^*$. So
    any functor $\delta_p : p \Lambda \to \mathbb{Z}_2$ induces a grading $\alpha_p$
    of $C^*(p\Lambda)$ and hence a grading $\alpha$ of $C^*(\Lambda)$. As seen in the
    preceding example, this grading typically does not arise from a functor from
    $\Lambda$ to $\ZZ_2$, but for $k = 1$, we can still apply
    Corollary~\ref{cor:graph Kgr} (to $p\Lambda$) to compute $\Kgr_*(C^*(\Lambda),
    \alpha)$.

\item\label{it:first Minfty computation} Let $F$ be the $1$-graph with vertices
    $\{v_n : n \in \ZZ\}$ and edges $\{e_n, f_n : n \in \NN\}$ where $r(e_n) = r(f_n)
    = v_n$ and $s(e_n) = s(f_n) = v_{n+1}$. Then $C^*(F)$ is Morita equivalent to the
    UHF-algebra $M_{2^\infty}$, and  so $K_*( C^* (F) ) = ( \ZZ[\frac12] , 0)$.
    Define $\delta : F^1 \to \ZZ_2$ by $\delta(e_n) = 0$ and $\delta(f_n) = 1$ for
    all $n$. Then the matrix $A^\delta_F$ is the zero matrix. Hence $1 - A^\delta_F$
    is the identity map from $\ZZ F^0$ to $\ZZ F^0$, and we obtain $\Kgr_*(C^*(F),
    \alpha_\delta) = (0,0)$ by Corollary~\ref{cor:graph Kgr}. (We can also recover
    this result by taking a direct-limit decomposition as in Example~\ref{eq:other
    Minfty computation} below.)
\end{enumerate}
\end{examples}

\begin{rmk}
Suppose that $\Lambda$ is a bipartite $P$-graph. That is, $\Lambda^0 = L \sqcup
R$ and for every edge $\lambda \in \Lambda$ either $s(\lambda) \in L$ and
$r(\lambda)\in R$, or vice versa. Then the gradings $\alpha_{\delta_\Lambda}$ of
$C^*(\Lambda)$ and $C^*(\Lambda, c_\Lambda)$ induced by the functors
$\delta_\Lambda$ of~\eqref{eq:deltaLambda} are inner because the grading automorphism
is implemented by the self-adjoint multiplier unitary $U = P_L - P_R$. Hence
\cite[14.5.2]{B} gives $\Kgr_*(C^*(\Lambda), \alpha_\Lambda) \cong
K_*(C^*(\Lambda))$.

To see why this observation is useful, observe that the skew-product of a $k$-graph
$\Lambda$ by the degree functor $\Lambda \times_d \ZZ^k$ is bipartite with $L = \Lambda^0
\times \{n \in \ZZ^k : \sum_i n_i\text{ is even}\}$ and $R = \Lambda^0 \times \{n \in
\ZZ^k : \sum_i n_i\text{ is odd}\}$. If $\Lambda$ is the $1$-graph $B_2$
from~(\ref{it:KOn}), then $B_2 \times_d \ZZ \cong F$ where $F$ as in~(\ref{it:first
Minfty computation}) above. Hence,  as graded algebras $C^*(B_2 \times_d \ZZ) \cong
C^*(F)$.

Also, let $\Lambda$ be a $k$-graph and let $\delta = \delta_\Lambda : \Lambda \to \ZZ_2$
be as in~\eqref{eq:deltaLambda}. Then the skew product graph $\Lambda \times_\delta
\ZZ_2$ is bipartite (with $L = \Lambda^0 \times \{0\}$ and $R = \Lambda^0 \times \{1\}$),
and so the grading on $C^*(\Lambda \times_\delta \ZZ_2)$ induced by $\delta_{\Lambda
\times_\delta \ZZ_2}$ is inner.
\end{rmk} \label{rmk:bipartite}

\begin{example}\label{eq:other Minfty computation}
Consider again the graph and functor of Examples~\ref{eg:exampleorama}(\ref{it:first
Minfty computation}). We have $C^*(F) = \overline{\bigcup C^*(F_n)}$ where $F_n$ is the
subgraph of $F$ with
\[
F^0_n = \{v_1, \dots, v_n\}\quad\text{ and }\quad F^1_n = \{e_1, f_1, \dots,
e_{n-1}, f_{n-1}\}.
\]
Fix $n \in \NN$. We have $C^*(F_n) \cong M_{F v_n}$ via $s_\mu s^*_\nu \mapsto
\theta_{\mu,\nu}$. Extend $\delta$ to $F$ by setting $\delta(\mu) = \sum^{|\mu|}_{i=1}
\delta(\mu_i)$, and define
\[
U = \sum_{\mu \in F v_n} (-1)^{\delta(\mu)} s_\mu s^*_\mu.
\]
Then $U$ is a self-adjoint unitary in $C^*(F_n)$ that implements the grading by
conjugation. So the grading on each $C^*(F_n)$ is inner, and therefore $\Kgr_*(C^*(F_n),
\alpha_\delta) = K_*(C^*(F_n)) = K_*(M_{F v_n}) = (\ZZ, 0)$, with generator $[s_{v_n}]$.

The inclusion map $\iota_n : C^*(F_n) \hookrightarrow C^*(F_{n+1})$ is given by $s_{\mu}
s^*_{\nu} \mapsto s_{\mu e_n} s^*_{\nu e_n} + s_{\mu f_n} s^*_{\nu f_n}$. In particular
$\iota_n(s_{v_n}) = s_{e_n} s^*_{e_n} + s_{f_n} s^*_{f_n}$. The partial isometry $V :=
s_{e_n} s^*_{f_n}$ is odd and satisfies $V^*V = s_{e_n} s^*_{e_n}$ and $V^*V = s_{f_n}
s^*_{f_n}$. So $\iota_n(s_{v_n}) = V^*V + VV^*$. By Example~\ref{eg:odd pi}, we have
$[V^*V] = -[VV^*]$ in $\Kgr_0 (C^*(F_{n+1}), \alpha_\delta)$, and it follows that
$\iota_* : \Kgr_0 (C^*(F_n), \alpha_\delta) \to \Kgr_0 (C^*(F_{n+1}), \alpha_\delta)$
sends $[s_{v_n}]$ to zero and hence is the zero map. Hence
\[
\Kgr_*(C^*(F), \alpha_\delta)
    \cong \big( \varinjlim(\Kgr_0 (C^*(F_n), \alpha_\delta), \iota_*) , 0 \big)
    \cong ( \varinjlim(\ZZ, 0) , 0 )
    = (0,0) \text{ as before.}
\]
\end{example}

We turn next to some applications of Corollary~\ref{cor:gradedPV} to the crossed products
of the $C^*$-algebra of a $1$-graph $E$. To do this, we first need to describe the map in
graded $K$-theory induced by an automorphism determined by a function from $E^1$ to
$\ZZ_2$.

\begin{lem} \label{lem:justintime}
Let $E$ be a row-finite $1$-graph with no sources and $\delta : E^1 \to \ZZ_2$ a
function. Let $\delta : E \to \ZZ_2$ be the induced functor.  The map $\alpha_*$ on
$\Kgr_*(C^*(E), \alpha_\delta)$ induced by the automorphism $\alpha_\delta$ is the
identity map.
\end{lem}
\begin{proof}
Let $X(E)$ denote be the graph module described in Remark~\ref{rmk:split}. For $v \in
E^0, e \in E^1$, the grading operator $\alpha_\delta$ on $X(W)$ satisfies
\[
\alpha_\delta (1_v \cdot 1_e) = \delta_{v, r(e)} \alpha_\delta (1_e) =
 \delta_{v, r(e)}(-1)^{\delta (e)} 1_e = 1_v \cdot \alpha_\delta (1_e)
\]
and similarly $\alpha_\delta (1_e \cdot 1_v) = \alpha_\delta
(1_e)\cdot 1_v$. So, by linearity and continuity, $\alpha_\delta : X(E) \to X(E)$ is a
bimodule map. Moreover for $e,f \in E^1$ we have
\[
\langle \alpha_\delta (1_e), \alpha_\delta (1_f)\rangle_{C_0(E^0)}
    = (-1)^{\delta (e)+\delta(f)}  \langle 1_e, 1_f\rangle_{C_0(E^0)}
    = \begin{cases}
        1_{s(e)} &\text{ if $e = f$}\\
        0 &\text{ otherwise,}
    \end{cases}
\]
which is precisely $\langle 1_e, 1_f\rangle_{C_0(E^0)}$. So $\alpha_\delta$ is a graded
automorphism of $X(E)$.

Thus the final statement of Theorem~\ref{thm:6-term} implies that $\alpha_\delta$ induces
an automorphism of the exact sequence
\[
0 \to \Kgr_1(C^*(E) , \alpha_\delta ) \hookrightarrow
    \Kgr_0(C_0(E^0)) \stackrel{{\scriptstyle 1 - [X(E)]}}{\longrightarrow}
    \Kgr_0(C_0(E^0)) \twoheadrightarrow
    \Kgr_0(C^*(E) , \alpha_\delta ) \to 0.
\]
Since this automorphism is the identity map on the two middle terms in the sequence, we
deduce that it is the identity map on $\Kgr_*(C^*(E),\alpha_{\delta})$ as claimed.
\end{proof}

\begin{examples} \label{ex:above2}
\begin{enumerate}[(i)]
\item \label{eg:more cp by alpha} Recall Example~\ref{ex:above}(\ref{eg:cp by
    alpha}). Let $E$ be a row-finite $1$-graph with no sources. Give $C^* (E)$ the
    grading $\alpha$ induced by the functor  $\delta ( \lambda ) = | \lambda |
    \pmod{2}$.  Consider the crossed product $C^*(E) \rtimes_\alpha \ZZ$ under the
    grading $\tilde{\alpha}$ satisfying $\tilde{\alpha}(i_A(a)i_\ZZ(n)) = (-1)^n
    i_A(\alpha(a))i_\ZZ(n)$. By applying Corollary~\ref{cor:gradedPV} with $k=1$ (so
    that $\tilde{\alpha} = \beta^1$), and Lemma~\ref{lem:justintime} we obtain the
    following exact sequence.
\begin{equation} \label{eq:cp by alpha}
\parbox{0.8\textwidth}{\mbox{}\hfill\begin{tikzpicture}[yscale=0.8, >=stealth]
    \node (00) at (0,0) {$\Kgr_1(C^*(E) \rtimes_{\alpha_\delta} \ZZ, \tilde{\alpha})$};
    \node (40) at (4,0) {$\Kgr_1(C^*(E), \alpha_\delta)$};
    \node (80) at (8,0) {$\Kgr_1(C^*(E), \alpha_\delta)$};
    \node (82) at (8,2) {$\Kgr_0(C^*(E) \rtimes_{\alpha_\delta} \ZZ, \tilde{\alpha})$};
    \node (42) at (4,2) {$\Kgr_0(C^*(E), \alpha_\delta)$};
    \node (02) at (0,2) {$\Kgr_0(C^*(E), \alpha_\delta)$};
    \draw[->] (02)-- node[above] {${\scriptstyle \times 2}$} (42);
    \draw[->] (42)-- node[above] {${\scriptstyle i_*}$} (82);
    \draw[->] (82)--(80);
    \draw[->] (80)-- node[above] {${\scriptstyle \times 2}$} (40);
    \draw[->] (40)-- node[above] {${\scriptstyle i_*}$} (00);
    \draw[->] (00)--(02);
\end{tikzpicture}\hfill\mbox{}}
\end{equation}
By Example~\ref{ex:above}(\ref{eg:cp by alpha}), we have a graded isomorphism
\[
\big( C^*(E) \hatimes C^*(T_1) , \alpha_{\delta_E} \hatimes\alpha_{\delta{T_1}} \big)
\cong (C^*(E) \rtimes_\alpha \ZZ, \tilde{\alpha}).
\]
We use this to compute $\Kgr_*(C^*(E) \hatimes C^*(T_1))$:

Since $\Kgr_1(C^*(E),\alpha_\delta)$ has no torsion, multiplication by $2$ is
injective on that group, so exactness implies that the right-hand boundary map is
zero. Therefore $\Kgr_0(C^*(E) \hatimes C^*(T_1))$ is isomorphic to the cokernel of
the times-two map on $\Kgr_0(C^*(E),\alpha_\delta)$; that is
\[
\Kgr_0(C^*(E) \hatimes C^*(T_1)) \cong \Kgr_0(C^*(E),\alpha_\delta) / 2\Kgr_0(C^*(E),\alpha_\delta).
\]
Exactness of the bottom row gives
\[
i_*(\Kgr_1(C^*(E),\alpha_\delta)) \cong \Kgr_1(C^*(E),\alpha_\delta) /
2\Kgr_1(C^*(E),\alpha_\delta) ,
\]
so $K_1(C^*(E) \hatimes C^*(T_1))$ is an extension of the $2$-torsion subgroup
\[
\{a \in \Kgr_0(C^*(E),\alpha_\delta) : 2a = 0\}
\]
by $\Kgr_1(C^*(E),\alpha_\delta) / 2\Kgr_1(C^*(E),\alpha_\delta)$. In particular, if
$\Kgr_0(C^*(E),\alpha_\delta)$ has no 2-torsion, then
\[
\Kgr_1(C^*(E) \hatimes C^*(T_1)) \cong \Kgr_1(C^*(E),\alpha_\delta) / 2\Kgr_1(C^*(E),\alpha_\delta).
\]
\item Recall from Examples~\ref{ex:above}(\ref{it:eg Tk}) that $C^* (T_1 ) \hatimes
    C^* ( T_1 ) \cong C^*(T_2, c)$ where $c$ is the $2$-cocycle $c(m,n) = (-1)^{m_2
    n_1} \text{ for } (m,n) \in \NN^2$. We can compute $\Kgr_*(C^*(T_2, c) ,
    \delta_{T_2} )$, by taking $E = T_1$ in (i) above. Since $T_1=B_1$ we have
    $\Kgr_0(C^*(T_1),\delta_{T_1}) = \ZZ_2$ by Examples~\ref{eg:exampleorama}
    (\ref{it:KOn}). Then  by \eqref{eq:cp by alpha}, the times-two map on
    $\Kgr_0(C^*(T_1),\delta_{T_1})$ is the zero map, and so the exact sequence above
    for $\Kgr_*((C^*(T_2),c),\delta_{T_2})$ collapses to give $\Kgr_*(C^*(T_2,
    c),\delta_{T_2}) \cong (\ZZ_2, \ZZ_2)$. Observe that $C^*(T_2, c)$ is the
    rational rotation algebra $A_{\frac{1}{2}}$, so its ungraded $K$-theory is
    $(\ZZ^2, \ZZ^2)$ (see \cite{Elliott}).
\end{enumerate}
\end{examples}

\begin{rmk}
More generally, by Theorem~\ref{thm:swapping}, if $E$ is any row-finite $1$-graph with no
sources endowed with the grading induced by the functor $\delta(e) = 1$ for all $e \in
E^1$, then $C^*(E) \hatimes C^*(T_1) \cong C^*(E \times T_1, c_{E \times T_1})$ with the
grading induced by $\delta_{E \times T_1}$. Thus Example~\ref{ex:above2}(\ref{eg:more cp
by alpha}) computes the graded $K$-theory of this twisted $2$-graph $C^*$-algebra.
\end{rmk}

We finish with an example describing a $2$-graph $C^*$-algebra $C^*(\Lambda)$ that is
Morita equivalent to an irrational-rotation algebra, and a grading $\alpha$ on
$C^*(\Lambda)$ such that $\Kgr_*(C^*(\Lambda),\alpha) = (\ZZ_2 \times \ZZ_2, 0)$.

\begin{example}\label{rg:irrational rotations}
Consider the following $2$-coloured graph (see \cite[Example~6.5]{PRRS})
\[
\begin{tikzpicture}[>=stealth, xscale=1.5,yscale=1.2]
    \node[inner sep=1pt] (00) at (0,0) {\small$v_1$};
    \node[inner sep=1pt] (02) at (0,2) {\small$w_1$};
    \node[inner sep=1pt] (20) at (2,0) {\small$v_2$};
    \node[inner sep=1pt] (22) at (2,2) {\small$w_2$};
    \node[inner sep=1pt] (40) at (4,0) {\small$v_3$};
    \node[inner sep=1pt] (42) at (4,2) {\small$w_3$};
    \node[inner sep=1pt] (60) at (6,0) {\small$v_4$};
    \node[inner sep=1pt] (62) at (6,2) {\small$w_4$};
    \node at (7,0) {\dots};
    \node at (7,2) {\dots};
    \draw[blue, ->] (20) to node[below] {\small$1$} (00);
    \draw[blue, ->] (20) to node[anchor=north east, pos=0.79] {\small$2$} (02);
    \draw[blue, ->] (22) to node[anchor=south east, pos=0.79] {\small$2$} (00);
    \draw[blue, ->] (22) to node[above] {\small$3$} (02);
    \draw[blue, ->] (40) to node[below] {\small$2$} (20);
    \draw[blue, ->] (40) to node[anchor=north east, pos=0.79] {\small$3$} (22);
    \draw[blue, ->] (42) to node[anchor=south east, pos=0.79] {\small$3$} (20);
    \draw[blue, ->] (42) to node[above] {\small$5$} (22);
    \draw[blue, ->] (60) to node[below] {\small$3$} (40);
    \draw[blue, ->] (60) to node[anchor=north east, pos=0.79] {\small$5$} (42);
    \draw[blue, ->] (62) to node[anchor=south east, pos=0.79] {\small$5$} (40);
    \draw[blue, ->] (62) to node[above] {\small$8$} (42);
    \draw[red, dashed,->] (02) .. controls +(0.75,0.75) and +(-0.75,0.75) .. (02);
    \draw[red, dashed,->] (22) .. controls +(0.75,0.75) and +(-0.75,0.75) .. (22);
    \draw[red, dashed,->] (42) .. controls +(0.75,0.75) and +(-0.75,0.75) .. (42);
    \draw[red, dashed,->] (62) .. controls +(0.75,0.75) and +(-0.75,0.75) .. (62);
    \draw[red, dashed,->] (00) .. controls +(0.75,-0.75) and +(-0.75,-0.75) .. (00);
    \draw[red, dashed,->] (20) .. controls +(0.75,-0.75) and +(-0.75,-0.75) .. (20);
    \draw[red, dashed,->] (40) .. controls +(0.75,-0.75) and +(-0.75,-0.75) .. (40);
    \draw[red, dashed,->] (60) .. controls +(0.75,-0.75) and +(-0.75,-0.75) .. (60);
\end{tikzpicture}
\]
where the label on a blue (solid) edge indicates
the number of parallel blue edges. The pattern is that the numbers of edges are Fibonacci
numbers. The red edges are drawn dashed. Let $E = E_{\text{blue}}$ be the subgraph
consisting of blue edges. For $v,w \in E^0$ such that $w E_{\text{blue}}^1 v \neq
\emptyset$, write $v E_{\text{blue}}^1 w = S^0(v,w) \sqcup S^1(v,w)$, where $|S^0(v,w) -
S^1(v,w)|$ is $0$ if $v E_{\text{blue}}^1 w$ is even and $1$ if $v E_{\text{blue}}^1 w$
is odd.

Choose a permutation $\rho$ of $E_{\text{blue}}^1$ that preserves ranges and sources, and
cyclicly permutes the elements of each $S^j(v,w)$, for $j=0,1$. For each $v \in E^0$, let
$f_v$ be the dashed loop based at $v$. Let $\Lambda$ be the $2$-graph with the above
skeleton, and with factorisation rules given by $f_{r(e)} e = \rho(e) f_{s(e)}$ for all
$e \in E^1_{\text{blue}}$.

Since the numbers of parallel edges grow exponentially fast, $\Lambda$ has
large-permutation factorisations in the sense of \cite[Definition~5.6]{PRRS}. It is also
cofinal, and so $C^*(\Lambda)$ is simple with real-rank zero. Elliott's classification
theorem \cite{E} combined with \cite[Theorem~4.3]{PRRS} implies that $C^*(\Lambda)$ is
Morita equivalent to the irrational-rotation algebra $A_\theta$ where $\theta = \frac{1 +
\sqrt{5}}{2}$ (see \cite[Example~6.5]{PRRS}). Define $\delta : E_{\text{blue}}^{1} \to
\ZZ_2$ by $\delta(e) = k$ whenever $e \in S^k(v,w)$ for some $v,w$. This induces a
functor $\delta : E^*_{\text{blue}} \to \ZZ_2$. The matrix $A^\delta_{E_{\text{blue}}}$
defined in Corollary~\ref{cor:graph Kgr} has entries in $\{0, 1\}$, and corresponds to
the $1$-graph $F$ with skeleton
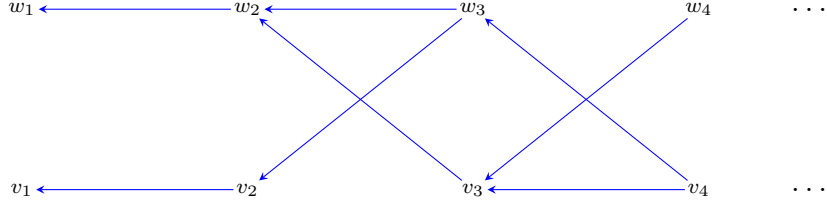
\[
\begin{tikzpicture}[>=stealth, xscale=1.5,yscale=1.2]
    \node[inner sep=1pt] (00) at (0,0) {$\scriptstyle v_1$};
    \node[inner sep=1pt] (02) at (0,2) {$\scriptstyle w_1$};
    \node[inner sep=1pt] (20) at (2,0) {$\scriptstyle v_2$};
    \node[inner sep=1pt] (22) at (2,2) {$\scriptstyle w_2$};
    \node[inner sep=1pt] (40) at (4,0) {$\scriptstyle v_3$};
    \node[inner sep=1pt] (42) at (4,2) {$\scriptstyle w_3$};
    \node[inner sep=1pt] (60) at (6,0) {$\scriptstyle v_4$};
    \node[inner sep=1pt] (62) at (6,2) {$\scriptstyle w_4$};
    \node at (7,0) {\dots};
    \node at (7,2) {\dots};
    \draw[blue, ->] (20) to (00);
    \draw[blue, ->] (22) to (02);
    \draw[blue, ->] (40) to (22);
    \draw[blue, ->] (42) to (20);
    \draw[blue, ->] (42) to (22);
    \draw[blue, ->] (60) to (40);
    \draw[blue, ->] (60) to (42);
    \draw[blue, ->] (62) to (40);
\end{tikzpicture}
\]
where the pattern of connecting edges repeats every three levels (note: there are no
parallel edges). By telescoping these three levels, and arguing as in
Example~\ref{eq:other Minfty computation}, we see that $\Kgr_0 (C^*(E),\alpha_\delta)
\cong \varinjlim\big(\ZZ^2, \big(\begin{smallmatrix} 1 & 2
\\ 0 & 1\end{smallmatrix}\big)\big)$. This matrix has determinant 1, and so
$\Kgr_*(C^*(E),\alpha_\delta) \cong (\ZZ^2, 0)$.

As in \cite{FPS}, the permutation $\rho$ of $E^1$ defining the factorisation rules
induces an automorphism $\tilde{\rho}$ of $C^*(E)$, and $C^*(\Lambda) \cong C^*(E)
\rtimes_{\tilde{\rho}} \ZZ$ by \cite[Theorem 3.4]{FPS}. There are two natural extensions
of $\delta$ to a $\ZZ_2$-valued functor on $\Lambda$: namely $\delta^0$, determined by
$\delta^0(f) = 0$ for all $f \in \Lambda^{e_2}$, and $\delta^1$, determined by
$\delta^1(f) = 1$ for all $f \in \Lambda^{e_2}$. The grading automorphisms
$\alpha_{\delta^k}$ correspond under the identification $C^*(\Lambda) \cong C^*(E)
\rtimes_{\tilde{\rho}} \ZZ$ with the automorphisms $\beta^k$ of
Corollary~\ref{cor:gradedPV}, for $k=0,1$. So we can compute the graded $K$-theory of
$C^*(\Lambda, \alpha_{\delta^k})$ by applying that result. The automorphism
$\tilde{\rho}$ of $C^*(E)$ permutes equivalent projections in approximating
finite-dimensional subalgebras of $C^*(E)$. So the automorphism $\tilde\rho_*$ of
$\Kgr_*(C^*(E))$ induced by $\alpha_\delta \circ \tilde{\rho}$ is the identity. By
Lemma~\ref{lem:justintime} we therefore have $(-(\alpha_{\delta})_*)^k\tilde{\rho}_* =
(-1)^k\id$ for $k = 0,1$.
Thus $\id - (-\alpha_*)^0\tilde{\rho}_* = 0$ and Corollary~\ref{cor:gradedPV} gives
$\Kgr_*(C^*(\Lambda), \alpha_{\delta^0}) \cong \big(\ZZ^2, \ZZ^2)$ (which is isomorphic
to $K_*(C^*(\Lambda))$ as a pair of abelian groups). And $\id -
(-(\alpha_\delta)_*)^1\tilde{\rho}_* = 2\cdot\id$, so Corollary~\ref{cor:gradedPV} gives
$\Kgr_*(C^*(\Lambda), \alpha_{\delta^1}) = (\ZZ_2^2, 0)$.
\end{example}

\section{Conjecture}\label{sec:conj}
Our results and examples, particularly Example~\ref{eg:odd pi}, lead us to ask whether
the graded $K_0$-group of a unital graded $C^*$-algebra $A$ consists of equivalence
classes of homogeneous projections  over $A$ subject to the relation
\[
[v^*v] = (-1)^{\partial v} [vv^*]\quad\text{ for every homogeneous partial isometry $v$.}
\]

To make this concrete, let $(A, \alpha)$ be a unital graded $C^*$-algebra. Let $P_0(A)$
denote the collection of homogeneous projections in $\Kk(\widehat{\Hh}_A)$. Let $p, q \in
P_0(A)$; we write $p \sim q$ if there is an \emph{even} partial isometry $v$ such that $p
= v^*v$ and $q = vv^*$.  If $p \perp q$, then $p + q$ is a projection and we write $[p] +
[q] = [p + q]$. Define $V_0(A) := P_0(A)/{\sim}$, which is an abelian monoid under the
binary operation induced by orthogonal addition. Given a homotopy $t \mapsto p_t$ in
$P_0(A)$, we have $[p_0] = [p_1]$ (see \cite[2.2.7]{RLL} or \cite[\S4]{B}).

Note that $V_0(A)$ may be identified with the set of isomorphism classes of graded projective
modules over $A$. Indeed, given $p  \in P_0(A)$ we may form the graded projective module
$p\widehat{\Hh}_A$ (with grading inherited from $\widehat{\Hh}_A$).
Given $p, q \in P_0(A)$. We have $p \sim q$  if and only
if $p\widehat{\Hh}_A \cong q\widehat{\Hh}_A$. Moreover, if $p \perp q$, then
$(p+q)\widehat{\Hh}_A \cong  p\widehat{\Hh}_A \oplus q\widehat{\Hh}_A$. By the
Stabilization Theorem  (see \cite[Theorem 14.6.1]{B}) every graded projective module is
isomorphic to a summand of $\widehat{\Hh}_A$ and therefore is isomorphic to
$p\widehat{\Hh}_A$ for some $p \in P_0(A)$. Thus we may and do regard $V_0(A)$ as the
semigroup of isomorphism classes of graded projective modules over $A$ with the binary
operation given by direct sum, that is, $[X] + [Y] = [X \oplus Y]$ where $X$ and $Y$ are
graded projective modules.

A graded projective module $Z$  is said to be \emph{degenerate} if there is a graded
projective module $X$ such that $Z \cong X \oplus X^\text{op}$. For $p \in P_0(A)$, the
graded projective  module $p\widehat{\Hh}_A$ is degenerate if and only if there is an odd
partial isometry $v$ such that $p =  v^*v + vv^*$. Let $D_0(A)$ denote the collection of
isomorphism classes of degenerate graded projective modules in $P_0(A)$. Observe that
$D_0(A)$ forms a submonoid of $V_0(A)$.

Let $X, Y$ be graded projective modules; we write $X \approx Y$ if there are degenerate
graded projective modules $Z, W$ such that $X \oplus Z \cong Y \oplus W$. Then $\approx$
forms an equivalence relation on graded projective modules (coarser than isomorphism) and
we let $L(A, \alpha)$  denote the collection of equivalence classes. We write $[X]_L$ for
the equivalence class of the graded projective module $X$. It is routine to show that the
direct sum of graded projective modules yields a well-defined binary operation on $L(A,
\alpha)$ which makes it an abelian semigroup. Recall that $\ell : \CC \to \Ll(X)$ denoted
the left action of $\CC$ by scalar multiplication on $X$.

\begin{prp}\label{prp:L}
The semigroup $L(A, \alpha)$ forms an abelian group with inverse given by $-[X]_L = [X^\text{op}]_L$
for $X$ a graded projective module and zero element given by the class of the trivial module
(or any degenerate module).
There is a group homomorphism $\varpi: L(A, \alpha) \to \Kgr_0(A, \alpha)$ such that
\[
\varpi([X]_L) = [\ell, X, 0, \alpha_{X}]
\in KK(\CC, A) = \Kgr_0(A, \alpha)
\]
for every graded projective module $X$.
\end{prp}
\begin{proof}
Let $X$ be  a graded projective module.  Then $[X]_L + [X^\text{op}]_L = [X \oplus
X^\text{op}]_L = [0]_L$. The map $\varpi$ is well defined since Example~\ref{eg:odd pi}
shows that the Kasparov element associated to a degenerate graded projective module maps
to zero and  $\varpi$ is clearly additive,
\end{proof}

\begin{conj}
The homomorphism $\varpi$ of Proposition~\ref{prp:L} is an isomorphism.
\end{conj}

It should not be difficult to show that our conjecture holds when $A$ is trivially
graded. In this case $L(A, \text{id}) \cong K_0(A)$ since graded projective modules over
$A$ are all of the form $X \cong Y \oplus Z^\text{op}$ where $Y$ and $Z$ are trivially
graded projective modules over $A$ and $\alpha_X \cong (\text{id}, - \text{id})$.
Moreover, $[Y_1 \oplus Z_1^\text{op}]_L = [Y_2 \oplus Z_2^\text{op}]_L$ if and only if
$([Y_1], [Z_1]) \sim ([Y_2], [Z_2])$ in the Grothendieck group  $K_0(A)$. This is is
closely related to the argument that $KK_0(\CC, A) \cong K_0(A)$ for ungraded $A$---see
\cite[Proposition~17.5.5]{B}.

\end{document}